\documentclass[leqno,a4wide]{amsart}
\usepackage{amssymb,amsmath}
\usepackage{times,bbm,epsfig,float,fancyhdr}
\usepackage{mathtools}
\usepackage{color}
\usepackage[pagebackref=true,colorlinks=true]{hyperref}

\usepackage{tikz}
\usepackage{wrapfig}
\usepackage{subfigure}
\graphicspath{{figures/}}
\usepackage{algorithm}
\usepackage{algorithmic}

\newcommand*{\mailto}[1]{\href{mailto:#1}{\nolinkurl{#1}}}

\definecolor{darkgreen}{rgb}{0.5,0.25,0}
\definecolor{darkblue}{rgb}{0,0,1}
\definecolor{answerblue}{rgb}{0,0,0.75}

\hypersetup{colorlinks,breaklinks,
            linkcolor=darkblue,urlcolor=darkblue,
            anchorcolor=darkblue,citecolor=darkblue}

\numberwithin{equation}{section}

\newcommand{\norm}[1]{\ensuremath{\left\|#1\right\|}}

\newcommand{\abs}[1]{\left|#1\right|}
\newcommand{\absb}[1]{\bigl|#1\bigr|}

\def\Cof#1{\,\mathrm{Cof}\,#1}

\newcommand{\bn}{\boldsymbol{n}}
\newcommand{\tr}{\mathrm{tr}}
\newcommand{\bd}{\boldsymbol{d}}
\newcommand{\bsigma}{\boldsymbol{\sigma}}
\newcommand{\bK}{\ensuremath{\mathbf{K}}}

\def\ptl{{\partial}}

\newcommand{\bP}{{\mathbf P}}
\newcommand{\cW}{\mathcal{W}}

\newcommand{\bx}{\mathbf{x}}
\newcommand{\bu}{\mathbf{u}}
\newcommand{\bU}{\ensuremath{\mathbf{U}}}
\newcommand{\Div}{\nabla\cdot\,}

\newcommand{\eps}{\varepsilon}

\newcommand{\Grad}{\ensuremath{\mathrm{\nabla}}}

\newcommand{\N}{\ensuremath{\mathbb{N}}}

\newcommand{\Set}[1]{\ensuremath{\left\{#1\right\}}}
\newcommand{\Setb}[1]{\ensuremath{\bigl\{#1\bigr\}}}
\newcommand{\R}{\ensuremath{\Bbb{R}}}
\newcommand{\U}{\ensuremath{\Bbb{U}}}

\newcommand{\bX}{\ensuremath{\Bbb{X}}}

\newcommand{\bF}{\mathbf{F}}

\newcommand{\Iapp}{I_{\mathrm{app}}}
\newcommand{\Iion}{I_{\mathrm{ion}}}

\newcommand{\bv}{\mathbf{v}}
\newcommand{\bh}{\boldsymbol{h}}

\newcommand{\bD}{\mathbf{D}}
\newcommand{\ff}{\boldsymbol{f}}
\newcommand{\fft}{\boldsymbol{\tilde f}}

\newcommand{\fftn}{\boldsymbol{\tilde f^n}}
\newcommand{\ffte}{\boldsymbol{\tilde f_\eps}}
\newcommand{\bg}{\mathbf{g}}
\newcommand{\bC}{\mathbf{C}}
\newcommand{\bI}{\mathbf{I}}

\newcommand{\D}{\mathcal{D}}
\newcommand{\Dp}{\D^\prime}

\newcommand{\cO}{\mathcal{O}}
\newcommand{\Om}{\ensuremath{\Omega}}

\newcommand{\pt}{\ensuremath{\partial_t}}

\newcommand{\Iap}{I_{\mathrm{app}}}
\newcommand{\Ion}{I_{\mathrm{ion}}}

\newcommand{\dx}{\ensuremath{\, d\bx}}

\newcommand{\dt}{\ensuremath{\, dt}}

\newcommand{\dW}{\ensuremath{\, dW}}
\newcommand{\ds}{\ensuremath{\, ds}}

\newcommand{\En}{\ensuremath{\mathbf{1}}}

\newcommand{\cF}{\mathcal{F}}
\newcommand{\cG}{\mathcal{G}}
\newcommand{\cFt}{\ensuremath{\cF_t}}
\newcommand{\tcFt}{\ensuremath{\tilde\cF_t}}
\newcommand{\cS}{\ensuremath{\mathcal{S}}}
\newcommand{\cD}{\ensuremath{\mathcal{D}}}
\newcommand{\cDp}{\ensuremath{\mathcal{D}^\prime}}
\newcommand{\cL}{\ensuremath{\mathcal{L}}}
\newcommand{\cP}{\ensuremath{\mathcal{P}}}
\newcommand{\cB}{\ensuremath{\mathcal{B}}}
\newcommand{\E}{\ensuremath{\Bbb{E}}}

\newcommand{\cX}{\ensuremath{\mathcal{X}}}

\newcommand{\cZ}{\ensuremath{\mathcal{Z}}}
\newcommand{\cK}{\ensuremath{\mathcal{K}}}

\newcommand{\vphi}{\ensuremath{\varphi}}

\newcommand{\bM}{\ensuremath{\mathbf{M}}}

\newcommand{\weak}{\rightharpoonup}

\newcommand{\Span}{\operatorname{Span}}

\newcommand{\toR}{\xrightarrow{R\uparrow \infty}}
\newcommand{\ton}{\xrightarrow{n\uparrow \infty}}
\newcommand{\tok}{\xrightarrow{k\uparrow \infty}}
\newcommand{\tom}{\xrightarrow{m\uparrow \infty}}
\newcommand{\toeps}{\xrightarrow{\eps\downarrow 0}}
\newcommand{\weakn}{\xrightharpoonup{n\uparrow \infty}}
\newcommand{\weakeps}{\xrightharpoonup{\eps\downarrow 0}}

\newtheorem{thm}{Theorem}[section]
\newtheorem{lem}[thm]{Lemma}
\newtheorem{cor}[thm]{Corollary}
\newtheorem{prop}[thm]{Proposition}
\newtheorem{rem}[thm]{Remark}
\newtheorem{defi}[thm]{Definition}

\allowdisplaybreaks

\begin{document}

\title[Stochastic electromechanical bidomain model]
{Stochastic electromechanical bidomain model}

\author[M. Bendahmane]{M. Bendahmane}
\address[Mostafa Bendahmane]
{Institut de Math\'ematiques de Bordeaux UMR CNRS 525, 
Universit\'e Victor Segalen Bordeaux 2, F-33076 Bordeaux Cedex, France}
\email[]{\mailto{mostafa.bendahmane@u-bordeaux.fr}}

\author[K. H. Karlsen]{K.H. Karlsen}
\address[Kenneth H. Karlsen]
{Department of mathematics, University of Oslo, 
P.O. Box 1053,  Blindern, NO--0316 Oslo, Norway}
\email[]{\mailto{kennethk@math.uio.no}}

\author[F. Mrou\'e]{F. Mrou\'e}
\address[Fatima Mrou\'e]
{Center for Advanced Mathematical Sciences 
and Department of Mathematics, American University of Beirut, 
Beirut 1107 2020, Lebanon}
\email[]{\mailto{fm47@aub.edu.lb}}



\thanks{The research conducted by M.~Bendahmane is supported 
by the FINCOME project.}

\date{\today}

\begin{abstract}
We analyze a system of nonlinear 
stochastic partial differential equations (SPDEs) 
of mixed elliptic-parabolic type that models the 
propagation of electric signals and their effect 
on the deformation of cardiac tissue. The system governs the 
dynamics of ionic quantities, intra and extra-cellular potentials, and 
linearized elasticity equations. We introduce a framework called 
the active strain decomposition, which factors the material gradient 
of deformation into an active (electrophysiology-dependent) part 
and an elastic (passive) part, to capture the coupling 
between muscle contraction, biochemical reactions, 
and electric activity. Under the assumption of linearized elastic behavior 
and a truncation of the nonlinear diffusivities, we propose 
a stochastic electromechanical bidomain model, 
and establish the existence of weak solutions for this model. 
To prove existence through the convergence 
of approximate solutions, we employ a stochastic compactness 
method in tandem with an auxiliary non-degenerate system 
and the Faedo--Galerkin method. We utilize a stochastic 
adaptation of de Rham's theorem to deduce the 
weak convergence of the pressure approximations.
\end{abstract}

\maketitle

\tableofcontents

\section{Introduction}

\subsection{Background}\label{subsec:background}

Our main interest lies in the mathematical investigation of the interplay 
between propagation of electrical potentials in cardiac tissue (under 
stochastic effects) and its elastic mechanical response. 
These distinct processes in cardiac function are 
intricately linked by several complex phenomena 
occurring at various spatio-temporal scales. 
At the macroscopic level, the bidomain equations describe 
the propagation of electrical potentials in the heart by conserving 
electrical fluxes between the extra- and intra-cellular domains, 
which are separated by a membrane functioning as a capacitor. 
The bidomain equations are influenced by differing conductivities 
within these domains, which also vary based on the orientation 
of the cardiac tissue fibers. A multi-continuum description of the heart 
can be obtained by homogenization arguments, wherein both 
constituents (intra- and extra-cellular potentials) coexist at 
each material point. Several references, such 
as \cite{Colli-Franzone:2014aa, Sundnes:2006aa}, 
describe the bidomain equations, whereas 
\cite{Bendahmane:2019tf, Collin:2018uv, Grandelius:2018aa, 
Neu:1993aa, Pennacchio:2005aa} provide insight 
into homogenization arguments for the bidomain model.

At the macroscopic level, muscle deformation can be 
described using the equations of motion for 
a hyperelastic material in its reference configuration. 
The muscle tissue itself is active, meaning 
that it can contract without external loads, but instead through 
intrinsic mechanisms that occur primarily at the microscale. 
There are different approaches to incorporating these effects, but 
one common method is to decompose stresses into active and passive components, 
resulting in what is known as the \textit{active stress} formulation. 
This approach has been widely applied in various fields, as 
seen in studies such as \cite{goktepe10,nash00,trayanova11}.

Another approach to incorporating the intrinsic mechanisms 
in muscle contraction is to use the \textit{active strain} 
formulation \cite{cherub,nardin}. 
This involves factoring the deformation gradient into active and passive 
components, with the active deformation representing the fiber contraction 
driven by the depolarization of the cardiomyocytes. 
This approach directly incorporates micro-level information on fiber 
contraction and direction into the kinematics, without the need for 
intermediate transcription in terms of stress. In this formulation, the 
strain energy function depends on auxiliary internal state variables that 
represent the level of tissue activation across scales \cite{rossi13}. 
We adopt the active strain approach in our work, but \cite{Ambrosi:2012aa,rossi12} 
provide comparisons between the two (active stress / active strain) approaches 
in terms of numerical implementation, constitutive issues, and stability.

The analysis of deterministic macroscopic cardiac models has been 
primarily focused on the study of solutions to the bidomain equations 
coupled with phenomenological or physiologically-based ionic models. 
A variational approach was first introduced in \cite{Colli-Franzone:2002aa} 
and has since been extended in various directions. 
These extensions include the nonlinear, degenerate  conductivity tensors \cite{Bendahmane:2006au}, more complex ionic models \cite{Bourgault:2009aa,Veneroni:2009aa}, and the existence of 
global classical solutions \cite{Matano:2011tj}. 
Additionally, further references are available in 
\cite{Andreianov:2010uq, Boulakia:2008aa, Colli-Franzone:2014aa, 
Giga:2018aa, Kunisch:2013aa}. 
The aforementioned works provide well-posedness results 
for the bidomain model, utilizing differing solution concepts 
and technical frameworks.

Existence theories for nonlinear elasticity can be found 
in \cite{ciarlet}, while applications of these theories to the particular 
case of hyperelastic materials and cardiac mechanics (and 
their discretizations) are explored in \cite{nash00}. Despite the extensive  
literature on numerical methods and models for cardiac 
electromechanics, including \cite{goktepe10,trayanova11}, rigorous 
studies on the solvability and stability of solutions are still lacking.

The aim of this paper is to explore \textit{stochastic} 
electromechanical coupling by considering the dependence of 
the electrical potentials, governed by the stochastic bidomain 
equations of \cite{Bendahmane:2018aa}, on 
the deformation gradient. This dependence arises after 
transforming from Eulerian to Lagrangian coordinates and by virtue of 
the Piola identity. In turn, the active part of the 
deformation is assumed to incorporate the influence of calcium kinetics 
into the equations of linear structural mechanics. 
This modeling approach enables coupling in the 
opposite direction (feedback), achieving stochastic 
electromechanical coupling.

This work improves the stochastic bidomain model for cardiac 
electrophysiology \cite{Bendahmane:2018aa} by 
incorporating the mechanical activity of the heart. By doing so, the 
resulting model gains complexity in its mathematical analysis and 
becomes more realistic and applicable from a practical standpoint.  
To achieve a realistic model, it is crucial to blend microscale 
effects into the model that capture the 
heart's electrical and mechanical behaviors. Our approach 
incorporates ``current noise" into the bidomain equations to model 
the stochastic behavior of ion channels on cellular voltage dynamics. 
Furthermore, we include ``subunit noise" into the 
equations governing the gating variables, capturing 
the stochastic transitions of ion channels between 
different states \cite{Goldwyn:2011aa}.  
The model considers the mechanical contractions of the heart, which 
has a feedback effect on the electrical activity by changing the 
electrical conductivity due to tissue strain \cite{cherub}. 
This necessitates coupling the stochastic 
bidomain equations with the elasticity equations, 
showcasing the complex interaction between the 
electrical signals of the heart and its mechanical actions.

Realistic models play a crucial role in bridging the 
gap between organ-level phenomena and their biophysical underpinnings. 
They also furnish clinicians with quantitative 
insights for refining diagnostic, prognostic, and therapeutic strategies
\cite{quarteroni2022modeling,biglino2016computational,niederer2019computational}. 
A case in point is electrocardiography, where the traditional analysis 
of ECG (electrocardiogram) and BSPM (body surface potential mapping) data, 
based on \textit{static} cardiac models, ignores the heart's 
dynamic contractions and the influence of random effects. 
This limitation can lead to inaccuracies in interpreting 
electrical signals. Incorporating heart contractions and stochastic factors 
into our model, as \cite{moss2021fully} suggests, allows for a 
more dynamic and accurate mapping of electrical potentials, 
enhancing our understanding of cardiac functions 
beyond the constraints of static models.

We will revisit the stochastic bidomain model later in this introduction, but for 
now, let us focus on the elasticity problem.

\subsection{Nonlinear elasticity problem}\label{sub:elasticity}

We model the heart as a homogeneous continuous material 
occupying a bounded domain $\cO_R\subset\mathbb{R}^d$ 
with Lipschitz continuous boundary $\partial\cO_R$, where $d=3$. 
The deformation of the heart is described by the equations of motion 
written in the reference configuration $\cO_R$, and the 
current configuration is denoted by $\cO$. 
The goal is to determine the deformation field 
$\varphi:\cO_R\to\mathbb{R}^d$, which maps a 
material particle initially located at position $X$ to 
its current position $\bx:=\varphi(X)$. The tensor gradient of deformation 
is given by $\bF:=\bD_X\varphi$, where $\bD_X$ denotes the 
gradient operator with respect to the material coordinates $X$, noting 
that $\text{det}(\bF)>0$.

We assume that the cardiac tissue is a hyperelastic 
incompressible material, and as such, there exists a strain 
stored energy function $\cW=\cW(X,\bF)$, which is differentiable 
with respect to $\bF$. This function allows us to obtain 
constitutive relations between strains and stresses. 
Notably, we assume the incompressibility 
of the material, which is enforced via a scalar Lagrange 
multiplier $p$ associated with the incompressibility constraint, 
interpreted as "hydrostatic pressure". By assuming this constraint, 
we minimize the total elastic energy subject to the condition $\det(\bF) =1$. 
Furthermore, the first Piola stress tensor $\bP$, which 
represents force per unit undeformed surface, is given by:
$$
\bP=\frac{\partial\cW}{\partial(\bF)}-p\Cof(\bF),
$$
where $\Cof(\cdot)$ is the cofactor matrix.

We seek to find the deformation field $\varphi$ and Lagrange multiplier $p$ 
that satisfy the balance equations for deformations and pressure given by
\begin{equation}\label{mechanical-strong}
	\Div \bP(\bF,p)=\mathbf{g} \quad \text{in $\cO_R$},
	\qquad
	\det (\bF) =1 \quad \text{in $\cO_R$},
\end{equation}
where $\mathbf{g}$ is a prescribed body force, and $\bn$ represents 
the unit outward normal vector to $\partial\cO_R$. 
We complete these equations with the Robin boundary condition
\begin{equation}\label{bc}
	\bP\bn=-\alpha\varphi\quad 
	\text{on }\partial\cO_R,
\end{equation}
where $\alpha>0$ is a constant parameter. 
Our choice of boundary condition in \eqref{bc} can be 
adjusted to replicate the global motion 
of the cardiac muscle~\cite{rossi13}.

To obtain a precise form of the equation in \eqref{mechanical-strong}, we 
require a particular constitutive relation defining $\cW$. 
Our study is limited to Neo-Hookean materials, where 
$$
\cW=\frac{1}{2}\mu\tr\left[\bF^{T}\bF-\bI\right],
$$
which gives $\frac{\partial \cW}{\partial \bF}=\mu \bF$ 
and $\bP=\mu\bF-p\Cof(\bF)$, where $\mu$ 
is an elastic modulus.

While simplified, this description of the passive response of the 
muscle exhibits a non-linear strain-stress relationship arising from 
the incompressibility constraint. Another form of strain-stress 
non-linearity results from the anisotropy inherited from the active 
strain incorporation (see, e.g., \cite{nash00,rossi13,rossi12} 
for more information).

\subsection{Stochastic bidomain model}\label{subsec:stoch-bidomain}

The deterministic bidomain equations, first introduced in \cite{Tung:1978aa} 
and further discussed in the books \cite{Colli-Franzone:2014aa, Sundnes:2006aa}, 
are derived from the application of Ohm's electrical conduction law 
and the continuity equation, which ensures the conservation of 
electrical charge, to both the intracellular and extracellular domains.
Recently, a stochastic extension of the bidomain equations 
has been proposed in \cite{Bendahmane:2018aa}. 
In this work, we will refer to these equations as 
the stochastic bidomain model (or equations).

Let $T>0$ be a fixed final time, and let $\cO \subset \R^3$ be 
a bounded open subset that represents the heart. We make the 
assumption that the intracellular and extracellular current densities can 
be expressed in terms of potentials $v_i=v_i(t,\bx)$ and $v_e=v_e(t,\bx)$, 
respectively, at points $(t,\bx) \in \cO_T:=(0,T)\times \cO$. 
We define the transmembrane potential and the gating 
or recovery variable as $v = v(t,\bx):=v_i-v_e$ 
and $w=w(t,\bx)$, respectively. 

In the global coordinate system, 
cardiac electrical conductivity is represented by the orthotropic tensors
$$
\bK_k(\bx)=\sigma_k^l\bd_l\otimes\bd_l+\sigma_k^t
\bd_t\otimes\bd_t+\sigma_k^n\bd_n\otimes\bd_n,
\quad k \in \Set{e,i},
$$
where $\sigma_k^s = \sigma_k^s(\bx)\in \boldsymbol{C}^1(\R^3)$ 
for $k \in \Set{e,i}$ and $s \in \Set{l,t,n}$, denote the intra- and 
extracellular conductivities along the directions of the fibers. 
Here, $\bd_s = \bd_s(\bx)$ for $s \in \Set{l,t,n}$, represents 
the direction of the fibers, which is a local quantity.

The stochastic bidomain equations \cite{Bendahmane:2018aa} 
take the following form:
\begin{equation}\label{eq:veq-noise}
	\begin{split}
		& \chi_m c_m dv
		-\nabla \cdot \left(\bK_i \nabla v_i\right)\dt 
		+\chi_m \Iion(v,w)\dt=\Iap \dt+\beta_v(v) \dW^v,
		\\
		&\chi_m c_m dv
		+\nabla \cdot \left(\bK_e \nabla v_e\right)\dt 
		+\chi_m \Iion(v,w)\dt=-\Iap\dt+\beta_v (v) \dW^v,
		\\
		& dw  = H(v,w)\dt + \beta_w(w)\dW^w, 
		\qquad (t,\bx) \in \cO_T,
	\end{split}
\end{equation}
where $c_\mathrm{m}$ and $\chi$ are model parameters, 
and the externally applied stimulation current is represented 
by the function $\Iap$.

In \eqref{eq:veq-noise}, $W^\kappa$ is a cylindrical Wiener process 
with a noise amplitude $\beta_\kappa$ for $\kappa = v, w$. 
The multiplicative stochastic term $\beta_v(v) \dW^v$ added to the equations 
for the membrane potential $v$ is commonly referred to as \textit{current noise}, 
and it represents the aggregated effect of the random activity of 
ion channels on the voltage dynamics. On the other hand, the 
gating variable $w$ represents the fraction of open channel 
subunits of varying types, and the multiplicative 
noise $\beta_w(w)\dW^w$ added to the $w$-equation 
is sometimes referred to as \textit{subunit noise}.  
For a discussion of different ways of adding noise to 
the Hodgkin-Huxley equations, we refer 
readers to \cite{Goldwyn:2011aa}.

We complete the SPDE system \eqref{eq:veq-noise} with 
homogeneous Neumann boundary conditions for all fields. 
It is worth noting that the choice of the membrane model 
to be used is reflected in the functions $H(v, w)$ and $\Ion(v, w)$. 
For a phenomenological description of the action potential, the 
FitzHugh-Nagumo model \cite{Nagumo,FitzHugh:1961aa}, given by 
assumption ({\bf A.4}) below, will be used this paper.

The available literature on the mathematical theory of 
the stochastic bidomain system \eqref{eq:veq-noise} 
is currently limited. The work \cite{Bendahmane:2018aa} 
confirmed the global well-posedness of weak solutions for multiplicative noise. 
For strong well-posedness results, assuming additive noise, 
we refer to the works \cite{Hieber:2020vk} 
and \cite{Kapustyan:2022aa}.

\subsection{Stochastic electromechanical equations}
\label{sub:stoch-electromech}

The \textit{active strain model} \cite{cherub} is used to couple the 
electrical and mechanical behaviors. This model factorizes the 
deformation (gradient) tensor $\mathbf{F}$ into a 
passive component $\mathbf{F}_p$ at the tissue level and 
an active component $\mathbf{F}_a$ at the cellular level 
such that $\mathbf{F} = \mathbf{F}_p\mathbf{F}_a$. 
The tensor $\mathbf{F}_p$ accounts for the deformation of 
the material to ensure compatibility and possible tension due to external loads. 
The active tensor $\mathbf{F}_a$ represents the distortion that 
dictates deformation at the fiber level and is assumed to 
depend directly on the electrophysiology, as described in \cite{aanq}:
$$
\mathbf{F}_a=\bI+\gamma_l \bd_l\otimes \bd_l
+\gamma_t \bd_t\otimes \bd_t+\gamma_n \bd_n\otimes \bd_n,
$$
where $\gamma_\kappa$, $\kappa \in \Set{l,t,n}$ are 
quantities that depend on the electrophysiology equations.

The active strain approach used for decomposing the deformation 
tensor $\bF$ assumes the existence of a virtual intermediate configuration 
located between the reference and the current frames. 
In this configuration, the strain energy function depends only 
on the macroscale deformation $\bF_p$, as expressed in \cite{aanq}:
$$
\cW=\cW(\bF_p)=\cW(\bF\bF^{-1}_a)
=\dfrac{\mu}{2}\textrm{tr}\left[\bF^T_p\bF_p-\mathbf{I}\right]
=\dfrac{\mu}{2}\textrm{tr}\left[\bF^{-T}_a\bF^T\bF\bF^{-1}_a
-\mathbf{I}\right].
$$
Moreover, the strain energy can be rewritten such 
that the Piola stress tensor becomes
$$
\mathbf{P}=\mu \mathbf{F} \mathbf{C}_a^{-1}
-p\textrm{Cof}(\mathbf{F}),
$$
where $\mathbf{C}_a^{-1}:=\textrm{det}(\mathbf{F}_a)
\mathbf{F}_a^{-1}(\mathbf{F}_a^{-\textrm{T}})$, see \cite{aanq}. 

It should be noted that mechanical activation is primarily influenced 
by the release of intracellular calcium~\cite{nordsletten11}. 
Furthermore, local strain dynamics closely follow calcium release 
rather than transmembrane potential. Physiologically-based 
activation models require a dependence of $\gamma_\kappa$ not 
only on calcium, but also on local stretch, local stretch rate, sliding 
velocity of crossbridges, and other force-length experimental 
relationships \cite{rossi13}. However, for simplicity, we will 
restrict ourselves to a phenomenological description of local activation 
in terms of gating variables, where the scalar fields $\gamma_l$, $\gamma_t$, 
and $\gamma_n$ are functions of a common scalar $\gamma$ as expressed in 
\begin{equation*}
	\gamma_{\kappa}=\gamma_{\kappa}(\gamma) 
	\quad \text{for $\kappa \in \Set{l,t,n}$}.
\end{equation*}
Here, each $\gamma_{\kappa}:\R\mapsto [-\Gamma_{\kappa},0]$ 
is a Lipschitz continuous monotone function and 
$\Gamma_{\kappa}$ is small enough to ensure that $\det(\bF_a)$ 
stays uniformly bounded away from zero for $\gamma\in \R$ and 
$\kappa \in \Set{l,t,n}$. More explicitly, the functions 
$\gamma_{\kappa}$ are assumed to have the form:
\begin{equation*}
	\gamma_{\kappa}(\gamma) = 
	-\Gamma_{\kappa}\frac{2}{\pi}
	\arctan\left(\gamma^+/\gamma_R\right),
	\quad \kappa\in \Set{l,t,n},
\end{equation*}
where $\gamma^+:=\max(0,\gamma)$ and $\gamma_R$ 
is a reference value. The scalar field $\gamma$ is governed by 
the ordinary differential equation
$$
\pt \gamma -G(\gamma,w)  =0 
\quad \text{ in $\cO_T$}, 
\quad 
G(\gamma,w)=\eta_1\left(\beta w-\eta_2\gamma\right),
$$
where $\beta,\eta_1,\eta_2\ge 0$ are physiological parameters. 

In order to couple mechanical and electrical behavior, the stochastic 
bidomain equations are modified by changing variables from the 
current configuration (Eulerian coordinates) to the reference 
configuration (Lagrangian coordinates). 
This modification introduces a conduction term 
that depends on the deformation gradient $\bF$. As a result, the 
stochastic electromechanical activity in the heart is 
governed by the following system of coupled random elliptic 
equations and stochastic degenerate parabolic 
reaction-diffusion equations, along with 
two differential equations:
\begin{equation}\label{coupbid}
	\begin{aligned}
		& -\Div \bigl( a(\bx,\gamma,\bF,p) \bigr)=\mathbf{g}, 
		\quad \det(\bF) =1
		&& \text{in $\cO_R$, $t\in (0,T)$},
		\\
		& \chi_m c_m dv -\nabla \cdot 
		\left(\bM_i(\bx,\bF)\Grad v_i\right)\dt
		+\chi_m \Iion(v,w)\dt
		\\ 
		& \qquad\qquad\qquad\qquad \qquad
		=\Iap \dt+\beta_v(v) \dW^v 
		&& \text{in }{\cO}_{T,R}, 
		\\
		& \chi_m c_m dv+\nabla \cdot 
		\left(\bM_e(\bx,\bF)\Grad v_e\right)\dt 
		+\chi_m \Iion(v,w)\dt
		\\ 
		& \qquad\qquad \qquad\qquad\qquad
		=-\Iap \dt+\beta_v(v) \dW^v 
		&& \text{in $\cO_{T,R}$}, 
		\\ 
		& dw  = H(v,w)\dt + \beta_w (v)\dW^w, 
		\quad \pt \gamma-G(\gamma,w)=0 
		&& \text{in $\cO_{T,R}$},
	\end{aligned}
\end{equation}
along with the algebraic identity $v_i-v_e=v$ 
in $\cO_{T,R}$. Here, $\cO_{T,R}:=
(0,T)\times\cO_R$ and $\cO_R$ 
is the reference configuration defined in Section \ref{sub:elasticity}. 

In \eqref{coupbid}, 
\begin{equation}\label{eq:flux-derived}
	a(\bx,\gamma,\bF,p) =\mu \bF\bC_a^{-1}(\bx,\gamma)
	-p\Cof(\bF)
\end{equation}
and
\begin{equation*}
	\bM_j(\bx,\bF)
	=(\bF)^{-1}\bK_k(\bx)(\bF)^{-T}, 
	\quad j=i,e,
\end{equation*}
where $\bC_a^{-1}$ is the tensor obtained from the 
active factor of the deformation gradient. To complete the 
system of equations \eqref{coupbid}, we need to specify suitable 
initial conditions for $v$, $w$ and $\gamma$, as well as 
boundary conditions for $v_i,v_e$ and 
the elastic flux $a(\cdot,\cdot,\cdot,\cdot)$.\\

\subsection{Linearizing the mechanical equations}
\label{subsec:linear}

To facilitate the mathematical analysis of the stochastic electromechanical
system \eqref{coupbid}, we employ a linearization technique 
for both the incompressibility condition $\textrm{det}(\mathbf{F})=1$ 
and the flux in the equilibrium equation. 

Specifically, we use the following approach for the determinant:
\begin{align*}
	\textrm{det}(\mathbf{F}) &= \textrm{det}(\bI)
	+\dfrac{\partial(\textrm{det})}{\partial\mathbf{F}}(\bI)
	\left(\mathbf{F}-\bI\right)
	+o\left(\mathbf{F}-\bI\right) 
	\\ & = 1+\textrm{tr}(\mathbf{F}-\bI)
	+o\left(\mathbf{F}-\bI\right).
\end{align*}
Since $\textrm{det}(\mathbf{F})=1$, we can use the approximations
$$
\textrm{tr}(\mathbf{F}-\bI)\approx 0 
\quad
\nabla\cdot\boldsymbol{\phi}
=\textrm{tr}(\mathbf{F})
\approx\textrm{tr}(\bI)=n.
$$

We can express the displacement $\bu$ as the difference 
between the current position $\boldsymbol{\phi}(X)$ and 
the reference position $X$. With this notation, the linearized 
incompressibility condition takes the form $\nabla\cdot\bu=0$. 
To linearize the flux in equation \eqref{eq:flux-derived} with 
respect to $\bF$, we employ a Taylor expansion of $\textrm{Cof}(\bF)$ 
around the identity matrix $\mathbf{I}$:
\begin{align*}
	\textrm{Cof}(\bF) & =\textrm{Cof}(\bI)
	+\dfrac{\partial\textrm{Cof}}{\partial\mathbf{\bF}}(\bI)
	\left(\bF-\bI\right)+o\left(\bF-\bI\right)
	\\ & = \bI+\textrm{tr}\left(\bF-\bI\right)\bI
	-\left(\bF-\bI\right)^T+o\left(\bF-\bI\right).
\end{align*}
As a result, we arrive at
\begin{equation*}
	a(\bx,\gamma,\bF,p)
	=\mu \bF \bC_a^{-1}(\bx,\gamma)-p\bI.
\end{equation*}

We can use the displacement gradient $\Grad\bu$ 
to write the first equation in \eqref{coupbid} as
$$
-\Div \bigl((\mathbf{I}+\Grad\bu)
\sigma(\bx,\gamma)\bigr) +\Grad p=\mathbf{g},
\quad 
\sigma(\bx,\gamma):=\mu\bC_a^{-1}(\bx,\gamma).
$$
Therefore, we can reformulate the last equation 
to obtain a Stokes-like equation of the form
$$
-\Div \bigl(\Grad\bu\,\sigma(\bx,\gamma)\bigr) 
+ \Grad p=\ff(t,\bx,\gamma),
$$
where 
\begin{equation}\label{def:ff}
	\ff(t,\bx,\gamma)=\Div \sigma(\bx,\gamma)+\mathbf{g}.
\end{equation}

\subsection{Problem to be solved and main result}
\label{subsec:main-problem}

Gathering the relevant equations from the previous subsections,
we obtain the stochastic electromechanical bidomain model that 
is proposed in this paper. The model combines the electrical activity of the heart 
and its mechanical deformation. We will prove a global 
existence result in the upcoming sections. To state our main global 
existence result, we now present the final form of the model. 

For simplicity of notation, we use $\cO$ and $\cO_T$ to 
denote $\cO_R$ and $\cO_{T,R}$, respectively, 
throughout the rest of the paper. We also redefine the 
conductivities $\bM_{i,e}$ to
$\frac{1}{\chi_m c_m}\bM_{i,e}$ 
and $\Iion$, $\Iap$, $\beta_v$ to 
$\frac{1}{c_m} \Iion$, $\frac{1}{\chi_m c_m} \Iap$, 
$\frac{1}{\chi_m c_m} \beta_v$.

\medskip

\noindent The final form of \textit{stochastic 
electromechanical bidomain model} is
\begin{equation}\label{S1}
	\boxed{
	\begin{aligned}
		&-\Div \bigl( \nabla\bu \sigma (\bx,\gamma)\bigr)
		+\Grad p =\ff(t,\bx,\gamma), 
		\quad \Div \bu =0
		&& \text{in $\cO$, $t\in (0,T)$},
		\\
		&dv -\nabla \cdot 
		\left(\bM_i(\bx,\nabla\bu)\Grad v_i\right)\dt
		+\Iion(v,w)\dt
		\\ 
		& \qquad\qquad\qquad\qquad \qquad\qquad
		=\Iap \dt+\beta_v(v) \dW^v 
		&& \text{in }{\cO}_T, 
		\\
		& dv+\nabla \cdot 
		\left(\bM_e(\bx,\nabla\bu)\Grad v_e\right)\dt 
		+\Iion(v,w)\dt
		\\ 
		& \qquad\qquad \qquad\qquad\qquad\qquad
		=-\Iap \dt+\beta_v(v) \dW^v 
		&& \text{in $\cO_T$}, 
		\\ 
		& dw  = H(v,w)\dt + \beta_w (v)\dW^w, 
		\quad \pt \gamma-G(\gamma,w)=0 
		&& \text{in $\cO_T$},
	\end{aligned}}
\end{equation}
along with the algebraic relation $v_i-v_e=v$ in $\cO_T$.

\medskip
The system \eqref{S1} is fully specified by the inclusion 
of boundary conditions, which includes the linearization of the 
condition \eqref{bc}:
\begin{equation}\label{S-bc}
	\begin{split}
		& \nabla\bu \sigma (\bx,\gamma)\,\bn 
		-p\bn=-\alpha \bu 
		\quad \text{on $\ptl\cO$, $t\in (0,T)$},
		\\ & \left(\bM_{k}(\bx,\nabla\bu)\Grad v_{k}\right)\cdot \bn =0 
		\quad \text{on $(0,T)\times \ptl\cO$, $k\in \Set{i,e}$},
	\end{split}
\end{equation}
for some $\alpha>0$. By selecting appropriate 
boundary conditions \eqref{S-bc}, it becomes 
feasible to impose the compatibility constraint \eqref{compat1} mentioned 
below, among other benefits.

We also impose the following initial conditions:
\begin{equation}\label{S-init}
	v(0,\cdot)=v_0,
	\quad w(0,\cdot)=w_0,
	\quad 
	\gamma(0,\cdot)=\gamma_0
	\quad \text{in $\cO$}.
\end{equation}

The following list of assumptions made in the model 
described by \eqref{S1}, \eqref{S-bc}, and \eqref{S-init} 
play a crucial role in the subsequent analysis:

\medskip

\noindent {\bf (A.1)} $\sigma(\bx,\gamma)$ is 
a symmetric, uniformly bounded, and positive definite tensor:
\begin{align*}
	& \text{$\exists c>0$ such that for 
	a.e.~$\bx\in\cO$, $\forall\gamma\in\R$, 
	$\forall \bM\in\mathbb{M}_{3\times3}$},
	\\ & \qquad \frac{1}{c}\abs{\bM}^2
	\leq \bigl(\sigma(\bx,\gamma)\bM\bigr):\bM 
	\leq c \abs{\bM}^2;
\end{align*}
Furthermore, the function $\sigma(\cdot,\cdot)$ is in the 
class $C^1(\bar{\cO}\times \mathbb{R})$. Meanwhile, 
the function $\bg$, see \eqref{def:ff}, is an 
element of $(L^2(\cO_T))^3$.

\medskip

\noindent {\bf (A.2)} $\bM_{i,e}(\bx,\bM)$ are symmetric, 
uniformly bounded and positive definite matrices:
\begin{align*}
	& \text{$\exists c>0$ such that for a.e.~$\bx\in\cO$, 
	$\forall \bM\in \mathbb{M}_{3\times 3}$, 
	$\forall \xi\in\mathbb{R}^3$}, 
	\\ & \qquad \frac{1}{c} \abs{\xi}^2
	\leq \bigl(\bM_{i,e}(\bx,\bM)\xi\bigr)
	\cdot \xi \leq c \abs{\xi}^2;
\end{align*}
Besides, the maps $\bM \mapsto \bM_{i,e}(\bx,\bM)$ are 
Lipschitz continuous, uniformly in $\bx$.

\medskip

\noindent {\bf (A.3)} The function $G$ is given by 
$G(\gamma,w)=\eta_1\bigl(\beta w-\eta_2\gamma\bigr)$ 
with $\beta,\eta_1,\eta_2>0$.

\medskip

\noindent {\bf (A.4)}  The functions $H$, $\Ion$ are given 
by the generalized FitzHugh-Nagumo kinetics:
$$
\Ion(v,w)=I_1(v)+I_2(v)w,\qquad H(v,w)=h(v)+c_{H,1} w,
$$
where $I_1,I_2, h\in C^1(\R)$ and for all $v\in \R$,
\begin{align*}
	& \abs{I_1(v)}\le c_{I,1}\left(1+ \abs{v}^3\right), 
	\quad 
	I_1(v)v \ge \underline{c}_I \abs{v}^4 - c_{I,2} \abs{v}^2,
	\\ & I_2(v)=c_{I,3}+c_{I,4}v, \quad 
	\abs{h(v)}^2\le c_{H,2}\left(1 + \abs{v}^2\right),
\end{align*}
for some positive constants $c_{I,1},c_{I,2},c_{I,3},c_{I,4}, 
c_{H,1},c_{H,2}$, and $\underline{c}_I>0$. 
Moreover, there exist positive constants $\mu$, $C$ such that 
\cite[p.~478--479]{Bourgault:2009aa}
\begin{equation}\label{ineq:dissipative}
	\begin{split}
		& \mu\bigl( \Ion(v_2,w_2)-\Ion(v_1,w_1) \bigr)
		\left(v_2-v_1\right)
		\\ & \quad 
		-\bigl( H(v_2,w_2)-H(v_1,w_1) \bigr) \left(w_2-w_1\right)
		\\ & \quad\quad 
		\geq -C \min\bigl(1,\mu^{-1}\bigr)
		\bigl(\mu \abs{v_2-v_1}^2+\abs{w_2-w_1}^2\bigr),
		\quad\forall v_1,v_2,w_1,w_2 \in\R.
	\end{split}
\end{equation}

\medskip

\noindent {\bf (A.5)}  The following combatibility condition holds
\begin{equation}\label{compat1}
	\int_\cO v_e(t,\bx)\dx=0 
	\quad \text{for a.e.~$t\in(0,T)$}.
\end{equation}

To present our main result, we need to establish some 
additional assumptions about the stochastic data of the system \eqref{S1}. 
Specifically, we will assume the existence of a complete probability 
space $(\Omega,\cF, P)$, which is equipped with a complete 
right-continuous filtration $\Set{\cF_t}_{t\in [0,T]}$. 
Consider a sequence $\Set{W_k}_{k=1}^\infty$ 
of independent finite-dimensional Brownian motions adapted 
to the filtration $\Set{\cF_t}_{t\in [0,T]}$. We will refer to the 
collection of all these objects as the 
(Brownian) \textit{stochastic basis}, denoted by
\begin{equation*}
	\cS=\left(\Omega,\cF,\Set{\cFt}_{t\in [0,T]},P,
	\Set{W_k}_{k=1}^\infty\right).
\end{equation*}

All relevant processes are defined on a stochastic basis $\cS$ and 
they are assumed to be measurable (predictable) with respect to the 
filtration $\Set{\cFt}_{t\in [0,T]}$. For those who need background 
information on stochastic analysis and SPDEs, including 
stochastic integrals, It\^o's chain rule, and martingale inequalities, we 
recommend consulting \cite{DaPrato:2014aa,Prevot:2007aa}.
To learn about key concepts related to probability measures 
(on topological spaces), weak compactness and tightness, 
we suggest reading \cite{Bogachev-measures:2018}.
For a primer on quasi-Polish spaces and the Skorokhod--Jakubowski 
representation theorem \cite{Jakubowski:1997aa}, readers can consult \cite{Brzezniak:2013ab,Brzezniak:2016wz,Ondrejat:2010aa}.
For properties of Bochner spaces, such as 
$L^r(\Omega;X)= L^r\bigl(\Omega,\cF,P;X\bigr)$, where $X$ 
is a Banach space and $r\in [1,\infty]$, we 
suggest consulting \cite{BanachI:2016}.

Fixing a Hilbert space $\U$ equipped with a complete orthonormal 
basis $\Set{\psi_k}_{k\geq 1}$, we always consider cylindrical Brownian 
motions $W$ evolving over $\U$. Specifically, we define 
$W:=\sum_{k\geq 1} W_k\psi_k$, where $\Set{W_k}_{k\geq 1}$ 
are independent standard Brownian motions on $\mathbb{R}$. 
The vector space of all bounded linear operators from $\U$ to a 
separable Hilbert space $\bX$ with inner product $(\cdot,\cdot)_{\bX}$ 
and norm $\abs{\cdot}_{\bX}$ is denoted by $L(\U,\bX)$. 
In particular, we are interested in the collection of 
Hilbert-Schmidt operators from $\U$ to $\bX$, 
denoted by $L_2(\U,\bX)$. 

For the stochastic electromechanical 
bidomain system \eqref{S1}, a natural choice for $\bX$ is the 
separable Hilbert space $L^2(\cO)$:
$$
\bX=L^2(\cO).
$$
 To ensure that the 
infinite series $\sum_{k\geq 1} W_k\psi_k$ converges, we can 
construct an auxiliary Hilbert space $\U_0$ that contains $\U$, and 
a Hilbert-Schmidt embedding $J$ from $\U$ to $\U_0$. 
With this construction, the 
series converges in the space $L^2\bigl(\Omega;
C([0,T];L_2(\U_0,L^2(\cO)))\bigr)$ \cite{DaPrato:2014aa}.

Considering a cylindrical Wiener process $W$ 
with noise coefficient $\beta$, the stochastic integral $\int \beta\,dW$ is 
defined in the sense of It\^{o} as follows \cite{DaPrato:2014aa}:
\begin{equation*}
	\int_0^t \beta \dW
	=\sum_{k=1}^\infty \int_0^t 
	\beta_{k} \dW_{k}, 
	\qquad \beta_{k} := \beta\psi_k, 
\end{equation*}
for any predictable integrand $\beta$ satisfying 
$$
\beta \in L^2\bigl(\Omega,\cF,P;
L^2(0,T;L_2(\U,L^2(\cO)))\bigr).
$$
The stochastic integral $\int_0^{\cdot} \beta \dW$ 
is a predictable process in $L^2(\Omega\times [0,T];L^2(\cO))$. 
We will write $\int_0^t \int_{\cO}  \beta \vphi \dx \dW(s)$
instead of $\left(\int_0^t\beta \dW(s),
\vphi \right)_{L^2(\cO)}$, for any $\vphi=\vphi(\bx)\in L^2(\cO)$.

This explains how we interpret the stochastic integrals 
in \eqref{S1}, with $(\beta,\beta_k, W,W_k)$ replaced by
$(\beta_v(v),\beta_{v,k}(v),W^v,W^v_k)$ and 
$(\beta_w(v),\beta_{w,k}(v),W^w,W^w_k)$.

\medskip

Next, we gather the conditions that need to be 
imposed on the noise coefficients.

\medskip

\noindent {\bf (A.6)}  For each $z\in L^2(\cO)$ and 
$u=v,w$, we define $\beta_u(z):\U\to L^2(\cO)$ by
$$
\beta_u(z)\psi_k=\beta_{u,k}(z(\cdot)),  
\quad k\ge 1,
$$
for some real-valued functions $\beta_{u,k}(\cdot):\R\to \R$, 
$\beta_{u,k}(\cdot)\in C^1$, verifying 
\begin{equation}\label{eq:noise-cond}
	\begin{split}
		& \sum_{k=1}^\infty \abs{\beta_{u,k}(z)}^2 \le 
		C_\beta \left(1+ \abs{z}^2\right), 
		\qquad \forall z\in \R,
		\\ 
		& \sum_{k=1}^\infty 
		\abs{\beta_{u,k}(z_1)-\beta_{u,k}(z_2)}^2 
		\le C_\beta\abs{z_1-z_2}^2, 
		\qquad \forall z_1, z_2\in \R,
	\end{split}
\end{equation}
for a constant $C_\beta>0$. As a result  
of \eqref{eq:noise-cond}, one has 
\begin{equation}\label{eq:noise-cond2}
	\begin{split}
		&\norm{\beta_u(z)}_{L_2\left(\U,L^2(\cO)\right)}^2
		\le C_\beta \left(1 + \norm{z}_{L^2(\cO)}^2\right), 
		\quad z\in L^2(\cO),
		\\ & 
		\norm{\beta_u(z_1)
		-\beta_u(z_2)}_{L_2\left(\U,L^2(\cO)\right)}^2
		\le C_\beta \norm{z_1 - z_2}_{L^2(\cO)}^2, \quad
		z_1, z_2 \in L^2(\cO).
	\end{split}
\end{equation}

\medskip

The following theorem encapsulates our main global 
existence result for the stochastic electromechanical bidomain model. 
However, we will postpone the exact definition of 
the solution concept until Section \ref{sec:defsol}.

\begin{thm}[global existence of solution]\label{thm:martingale}
Assume that the conditions {\bf (A.1)}--{\bf (A.6)} are satisfied. 
Given a stimulation current $\Iapp\in L^2\left(\Omega;L^2(\cO_T)\right)$
and (initial) probability measures $\mu_{v_0}$ on $L^2(\cO)$ 
and $\mu_{w_0}, \mu_{\gamma_0}$ on $H^1(\cO)$, which 
satisfy the following moment conditions for some $q_0 > \frac{9}{2}$:
\begin{equation}\label{eq:moment-est}
	\int_{L^2(\cO)} 
	\norm{z}^{q_0}_{L^2(\cO)}
	\, d\mu_{z_0}(z)<\infty, \quad 
	\text{ for $z=v,w,\gamma$},
\end{equation}
the stochastic electromechanical bidomain model \eqref{S1}---with initial data  
\eqref{S-init} (where $v_0\sim\mu_{v_0}$, $w_0\sim \mu_{w_0}$, 
$\gamma_0\sim \mu_{\gamma_0}$) and boundary 
conditions \eqref{S-bc}---will have a weak 
martingale solution in the sense of Definition \ref{def:martingale-sol}.
\end{thm}

\medskip

The organization of the remaining part of the paper 
is as follows: First, Section \ref{sec:defsol} introduces 
the concept of a weak martingale solution. Then, in 
Section \ref{sec:approx-sol}, we present the construction 
of approximate (Faedo-Galerkin) solutions for an 
auxiliary non-degenerate system. Next, Section \ref{sec:apriori} 
presents a series of a priori estimates. These estimates play 
a crucial role in establishing the convergence of the approximate 
solutions for the auxiliary problem. This convergence is 
tied to demonstrating the tightness of the probability laws 
of the Faedo-Galerkin solutions, as discussed 
in Section \ref{sec:tightness}. Moving forward, Section \ref{sec:conv-nondegen} 
establishes the existence of a solution for the non-degenerate problem. 
Subsequently, in Section \ref{sec:conv-degen}, we prove the 
existence of a solution for the original problem. This section 
focuses on demonstrating the weak 
convergence of the pressure variable, which does not 
exhibit a uniform $L^2$ bound. 
By interconnecting and building upon one another, these sections 
form a proof of the main result (Theorem \ref{thm:martingale}). 
Finally, some numerical examples are 
presented in Section \ref{sec:numeric}.

\section{Notion of solution}\label{sec:defsol}
To specify the initial data for the problem under consideration, it 
is important to keep in mind that we are considering a probabilistic
weak notion of solution. For  probabilistic strong (pathwise) 
solutions, we prescribe the initial data as random variables. 
Specifically, we require that $v_0\in L^2(\Omega; L^2(\cO))$ 
and $w_0,\gamma_0\in L^2(\Omega; H^1(\cO))$. 
On the other hand, if we are interested in martingale or 
probabilistic weak solutions, we cannot rely on the knowledge of 
the underlying stochastic basis. Therefore, we must specify the 
initial data in terms of probability measures. More precisely, we 
prescribe the initial data as probability measures $\mu_{v_0}$ 
on $L^2(\cO)$ and $\mu_{w_0}, \mu_{\gamma_0}$ 
on $H^1(\cO)$. In this context, the measures 
$\mu_{v_0}$, $\mu_{w_0}$ and $\mu_{\gamma_0}$ should 
be viewed as the ``initial laws" because the laws of $v(0)$, $w(0)$, 
and $\gamma(0)$ will be required to coincide with $\mu_{v_0}$, $\mu_{w_0}$, 
and $\mu_{\gamma_0}$, respectively.

We will utilize the following notion of 
solution for our stochastic system \eqref{S1}.

\begin{defi}[weak martingale solution] \label{def:martingale-sol}
Let $\mu_{v_0}$ be a probability measure defined on $L^2(\cO)$ 
and $\mu_{w_0}$ and $\mu_{\gamma_0}$ be probability 
measures on $H^1(\cO)$. 
A weak martingale solution of the stochastic electromechanical bidomain 
system \eqref{S1}---with initial data \eqref{S-init} and 
boundary data \eqref{S-bc}---is a collection 
$$
\bigl(\cS,v_i,v_e,v,\bu,p,w,\gamma\bigr)
$$ 
satisfying the following conditions:
\begin{enumerate}
	\item\label{eq:mart-sto-basis}
	$\cS=\bigl(\Om, \cF, \Set{\cFt}, P, \Set{W_k^v}_{k=1}^\infty,
	\Set{W_k^w}_{k=1}^\infty\bigr)$ is a stochastic basis;

	\item \label{eq:mart-wiener}
	$W^v:=\sum_{k\ge 1} W^v_k e_k$ and 
	$W^w:=\sum_{k\ge 1} W^w_k e_k$ are 
	two independent \\ cylindrical Brownian motions, adapted 
	to the filtration $\Set{\cFt}$;

	\item\label{eq:mart-bup-reg}
	$\bu \in L^2(0,T;[H^1(\cO)]^3)$, a.s., 
	and $p\in L^2((0,T)\times \cO)$, a.s.;

	\item\label{eq:mart-uiue-reg}
	$v_i,v_e\in L^2(0,T;H^1(\cO))$, a.s. 
	Moreover, a.s., $v_e$ satisfies \eqref{compat1};
	 
	\item\label{eq:mart-vreg}
	$v=v_i-v_e\in L^2(0,T;H^1(\cO))
	\cap L^4((0,T)\times \cO)$, a.s.; 

	\item\label{eq:mart-adapt}
	$v,w:\Om\times [0,T]\to L^2(\cO)$  
	are $\Set{\cFt}$-predictable in $(H^1(\cO))^*$, and
	$$
	v,w,\gamma\in 
	L^\infty(0,T;L^2(\cO))\cap C([0,T];(H^1(\cO))^*)\,\, \text{a.s.};
	$$

	\item\label{eq:mart-data}
	The laws of $v_0:=v(0)$, $w_0:=w(0)$, $\gamma_0:=\gamma(0)$ are 
	respectively $\mu_{v_0}$, $\mu_{w_0}$, $\mu_{\gamma_0}$;

	\item\label{eq:mart-weakform}
	The following identity holds a.s., in $\Dp(0,T)$,
	\begin{equation*}
		\begin{split}
			& \int_\cO \Bigl(\nabla\bu \bsigma(\bx,\gamma):
			\nabla\bv - p\nabla \cdot \bv\Bigr)\dx 
			+\int_{\ptl \cO} \alpha \bu\cdot\bv \,d S(\bx)
			\\ & \qquad 
			=\int_\cO\ff(t,\bx,\gamma)\cdot\bv \dx,
			\qquad 
			\int_\cO q\Div \bu \dx =0,
		\end{split}
	\end{equation*}
	for all $\bv\in H^1(\cO)^3$ and $q\in L^2(\cO)$, 
	and the following identities hold a.s., 
	for a.e.~time $t \in [0,T]$:
	\begin{equation}\label{eq:weakform}
		\begin{split}
			& \int_{\cO} v(t) \vphi_i \dx 
			+ \int_0^t \int_{\cO} 
			\Bigl( \bM_i(\bx,\nabla\bu)\Grad v_i \cdot \Grad \vphi_i 
			+ \Iion(v,w) \vphi_i \Bigr) \dx\ds
			\\ &  \qquad \qquad 
			= \int_{\cO} v_0 \, \vphi_i \dx + \int_0^t \int_{\cO} 
			 \Iap \vphi_i  \dx\ds			
			+  \int_0^t \int_{\cO} \beta_v(v) \vphi_i \dx  \dW^v (s), 
			\\ & 
			\int_{\cO} v(t)  \vphi_e \dx
			+ \int_0^t \int_{\cO} 
			\Bigl (- \bM_e(\bx,\nabla\bu)\Grad v_e\cdot \Grad \vphi_e
			+ \Iion(v,w) \vphi_e \Bigr) \dx\ds 
			\\ &  \qquad \qquad
			=\int_{\cO} v_0 \vphi_e \dx-\int_0^t \int_{\cO} 
			 \Iap \vphi_e  \dx\ds
			 + \int_0^t \int_{\cO} \beta_v(v)  \vphi_e \dx \dW^v(s),
			 \\ &
			 \int_{\cO} w(t) \vphi\dx =\int_{\cO} w_0 \vphi \dx
			 +\int_0^t \int_{\cO} H(v,w)\vphi \dx \ds
			 \\ & \qquad \qquad \qquad\qquad \qquad\qquad\,
			 + \int_0^t \int_{\cO}  \beta_w(v) \vphi \dx \dW^w(s),
			 \\ &
			 \int_{\cO} \gamma(t) \vphi\dx 
			 =\int_{\cO} \gamma_0 \vphi \dx
			 +\int_0^t \int_{\cO} G(\gamma,w)\vphi \dx \ds,
		\end{split}
	\end{equation}
	for all $\vphi_i,\vphi_e \in H^1(\cO)$ 
	and $\vphi\in L^2(\cO)$. 
\end{enumerate} 
\end{defi} 
 
\begin{rem}
Given the regularity conditions outlined in Definition \ref{def:martingale-sol}, 
it can be readily verified that the deterministic and stochastic 
integrals presented in \eqref{eq:weakform} are well-defined.

Denote by $C\left([0,T];L^2(\cO)\mathrm{-weak}\right)$
the space of weakly continuous $L^2(\cO)$ functions. 
Sometimes we write $(L^2(\cO))_w$ instead of 
$L^2(\cO)\mathrm{-weak}$, and similarly for other spaces.
According to \cite[Lemma 1.4,p. 263]{Temam:1977aa}, part \eqref{eq:mart-adapt} 
of Definition \ref{def:martingale-sol} implies that 
$$
v,w\in C\left([0,T];L^2(\cO)\mathrm{-weak}\right),\quad \text{a.s.}
$$ 
Ultimately, our analysis will lead us to the conclusion that 
$v,w,\gamma\in C([0,T];L^2(\cO))$ 
and $w,\gamma \in L^\infty(0,T;H^1(\cO))$, a.s.
\end{rem}

\section{Approximate solutions}\label{sec:approx-sol}

In this section, we define the Faedo--Galerkin 
approximations for the stochastic system \eqref{S1}. 
We construct approximate solutions 
relative to a stochastic basis
\begin{equation*}
	\cS=\left( \Om, \cF, \Set{\cF_t}_{t\in [0,T]}, P, 
	\Set{W_k^v}_{k=1}^\infty, \Set{W_k^w}_{k=1}^\infty\right),
\end{equation*}
given $\cF_0$-measurable initial data 
$v_0,w_0 \in L^2(\Omega;L^2(\cO))$ and 
$\gamma_0 \in L^2(\Omega;H^1(\cO))$ 
with respective laws $\mu_{v_0},\mu_{w_0}$ 
on $L^2(\cO)$ and $\mu_{\gamma_0}$ 
on $H^1(\cO)$.

To define the Faedo--Galerkin solutions, we will 
use classical Hilbert bases that are orthonormal in $L^2$ 
and orthogonal in $H^1$, see for 
example \cite[p.~138--145]{raviart1998introduction}. 
More precisely, consider bases $\Set{\boldsymbol{\psi}_l}_{l\in \mathbb{N}}$ 
and $\Set{\zeta_l}_{l\in\mathbb{N}}$ such that 
$\Span\Set{\boldsymbol{\psi}_l}$ 
is dense in $ L^2(\cO)^3$ and $H^1(\cO)^3$, 
while $\Span\Set{\zeta_l}$ is dense 
in $L^2(\cO)$ and $H^1(\cO)$. 
In order to impose the compatibility condition \eqref{compat1}, we 
introduce the modified functions
$$
\mu_l:=\zeta_l-\dfrac{1}{\abs{\cO}}
\int_{\cO}\zeta_l\dx
\quad \Longrightarrow \quad
\int_{\cO}\mu_l\dx=0.
$$
We observe that $\Span\Set{\mu_l}$ is dense 
in the space 
$$
H^{1,0}(\cO):=\Set{\mu\in H^1(\cO):
\int_{\cO}\mu\dx=0}.
$$
Moreover, we can apply the Gram-Schmidt process to 
orthonormalize the basis $\Set{\mu_l}$. The 
resulting $L^2(\cO)$-orthonormal 
basis is still denoted by $\Set{\mu_l}$. 
Furthermore, $\Span\Set{\mu_l}$ is 
dense in $H^{1,0}(\cO)$. 
For $n\in \N$, we introduce the finite dimensional spaces 
\begin{equation}\label{eq:Galerkin-spaces}
	\begin{split}
		& \mathbf{H}_n=\operatorname{span}
		\Set{\boldsymbol{\psi}_1,\ldots,\boldsymbol{\psi}_n}
		\subset \bigl(H^1(\cO)\cap L^\infty(\cO)\bigr)^3,
		\\ & 
		\Bbb{L}_n=\operatorname{span}
		\Set{\mu_1,\ldots,\mu_n}
		\subset H^{1,0}(\cO)\cap L^\infty(\cO),
		\\ & 
		\Bbb{W}_n=\operatorname{span}
		\Set{\zeta_1,\cdots,\zeta_n}
		\subset H^1(\cO)\cap L^\infty(\cO).
	\end{split}
\end{equation}
The basis functions $\mu_l$ and $\zeta_l$ are 
conventionally derived from the eigenfunctions 
of the Neumann-Laplace operator. This condition ensures that 
$\partial e_l/\partial n = 0$ on $\partial \cO$, 
where $e_l$ represents either $\mu_l$ or $\zeta_l$ for 
each $l \in \N$. Drawing from elliptic regularity theory, we 
recognize that every such eigenfunction $e_l$ is a member of $H^2(\cO)$. 
Furthermore, the degree of smoothness of $e_l$ within $\bar{\cO}$ 
is determined by the smoothness of $\partial\cO$. For instance, 
if $\partial \cO$ is $C^\infty$, then $e_l \in C^\infty$. For 
more details, see the summary in \cite[p. 9-10]{Bendahmane:2022vy} 
(for example).

If we were to apply the Faedo--Galerkin method directly to 
the highly degenerate system \eqref{S1}, we would end up with 
a set of coupled stochastic differential equations (SDEs) 
and algebraic equations that would need to be solved at each time. 
Unfortunately, the classical SDE theory does not provide a straightforward 
guarantee of the existence and uniqueness of solutions for this type of system.  
To overcome this challenge, we will use a time regularization 
approach inspired by \cite{Bendahmane:2006au}.
This approach considers an auxiliary non-degenerate system 
obtained by adding temporal regularizations to the first 
four equations in \eqref{S1}. Specifically, we add 
$\eps \pt \bu_\eps$, $\eps \pt p_\eps$, $\eps d v_{i,\eps}$, 
and $-\eps d v_{e,\eps}$, respectively, where $\eps>0$ 
is a small parameter that we later send to zero. 

Fixing a small number $\eps>0$, we first look for a solution 
$$
\bU_\eps=\left(\bu_\eps, p_\eps,
v_\eps,v_{i,\eps}, v_{e,\eps},
w_\eps,\gamma_\eps\right)
$$ 
of the following non-degenerate system 
\begin{equation}\label{eq:nondegen-S1}
	\begin{split}
		&\eps \pt \bu_\eps
		-\Div \bigl( \nabla\bu_\eps 
		\sigma(\bx,\gamma_\eps)\bigr)
		+\Grad p_\eps =\ff(t,\bx,\gamma_\eps),
		\\ &  \eps \pt p_\eps +\Div\bu_\eps =0, 
		\\ 
		& d v_\eps+\eps d v_{i,\eps}
		-\Div\bigl(\bM_i(\bx,\nabla\bu_\eps)
		\Grad v_{i,\eps}\bigr)\dt
		\\
		& \qquad \qquad 
		+\Ion(v_\eps,w_\eps)\dt
		=\Iap\dt+\beta_v(v_\eps) \dW^v, 
		\\ 
		& dv_\eps-\eps d v_{e,\eps}
		+\Div\bigl(\bM_e(\bx,\nabla\bu_\eps)
		\Grad v_{e,\eps}\bigr)\dt
		\\ 
		& \qquad \qquad 
		+\Ion(v_\eps,w_\eps)\dt
		=-\Iap\dt+\beta_v(v_\eps) \dW^v , 
		\\ 
		& d w_\eps =H(v_\eps,w_\eps)\dt 
		+\beta_w(v_\eps)\dW^w,
		\\ &
		\pt \gamma_\eps
		-G(\gamma_\eps,w_\eps) =0,
	\end{split}
\end{equation}
where 
$$
v_\eps=v_{i,\eps}-v_{e,\eps}.
$$
In the system \eqref{eq:nondegen-S1}, the elasticity equation 
has been augmented to incorporate the dynamic 
velocity term $\eps \partial_t \bu_\eps$. Meanwhile, the addition of 
the term $\eps \partial_t p_\eps$ in the second equation 
represents an ``artificial compressibility''. 
The system \eqref{S1} also includes the boundary conditions \eqref{S-bc}. 
However, when it comes to specifying the initial data 
for $\bU_\eps$, we need to exercise extra caution 
since \eqref{S-init} only provides initial data for 
$v_\eps$, $w_\eps$, and $\gamma_\eps$. Therefore, we must 
also supply auxiliary initial data for $\bu_\eps$, 
$p_\eps$, $v_{i,\eps}$, and $v_{e,\eps}$. We will 
explore this issue in greater detail later on.

We can now move on to the next step, which involves applying 
the Faedo--Galerkin method to the non-degenerate system 
\eqref{eq:nondegen-S1}.This method consists of projecting 
\eqref{eq:nondegen-S1} onto the finite-dimensional spaces 
in \eqref{eq:Galerkin-spaces}. To accomplish this, we 
consider approximate solutions that take the following form:
\begin{equation}\label{def:approx-sol}
	\begin{split}
		& \bu^n=\sum_{l=1}^{n}\bu^n_{l}
		\boldsymbol{\psi}_l, 
		\quad 
		p^n=\sum_{l=1}^{n}p^n_{l}\zeta_l,
		\\ & 
		w^n=\sum_{l=1}^{n}w^n_{l}\zeta_l, 
		\quad
		\gamma^n=\sum_{l=1}^{n}
		\gamma^n_{l}\zeta_l,
		\\ 
		& v^n_i=\sum_{l=0}^{n}v^n_{i,l}\zeta_l, 
		\quad
		v^n_e=\sum_{l=1}^{n}v^n_{e,l}\mu_l,
		\quad 
		v^n=v^n_i-v^n_e.
	\end{split}
\end{equation}
The time-dependent coefficients $\bu^n_{l}(t), p^n_{l}(t), v^n_{i,l}(t), 
v^n_{e,l}(t), w^n_{l}(t),\gamma^n_{l}(t)$ are determined 
by requiring that the following (Galerkin) equations hold for $l=1,\ldots,n$:
\begin{equation}\label{syst:weakRegElast}
	\begin{split}
		& \eps \dfrac{d}{dt}\int_{\cO}
		\bu^n\cdot \boldsymbol{\psi}_{l}\dx
		+\int_{\cO}\nabla\bu^n
		\bsigma(\bx,\gamma^n):
		\nabla \boldsymbol{\psi}_l\dx
		-p^n\Grad\cdot\boldsymbol{\psi}_l \dx
		\\ 
		& \qquad \qquad \qquad 
		+\int_{\partial \cO}\alpha
		\, \bu^n\cdot \boldsymbol{\psi}_{l} \, dS(\bx)
		=\int_{\cO} \ff(t,\bx,\gamma^n)
		\cdot \boldsymbol{\psi_l} \dx,
		\\ 
		& \eps \dfrac{d}{dt}\int_{\cO}p^n\cdot \zeta_l\dx
		+\int_{\cO}\nabla\cdot\bu^n \zeta_l\dx=0,
		\\ 
		& d\int_{\cO} v^n\zeta_l\dx
		+\eps \, d \int_{\cO}v^n_i\zeta_l\dx
		+\int_{\cO}\bM_i(\bx,\Grad\bu^n)
		\nabla v^n_i\cdot\nabla\zeta_l \dx \dt
		\\ 
		& \,\,
		+\int_\cO\Ion(v^n,w^n)\zeta_l \dx \dt
		= \int_{\cO} \Iap\zeta_l\dx \dt
		+\sum_{k=1}^n\int_\cO
		\beta_{v,k}(v^n)\zeta_l\dx \dW_k^{v}(t),
		\\ 
		& d \int_{\cO}v^n\mu_l\dx
		-\eps \, d\int_{\cO}v^n_e\mu_l\dx
		-\int_{\cO} \bM_e(\bx,\Grad\bu^n)
		\nabla v^n_e \cdot \nabla\mu_l \dx\dt
		\\
		&\, \,
		+\int_\cO \Ion(v^n,w^n)\mu_l \dx \dt
		=-\int_{\cO} \Iap\mu_l\dx
		+\sum_{k=1}^n\int_\cO
		\beta_{v,k}(v^n)\mu_l\dx \dW_k^{v}(t),
		\\ 
		& d\int_{\cO} w^n \zeta_l\dx
		=\int_{\cO} 
		H(v^n,w^n)\zeta_l\dx\dt
		+\sum_{k=1}^n\int_\cO
		\beta_{w,k}(v^n)\zeta_l\dx \dW_k^{w}(t),
		\\ 
		& \dfrac{d}{dt}\int_{\cO}
		\gamma^n \zeta_l\dx
		=\int_{\cO}
		G(\gamma^n,w^n)\zeta_l\dx.
	\end{split}
\end{equation}

Let us comment on the (finite dimensional) 
approximations of the stochastic integrals used in 
\eqref{syst:weakRegElast}. In this context, let $(\beta, W)$ 
represent either $(\beta_v, W^v)$ 
or $(\beta_w, W^w)$. Let $e_l$ denote either $\zeta_l$ 
or $\mu_l$, so that $\Set{e_l}_{l=1}^\infty$ is 
an $L^2(\cO)$-orthonormal basis. Recall that $\beta$ maps between 
$L^2 \left(0, T;L^2(\cO)\right)$ and
$L^2\left(0, T; L_2\left(\U,L^2(\cO)\right)\right)$, for fixed 
$\omega \in \Omega$, where $\U$ comes 
with the orthonormal basis $\Set{\psi_k}_{k=1}^\infty$. 
Utilizing the following decomposition
$$
\beta_k(v)=\beta(v)\psi_k,
\qquad
\beta_k(v)=\sum_{\ell=1}^\infty 
\bigl( \beta_k(v),e_\ell\bigr)_{L^2(\cO)}e_\ell
=:\sum_{\ell=1}^\infty \beta_{k,\ell}(v)e_\ell,
$$
where $\beta_{k,\ell}(v)=
\bigl( \beta_k(v),e_\ell\bigr)_{L^2(\cO)}$, we can write
\begin{align*}
	\beta(v) \dW & = \sum_{k=1}^\infty \beta_k(v) \dW_k
	= \sum_{k=1}^\infty 
	\sum_{\ell=1}^\infty \beta_{k,\ell}(v) e_\ell \dW_k, 
\end{align*}
so that by orthonormality, for any $l=1,2,\ldots$,
$$
\int_{\cO}\beta(v)e_l\dx \dW
=\sum_{k=1}^\infty \int_{\cO}
\beta_{k}(v)e_l\dx \dW_k
=\sum_{k=1}^\infty \beta_{k,l}(v)\dW_k.
$$
For related decompositions, see, e.g., \cite{Chekroun:2016aa}. 
In the Faedo--Galerkin system \eqref{syst:weakRegElast}, we 
employ the usual finite dimensional approximation
\begin{equation*}
	\beta(v) \dW\approx 
	\beta(v) \dW^n :=\sum_{k=1}^n \beta_k(v)\dW_k
	=\sum_{k=1}^n\sum_{\ell=1}^\infty \beta_{k,\ell}(v) e_\ell \dW_k,
\end{equation*}
so that, for any $l=1,\ldots,n$,
\begin{align*}
	\int_{\cO}\beta(v)e_l\dx \dW
	& \approx 
	\int_{\cO}\beta(v)e_l\dx \dW^n
	\\ & =\sum_{k=1}^n \int_{\cO}
	\beta_{k}(v)e_l\dx \dW_k
	=\sum_{k=1}^n \beta_{k,l}(v)\dW_k.
\end{align*}
Clearly \cite[Chapter 4.2.2]{DaPrato:2014aa}, $W^n$ 
converges in $C([0,T];\U_0)$ for $P$-a.e.~$\omega\in \Omega$ 
and in $L^2\left(\Omega; C([0,T];\U_0)\right)$.

Since the original problem \eqref{S1}, \eqref{S-init} lacks 
initial conditions for the unknowns $\bu$, $p$, $v_i$, and $v_e$, it 
becomes necessary to augment our non-degenerate system 
\eqref{eq:nondegen-S1} with auxiliary initial conditions. 
To address this, we introduce the functions
\begin{equation*}
	v_{i,0} =\dfrac{v_0}{2}
	+\dfrac{1}{\abs{\cO}}
	\int_\cO\dfrac{v_0}{2}\dx ,
	\quad
	v_{e,0} = -\dfrac{v_0}{2}
	+\dfrac{1}{\abs{\cO}}
	\int_\cO\dfrac{v_0}{2}\dx,
\end{equation*}
so that $v_0=v_{i,0}-v_{e,0}$ and $\int_\cO v_{e,0}\dx=0$. 
Additionally, we choose $\bu_0\equiv 0$ and select 
an arbitrary $p_0\in L^2(\cO)$ 
to further supplement the initial conditions. 

Given the assumption \eqref{eq:moment-est}, we conclude that 
\begin{equation}\label{eq:vie-init-ass-q0}
	v_{i,0},v_{e,0},w_0,\gamma_0,p_0 
	\in L^{q_0}\bigl(\Omega,\cF,P;L^2(\cO)\bigr),
\end{equation}
with $q_0$ is defined in \eqref{eq:moment-est}.

The initial data for the Faedo--Galerkin approximations 
\eqref{def:approx-sol} are now defined as follows:
\begin{equation}\label{syst:ICreg}
	\begin{split}
		&\bu^n(0)\equiv 0,
		\qquad 
		p^n(0)=\sum_{l=1}^{n}p_{0,l}\zeta_{l}, 
		\quad \text{where $p_{0,l}:=\left( p_0,\zeta_l\right)_{L^2}$},
		\\
		& v^n_i(0)=\sum_{l=1}^{n}v_{i,0,l}\zeta_{l}, 
		\quad \text{where $v_{i,0,l}:=\left( v_{i,0},\zeta_l\right)_{L^2}$}, 
		\\
		& v^n_e(0)=\sum_{l=1}^{n}v_{e,0,l}\mu_{l},
		\quad \text{where $v_{e,0,l}:=\left( v_{e,0},\mu_l\right)_{L^2}$}, 
		\\ 
		& w^n(0)=\sum_{l=1}^{n}w_{0,l}\zeta_l,
		\quad \text{where $w_{0,l}:=\left( w_0,\zeta_l\right)_{L^2}$},
		\\
		& \gamma^n(0)=\sum_{l=1}^{n}\gamma_{0,l}\zeta_l, 
		\quad 
		\text{where $\gamma_{0,l}:=\left( \gamma_{0},\zeta_l\right)_{L^2}$}.
	\end{split}
\end{equation}

In the upcoming sections, we investigate the convergence 
properties of the sequences
$\Set{\bu^n}_{n=1}^\infty$, 
$\Set{p^n}_{n=1}^\infty$, 
$\Set{v^n_j}_{n=1}^\infty$ ($j=i,e$), 
$\Set{v^n=v^n_j-v^n_e}_{n=1}^\infty$, 
$\Set{w^n}_{n=1}^\infty$, 
$\Set{\gamma^n}_{n=1}^\infty$. 
These sequences are analyzed specifically 
for any fixed $\eps>0$. Currently, our assertion is focused on the existence 
of a (pathwise) solution to the SDE system defined by 
\eqref{syst:weakRegElast} and \eqref{syst:ICreg}, for 
each fixed value of $n\in \N$.

\begin{lem}[well-posedness of the Faedo--Galerkin equations]
\label{lem:fg-solutions}
For any fixed $n\in \N$, the regularized Faedo-Galerkin equations, 
specifically the equations \eqref{syst:weakRegElast} together 
with the initial values specified in \eqref{syst:ICreg}, 
possess a unique strong solution 
$$
\Set{\bigl(\bu^n(t),p^n(t),v^n_{i}(t),v^n_{e}(t), 
w^n(t),\gamma^n(t)\bigr)} 
\, \, \text{on $[0,T]$}.
$$
Moreover, almost surely, the processes $v^n_{i}$, $w^n$, 
$\gamma^n$, $p^n$ are members 
of $C([0,T],\Bbb{W}_n)$, 
while $v^n_{e}$ belongs to $C([0,T],\Bbb{L}_n)$, and 
$\bu^n$ is in $C([0,T],\mathbf{H}_n)$. 
Additionally, for each time $t\in [0,T]$, 
these processes are square-integrable with 
respect to the probability variable.
\end{lem}

\begin{proof}
By the orthonormality of the bases and incorporating 
the algebraic relation $v^n=v_i^n-v_e^n$, the 
system \eqref{syst:weakRegElast} 
simplifies into $(l=1,\cdots,n)$:
\begin{equation}\label{syst:discreteReg}
	\begin{aligned}
		&\eps \, d\bu_l^n = A_l \, dt,
		\quad 
		\eps \, dp_l^n = P_l \, dt,
		\\
		& (1+\eps) \, dv_{i,l}^n
		-\sum_{r=1}^n\left(\int_{\cO}
		\zeta_l\mu_r\dx \right)\, dv_{e,r}^n
		= A_{i,l}\, dt+\Gamma_{i,l} \, dW^{v,n},
		\\ & -\sum_{r=1}^n\left(\int_{\cO}
		\zeta_l\mu_r\dx\right)\, dv_{i,r}^n
		+(1+\eps) \, dv_{e,l}^n
		=A_{e,l}\, dt+\Gamma_{e,l} \, dW^{v,n},
		\\ & dw_l^n =  A_{H,l}dt+\Gamma_{w,l} \, dW^{w,n},
		\quad d\gamma_l^n = A_{G,l}\, dt,
	\end{aligned}
\end{equation} 
where 
\begin{align*}
	& A_l = -\int_{\cO}\nabla\bu^n:\nabla\psi_l \dx
	+\int_{\cO} p^n\nabla\cdot \psi_l \dx
	\\ & \qquad \quad
	-\int_{\partial\cO}\alpha\bu^n\cdot\psi_l\, dS(\bx)
	+\int_{\cO} f\cdot\psi_l\dx,
	\quad 
	P_l = -\int_{\cO}\nabla\cdot\bu^n\zeta_l \dx,
	\\ 
	& A_{i,l} =  -\int_\mathcal{O}M_i(x,D\bu^n)
	\nabla v_i^n\cdot\nabla\zeta_l\dx
	-\int_{\cO}\Ion(v^n,w^n)\zeta_l\dx
	+\int_{\cO}\Iap^n\zeta_l\dx,
	\\ 
	& A_{e,l} = -\int_{\cO} M_e(x,D\bu^n)\nabla v_e^n
	\cdot\nabla\mu_l \dx
	+\int_{\cO} \Ion(v^n,w^n)\mu_l
	+\int_{\cO}\Iap^n\mu_l\dx,
	\\ 
	& A_{H,l} = \int_{\cO} H(v^n,w^n)\zeta_l \dx,
	\quad
	A_{G,l} =\int_{\cO}G(\gamma^n,w^n)\zeta_l \dx.
\end{align*}
and 
\begin{align*}
	& \Gamma_{i,l} = \{\Gamma_{i,l,k}\}_{k=1}^n, 
	\,\, \Gamma_{i,l,k} 
	= \int_{\Omega} \beta_{v,k} (v^n) \zeta_l \, dx, 
	\,\, \Gamma_{i,l} \, dW^{v,n} 
	= \sum_{k=1}^{n} \Gamma_{i,l,k} dW_k^v, 
	\\
	& \Gamma_{e,l} = \{\Gamma_{e,l,k}\}_{k=1}^n, 
	\,\, \Gamma_{e,l,k} 
	= -\int_{\Omega} \beta_{v,k}(v^n) \mu_l \, dx, 
	\,\, \Gamma_{e,\ell} dW^{v,n} 
	= \sum_{k=1}^{n} \Gamma_{e,\ell,k} dW_k^v,
	\\ 
	& \Gamma_{w,l} = \{\Gamma_{w,l,k}\}_{k=1}^n, 
	\,\, \Gamma_{w,l,k} 
	= \int_{\Omega} \beta_{w,k} (w^n) \zeta_l \, dx, 
	\,\, \Gamma_{i,l} \, dW^{w,n} 
	= \sum_{k=1}^{n} \Gamma_{w,l,k} dW_k^v.
\end{align*}

The third and fourth equations in the 
system \eqref{syst:discreteReg} 
can be expressed equivalently in matrix form as follows:
$$
\mathbf{M}
\begin{pmatrix}
	dv_{i}^n\\
	dv_{e}^n
\end{pmatrix}=b,
\qquad
\mathbf{M} 
= 
\begin{pmatrix}
	(1+\eps)\mathbf{I}_{n} & -\mathbf{A} \\
	-\mathbf{A}^T & (1+\eps)\mathbf{I}_{n}
\end{pmatrix}.
$$
Here, $\mathbf{A}=\Set{a_{lr}}$ with $a_{lr}=
\int_{\cO} \zeta_l\mu_r \dx$, $\mathbf{I}_{n}$ 
represents the identity matrix, and $b$ is the vector containing 
the right-hand side terms corresponding to the equations in 
\eqref{syst:discreteReg}. 

In order to express this system in the form
$$
\begin{pmatrix}
	dv_{i}^n\\
	dv_{e}^n
\end{pmatrix}=\mathbf{M}^{-1}b,
$$ 
we need to verify that the matrix $\mathbf{M}$ is invertible.  
To do so, let us expand $\mathbf{M}$ as follows:
$$
\mathbf{M} = 
\begin{pmatrix}
	\mathbf{I}_{n} & -\mathbf{A} \\
	-\mathbf{A}^T & \mathbf{I}_{n}
\end{pmatrix} 
+\eps 
\begin{pmatrix}
	\mathbf{I}_{n} & 0 \\
	0 & \mathbf{I}_{n}
\end{pmatrix}.
$$
 It is enough to show that the matrix 
 $$
\mathbf{N} := 
\begin{pmatrix}
	\mathbf{I}_{n} & -\mathbf{A} \\
	-\mathbf{A}^T & \mathbf{I}_{n}
\end{pmatrix}
$$ 
is positive. 
Consider 
$\xi=
\begin{pmatrix}
	\xi_{i} \\
	\xi_{e}
\end{pmatrix}$, where 
$\xi_{i} = (\xi_{i,1},\cdots,\xi_{i,n})^T$ 
and $\xi_{e} = (\xi_{e,1},\cdots,\xi_{e,n})^T$ are 
vectors in $\R^{n}$. Then
\begin{align*}
	\xi^T\mathbf{N}\xi 
	& = \xi_{i}^T\xi_{i}-\xi_{i}^T\mathbf{A}\xi_{e}
	+\xi_{e}^T\xi_{e}-\xi_{e}^T\mathbf{A}^T\xi_{i}
	\\ & = \sum_{r,l}\left[\xi_{i,r}\xi_{i,l}\int_{\cO}
	\zeta_r\zeta_l\dx-2\xi_{i,r}a_{rl}\xi_{e,l}
	+\xi_{e,r}\xi_{e,l}\int_{\cO}\mu_r\mu_l\dx\right] 
	\\ & = \int_{\cO}\sum_{r,l}\Bigl[\xi_{i,r}\xi_{i,l}\zeta_r\zeta_l
	-2\xi_{i,r}\xi_{e,l}\zeta_l\mu_r
	+\xi_{e,r}\xi_{e,l}\mu_r\mu_l\Bigr]\dx 
	\\ & = \int_{\cO}
	\left(\sum_{l}\xi_{i,l}\zeta_l\right)^2
	-2\sum_{r,l}\xi_{i,r}\xi_{e,l}\zeta_l\mu_r
	+\left(\sum_{l}\xi_{e,l}\mu_l\right)^2 \dx 
	\\ & =  \int_{\mathcal{O}}\left(\sum_{l}\xi_{i,l}\zeta_l
	-\sum_{l}\xi_{e,l}\mu_l\right)^2\dx \ge	0.
\end{align*}
Thus, the matrix $\mathbf{M}$ is positive definite, 
and as a result, it is also invertible.

Let $Y^n(t)$ represent the vector 
that encompasses all variables of interest, including the 
sequences 
$\Set{\bu_l^n(t)}_{l=1}^n$, 
$\Set{p_l^n(t)}_{l=1}^n$, 
$\Setb{\sqrt{\eps}v_{i,l}^n(t)}_{l=1}^n$,
$\Setb{\sqrt{\eps}v_{e,l}^n(t)}_{l=1}^n$,
$\Set{w_l^n(t)}_{l=1}^n$, and
$\Set{\gamma_l^n(t)\}_{l=1}^n}$.
Furthermore, define $F(Y^n)$ as the vector that aggregates 
all the deterministic drift terms derived from the 
right-hand side of the system \eqref{syst:discreteReg}, and let $G(Y^n)$ 
represent the collection of all corresponding noise coefficients. 
Consequently, the system under consideration 
can be succinctly expressed as the SDE system 
\begin{equation}\label{SDEsystem}
	dY^n(t)=F(Y^n(t))\, dt+G(Y^n(t))\, dW^n(t),
	\quad Y^n(0)=Y^n_0,
\end{equation}
where $Y^n_0$ denotes the vector of 
initial conditions as specified by \eqref{syst:ICreg}.

If the vector-functions $F$ and $G$ satisfy global Lipschitz 
continuity conditions, classical theory on SDEs 
guarantees the existence and uniqueness of a strong solution. 
However, the inherent nonlinearity of the ionic models, 
as detailed in (\textbf{A.4}), implies that these global 
Lipschitz conditions are not met for the SDE system 
\eqref{SDEsystem}. 

Nevertheless, the weak monotonicity framework of 
\cite[Theorem 3.1.1]{Prevot:2007aa} is applicable. 
In fact, the verification of the following two conditions 
will imply the existence of a 
global strong solution to \eqref{SDEsystem}:
\begin{itemize}
	\item[$\bullet$] (local weak monotonicity) 
	$\forall Y_1,Y_2\in \R^{6n}$ and 
	$\forall r>0$ with $\abs{Y_1},\abs{Y_2}\leq r$,
	$$
	2\bigl(F(Y_1)-F(Y_2)\bigl)\cdot\left(Y_1-Y_2\right)
	+\abs{G(Y_1)-G(Y_2)}^2\leq K_r \abs{C_1-C_2}^2,
	$$
	for some $r$-dependent positive constant $K_r$.
	\medskip

	\item[$\bullet$] (weak coercivity) $\exists$ 
	constant $K>0$ such that $\forall Y\in \R^{6n}$,
	$$
	2F(Y)\cdot Y+ \abs{G(Y)}^2\leq K \bigl(1+\abs{C}^2\bigr).
	$$
\end{itemize}
Considering the assumptions {\bf (A.1)}--{\bf (A.4)}, 
especially with the utilization of \eqref{ineq:dissipative}, the 
demonstration of both the local weak monotonicity and the 
weak coercivity conditions follows the approach 
detailed in \cite[p.~5328]{Bendahmane:2018aa}. 
We will not repeat the details here.

Since these conditions are satisfied, for any 
$\cF_0$-measurable mapping $Y^n_0: \Omega \to \R^{6n}$, 
there exists a unique solution to the SDE system 
\eqref{SDEsystem}, unique up to $P$-indistinguishability. 
A solution $Y$ is defined here as the collection 
$Y=\Set{Y(t)}_{t\in [0,T]}$, which forms a $P$-almost surely continuous, 
$\R^{6n}$-valued, ${\cF_t}$-adapted process. 
For every $t \in [0,T]$, $P$-almost surely,
$$
Y^n(t) = Y^n_0 + \int_0^t F(Y^n(s)) \,ds 
+\int_0^t G(Y^n(s)) \, dW^n(s).
$$
Furthermore, for all $t \in [0,T]$,
$$
\E \abs{Y^n(t)}^2  \leq 
C_T\bigl( \E \abs{Y^n_0}^2 + 1\bigr),
$$
for a $T$-dependent constant $C_T$. We refer 
to \cite[Theorem 3.1.1]{Prevot:2007aa} for further details.
\end{proof}

Let us consider a basis $\Set{e_l}_{l=1}^\infty$ selected 
from those previously introduced in \eqref{eq:Galerkin-spaces}. 
We then introduce the projection operator, as detailed, 
for example, in \cite[page 1636]{Brzezniak:2013aa}:
$\Pi_n:(H^1)^*\to \Span\Set{e_l}_{l=1}^{\infty}$, 
$\Pi_n u^* := \sum_{l=1}^n 
\left \langle u^*,e_l \right
\rangle_{(H^1)^*,H^1}e_l$. 
The restriction of $\Pi_n$ to $L^2$ is also denoted by $\Pi_n$, 
in which case $\Pi_n$ is the orthogonal projection from $L^2$
to  $\Span\Set{e_l}_{l=1}^{\infty}$. The following properties hold:
\begin{equation*}
	\begin{split}
		& \norm{\Pi_n u}_{L^2}\le \norm{u}_{L^2},
		\quad 
		\left(\Pi_n u^*, u\right)_{L^2}
		=\left \langle u^*, \Pi_n u\right\rangle_{(H^1)^*,H^1},
		\\ & 
		\text{and} \quad 
		\norm{\Pi_n u -u}_{H^1(\cO)}\ton 0. 
	\end{split}
\end{equation*}

We have three projection operators, $\Pi_n^{\mathbf{H}}$, 
$\Pi_n^{\Bbb{L}}$ and $\Pi_n^{\Bbb{W}}$, which correspond to the 
finite-dimensional spaces $\mathbf{H}_n,\Bbb{L}_n,\Bbb{W}_n$ defined 
in \eqref{eq:Galerkin-spaces}. In order to streamline our notation, we 
will often omit the specific dependency of $\Pi_n$ on its associated space. 
For instance, we might write $\Pi_n$ in lieu of $\Pi_n^{\mathbf{H}}$. 
The specific space in question should then be inferred 
from the context and the variable being projected.

Utilizing the projection operators $\Pi_n$, we can express 
the Faedo-Galerkin equations from \eqref{syst:weakRegElast} in an 
integrated form. This representation allows us to 
present these equations as equalities between random 
variables in the dual space of $H^1$:
\begin{equation}\label{eq:approx-eqn-integrated}
	\begin{split}
		&\eps \bu^n(t)
		=\int_0^t \Pi_n\left[\Div \bigl( \nabla\bu^n 
		\sigma(\bx,\gamma^n)\bigr)+\Grad p^n\right]
		+\Pi_n[\ff]\ds
		\quad \text{in $\bigl[(H^1)^*\bigr]^3$},
		\\ 
		& \eps p^n(t)=\eps p^n_0 -\int_0^t\Pi_n\left[\Div\bu^n\right] \ds
		\quad \text{in $(H^1)^*$}, 
		\\ 
		& v^n(t) + \eps v^n_i(t) = v^n_0 + \eps v_{i,0}^n
		+\int_0^t  \Pi_n\left[\Iap\right]\ds
		\\ 
		& \qquad\qquad
		+ \int_0^t \Pi_n\left[\Div\bigl(\bM_i(\bx,\nabla\bu^n)
		\Grad v^n_i \bigr) 
		-\Ion(v^n,w^n) \right]\ds 
		\\ 
		& \qquad\qquad
		+ \int_0^t \Pi_n\left[ \beta_v(v^n)\right] \dW^{v,n}(s)
		\quad \text{in $(H^1)^*$},
		\\ 
		& v^n(t)-\eps v^n_e(t)=v^n_0-\eps v_{e,0}^n
		-\int_0^t  \Pi_n\left[ \Iap\right]\ds
		\\ & \qquad \qquad
		+\int_0^t \Pi_n\left[-\Div\bigl(\bM_e(\bx,\nabla\bu^n) 
		\Grad v^n_e \bigr)
		-\Ion(v^n,w^n) \right]\ds 
		\\ 
		& \qquad\qquad
		+ \int_0^t  \Pi_n\left[\beta_v(v^n)\right] \dW^{v,n}(s)
		\quad \text{in $(H^1)^*$},
		\\ & w^n(t) = w^n_0 
		+ \int_0^t \Pi_n\left[H(v^n,w^n)\right] \ds
		\\ 
		& \qquad\qquad 
		+ \int_0^t \Pi_n\left[\beta_w(v^n)\right]\dW^{w,n}(s)
		\quad \text{in $(H^1)^*$},
		\\ 
		& \gamma^n(t) = \gamma^n_0 
		+ \int_0^t \Pi_n \left[G(\gamma^n,w^n)\right] \ds
		\quad \text{in $(H^1)^*$},
	\end{split}
\end{equation}
where $v^n=v^n_i-v^n_e$ 
a.e.~in $\Omega\times [0,T]\times \cO$. 

The notation $\Pi_n[\ff]$ is a concise representation for 
$\Pi_n\bigl[\ff(t,\bx,\gamma^n)\bigr]$,  with 
$\ff$ defined in \eqref{def:ff}. In the same vein, 
$z_0^n$ stands for $\Pi_n z_0$, where $z$ encompasses 
the variables $p$, $v_i$, $v_e$, $w_0$, and $\gamma_0$. 
Moreover, the term $v^n_0$ is defined as the 
difference between $v_{i,0}^n$ and $v_{e,0}^n$.

\section{A priori estimates}\label{sec:apriori}
We will prove convergence of the Faedo-Galerkin 
approximations using the stochastic 
compactness method, which asks for the 
uniform tightness of the probability laws 
linked to these approximations. Tightness will hinge 
upon a series of a priori estimates, which hold uniformly in $n$. 
In our first set of estimates, we 
harness the structure of the equations \eqref{eq:nondegen-S1}, 
along with stochastic calculus, to deduce several basic energy-type estimates. 
Then we utilize these estimates and the governing equations
to derive temporal translation estimates in the space of $L^2_tL^2_x$. 
When brought together, these estimate will facilitate the 
tightness of the laws of the Faedo-Galerkin approximations.

\medskip
The non-degenerate system \eqref{eq:nondegen-S1} consists 
of a Stokes-like component for $(\bu_\eps,p_\eps)$, 
a pair of SPDEs for $(v_{i,\eps}, v_{e,\eps})$ (and 
$v_\eps=v_{i,\eps}-v_{e,\eps}$), a SDE for 
$w_\eps$, and an ODE for $\gamma_\eps$. 
The derivation of energy balance equations for 
$\eps\norm{\bu_\eps(t)}_{L^2}^2$, $\eps\norm{p_\eps(t)}_{L^2}^2$, 
and $\norm{\gamma_\eps(t)}_{L^2}^2$ can be accomplished by 
standard deterministic arguments. Similarly, by stochastic calculus, 
\begin{equation}\label{eq:Itovpmu2}
	d\abs{w_\eps}^2=2 H(v_\eps,w_\eps)
	+2 w_\eps\, \beta_w(w_\eps) \dW^w,
\end{equation}
from which we get the balance equation 
for $\norm{w_\eps(t)}_{L^2}^2$.

It is less straightforward 
to deduce the stochastic balance relation for 
$$
\norm{v_\eps(t)}_{L^2}^2
+\eps\norm{v_{i,\eps}(t)}_{L^2}^2
+\eps\norm{v_{e,\eps}(t)}_{L^2}^2.
$$
Following \cite{Bendahmane:2018aa}, we will present 
a formal argument that will lead to the energy balance for these variables. 
From \eqref{eq:nondegen-S1}, we deduce that
\begin{align*}
	dv_{i,\eps}  
	& =\Biggl[ \frac{1+\eps}{\eps(2+\eps)}
	\Div\bigl(\bM_i(\bx,\nabla\bu_\eps) \Grad v_{i,\eps} \bigr)
	+\frac{1}{\eps(2+\eps)}\Div\bigl(\bM_e(\bx,\nabla\bu_\eps) 
	\Grad v_{e,\eps} \bigr)
	\\ 
	& \qquad  \qquad
	-\frac{1}{2+\eps} \Ion(v_\eps,w_\eps)
	+\frac{1}{2+\eps}\Iap \Biggl]\dt
	+ \frac{1}{2+\eps}\beta_v(v_\eps) \dW^v,
\end{align*}
and
\begin{align*}
	dv_{e,\eps}& =\Biggl[ \frac{1}{\eps(2+\eps)}
	\Div\bigl(\bM_i(\bx,\nabla\bu_\eps) \Grad v_{i,\eps} \bigr)
	+\frac{1+\eps}{\eps(2+\eps)}\Div\bigl(\bM_e(\bx,\nabla\bu_\eps)
	\Grad v_{e,\eps} \bigr)
	\\ 
	& \qquad \qquad 
	+\frac{1}{2+\eps} \Ion(v_\eps,w_\eps)
	- \frac{1}{2+\eps} \Iap \Biggl]\dt
	-\frac{1}{2+\eps}\beta_v(v_\eps) \dW^v.
\end{align*}
By considering the aforementioned pair of SPDEs along 
with the third and fourth SPDEs stated in \eqref{eq:nondegen-S1}, 
we can employ the It\^{o} product rule to deduce the following result:
\begin{equation}\label{eq:Itovpmu1}
	\begin{split}
	&d\left(\abs{v_\eps}^2 +\eps \abs{v_{i,\eps}}^2 
	+\eps \abs{v_{e,\eps}}^2 \right)
	= d\left(v_{i,\eps} \left(v_\eps+\eps v_{i,\eps}\right) \right)
	+d\left(-v_{e,\eps} \left(v_\eps-\eps v_{e,\eps}\right) \right)
	\\ & \qquad 
	=  \Bigl [ 2 v_i \, \Div\bigl(\bM_i(\bx,\nabla\bu_\eps) \Grad v_{i,\eps}\bigr)
	+2 v_{e,\eps}\, \Div\bigl(\bM_e(\bx,\nabla\bu_\eps) 
	\Grad v_{e,\eps} \bigr) 
	\\ & \qquad \qquad
	-2 v \Ion(v_\eps,w_\eps)+2 v\Iap
	+\frac{2}{2+\eps}\abs{\beta_v(v_\eps)}^2 \Bigr]\dt 
	+2 v_\eps\, \beta_v(v_\eps) \dW^v,
	\end{split}
\end{equation}
which is the sought-after energy balance.

Following this, we will now transition into providing rigourous estimates 
for all the variables at the discrete level for the 
Faedo-Galerkin approximations. 

\begin{lem}\label{lem:apriori-est}
Assuming that conditions {\bf (A.1)} to {\bf (A.6)} 
and \eqref{eq:vie-init-ass-q0} are satisfied, consider 
the following variables:
$$
\bu^n(t),p^n(t),v^n_i(t),v^n_e(t),
v^n(t),w^n(t),\gamma^n(t), \quad t\in [0,T],
$$
which fulfill the Faedo--Galerkin equations given by 
\eqref{syst:weakRegElast}, with initial data from \eqref{syst:ICreg}. 
There exists a constant $C>0$, independent of $n$ 
and $\eps$, such that
\begin{align}
	\label{Gal:est3}
	& \E \left[\norm{v^n(t)}_{L^\infty(0,T;L^2(\cO))}^2\right]
	+\E \left[\norm{w^n(t)}_{L^\infty(0,T;L^2(\cO))}^2\right]
	\\ \notag
	& \qquad 
	+\sum_{j=i,e}
	\E \left[ \norm{\sqrt{\eps}v^n_j(t)}_{L^\infty(0,T;L^2(\cO))}^2 \right]\le C; 
	\\  \label{Gal:est2-bis}
	& \sum_{j=i,e} \E \left[\int_0^T \int_{\cO}
	\abs{\Grad v^n_j}^2 \dx \dt \right]
	+\E \left[\int_0^T \int_{\cO} 
	\abs{v^n}^4 \dx\dt \right] \le C;
	\\
	\label{Gal:est2-new}
	& \sum_{j=i,e} 
	\E \left[ \int_0^T \int_{\cO} \abs{v^n_j}^2 \dx \dt \right]\le C.
\end{align}
Moreover,
\begin{align}
	& \label{Gal:est1-mechanic-gamma}
	\E \left[ \norm{\gamma^n}_{L^\infty(0,T;L^2(\cO))}^2
	\right]\leq C,
	\\ & \label{Gal:est1-mechanic-u-p} 
	\E \left[
	\norm{\sqrt{\eps}\bu^n}_{L^\infty(0,T;L^2(\cO))}^2
	\right]\leq C, 
	\quad 
	\E \left[\norm{\sqrt{\eps}p^n}_{L^\infty(0,T;L^2(\cO))}^2
	\right]\leq C.
\end{align}
Finally, the following gradient estimates hold: 
\begin{align}
	\label{Gal:est2-mechanic}
	& \E\left[\norm{\bu^n}^2_{L^2(0,T;H^1(\cO)^3)}\right] \leq C
	\\ & \label{Gal:est2-mechanic-new}
	\E \left[\norm{\nabla w^n}_{L^\infty(0,T;L^2(\cO))}^2\right]
	\leq C,
	\quad
	\E \left[\norm{\nabla \gamma^n}^2_{L^\infty(0,T;L^2(\cO))}\right]
	\leq C.
\end{align}	
\end{lem}

\begin{proof}
The formal relationship previously indicated 
in equation \eqref{eq:Itovpmu1} 
can be rigorously derived for the Faedo-Galerkin approximations. 
This derivation is closely related to the details found in 
\cite{Bendahmane:2018aa}. Consequently, this implies that:
\begin{align}\label{eq:v-product}
	d & \int_{\cO} \abs{v^n}^2
	+\eps \abs{v^n_i}^2
	+\eps \abs{v^n_e}^2 \dx
	\\ 
	\notag
	& \quad
	= \Biggl [ 
	- 2\int_{\cO}  \bM_i(\bx,\nabla\bu^n)
	\Grad v^n_i \cdot \Grad v^n_i \dx 
	- 2\int_{\cO} \bM_e(\bx,\nabla\bu^n) 
	\Grad v^n_e\cdot \Grad v^n_e \dx
	\\
	\notag
	& \quad \quad\quad
	- 2\int_{\cO} v^n \Ion(v^n,w^n) \dx 
	+2 \int_{\cO}\Iap^n\, v^n \dx
	\\ \notag
	& \quad \quad \quad\quad  
	+ \frac{2}{2+\eps}\sum_{k=1}^n 
	\int_{\cO} \abs{\beta_{v,k}(v^n)}^2  \dx\Biggr] \dt
	+ 2\int_{\cO} v^n \beta_v(v^n) \dx \dW^{v,n}(t),
\end{align}
where $\Iap^n:=\Pi_n \left[ \Iap \right]$. 
In a similar manner, since \eqref{eq:Itovpmu2} 
can also be rigorously derived for the Faedo-Galerkin 
approximations, we obtain
\begin{equation}\label{eq:w-ito}
	\begin{split}
		d\int_{\cO} \abs{w^n}^2 \dx
		& = \left[\,  2\int_{\cO} w^n H(v^n,w^n) \dx 
		+2 \sum_{k=1}^n\int_{\cO}
		\abs{\beta_{w,k}(v^n)}^2\dx \right]\dt
		\\ & \qquad 
		+2\int_{\cO} w^n \beta_w(v^n) \dW^{w,n}.
      \end{split}
\end{equation}
By adding \eqref{eq:v-product} and \eqref{eq:w-ito}, performing
a time integration, and making use of the 
ellipticity assumption {\bf (A.2)}, we reach the following result:
\begin{equation}\label{eq:v+w-tmp1}
	\begin{split}
		& \frac12\norm{v^n(t)}_{L^2(\cO)}^2
		+\sum_{j=i,e}\frac12\norm{\sqrt{\eps}v^n_j(t)}_{L^2(\cO)}^2
		+\frac12\norm{w^n(t)}_{L^2(\cO)}^2
		\\ & \quad \qquad
		+ \frac 1c  \sum_{j=i,e} \int_0^t \int_{\cO}
		\abs{\Grad v^n_j}^2 \dx \ds 
		\\ & \quad
		= \frac12\norm{v^n(0)}_{L^2(\cO)}^2
		+\sum_{j=i,e}\frac12\norm{\sqrt{\eps}v_{j}^n(0)}_{L^2(\cO)}^2
		+\frac12\norm{w^n(0)}_{L^2(\cO)}^2
		\\ & \quad\qquad
		+\int_0^t\int_{\cO} \bigl( w^n H(v^n,w^n) 
		- v^n \Ion(v^n,w^n)+\Iap^n\,v^n\bigr) \dx \ds 
		\\ & \quad \qquad
		+\frac{1}{2+\eps}\sum_{k=1}^n\int_0^t\int_{\cO}
		\abs{\beta^n_{v,k}(v^n)}^2\dx \ds
		+\sum_{k=1}^n\int_0^t\int_{\cO}
		\abs{\beta_{w,k}(v^n)}^2\dx \ds\\
		&  \quad \qquad
		+\int_0^t\int_{\cO} v^n \beta_v(v^n) \dx \dW^{v,n}(s)
		+\int_0^t\int_{\cO} w^n \beta_w(v^n) \dx \dW^{w,n}(s).
      \end{split}
\end{equation}
By utilizing {\bf (A.4)} and repeatedly applying Cauchy's inequality,
\begin{equation}\label{eq:nonlinear-est}
	w H(v,w) - vI(v,w) \leq -C_1 \abs{v}^4 
	+C_2\left(\abs{v}^2 + \abs{w}^2\right) + C_3,
\end{equation}
for some constants $C_1>0$ and $C_2,C_3\ge 0$. 
Combining \eqref{eq:nonlinear-est} with the 
assumptions {\bf (A.4)} and {\bf (A.6)} 
in \eqref{eq:v+w-tmp1}, we obtain that
\begin{equation}\label{eq:v+w-tmp2}
	\begin{split}
		& \norm{v^n(t)}_{L^2(\cO)}^2
		+\sum_{j=i,e}\norm{\sqrt{\eps}v^n_j(t)}_{L^2(\cO)}^2
		+\norm{w^n(t)}_{L^2(\cO)}^2
		\\ 
		& \quad\qquad
		+ \sum_{j=i,e} \int_0^t \int_{\cO}
		\abs{\Grad v^n_j}^2 \dx \ds+\int_0^t \int_{\cO}
		\abs{ v^n}^4 \dx \ds
		\\  
		& \quad \leq \norm{v^n(0)}_{L^2(\cO)}^2
		+\sum_{j=i,e}\norm{\sqrt{\eps}v_j^n(0)}_{L^2(\cO)}^2
		+\norm{w^n(0)}_{L^2(\cO)}^2
		\\ 
		& \quad\qquad
		+K_1(T,\Omega) 
		+K_2\int_0^t \norm{v^n(s)}_{L^2(\cO)}^2 \ds 
		+ K_3\int_0^t \norm{w^n(s)}_{L^2(\cO)}^2 \ds
		\\ 
		& \quad\qquad
		+2\int_0^t\int_{\cO} v^n \beta_v(v^n) \dx \dW^{v,n}(s)
		+2\int_0^t\int_{\cO} w^n \beta_w(v^n) \dx \dW^{w,n}(s).
	\end{split}
\end{equation}

Note that the martingale property of stochastic integrals 
guarantees that the expected value of each of the last 
two terms in \eqref{eq:v+w-tmp2} is zero. Consequently, by taking 
the expectation in \eqref{eq:v+w-tmp2} and 
applying Gr{\"o}nwall's inequality, 
\begin{align}
	\label{Gal:est1}
	& \E \left[ \norm{v^n(t)}_{L^2(\cO)}^2\right]
	+\E \left[ \norm{w^n(t)}_{L^2(\cO)}^2\right]
	\\ \notag
	& \qquad 
	+\sum_{j=i,e}\E \left[ \norm{\sqrt{\eps}
	v^n_j(t)}_{L^2(\cO)}^2 \right]
	\le C, \quad \forall t\in [0,T].
\end{align}
and that the sought-after estimates \eqref{Gal:est2-bis} hold.

The following equation holds pointwise in 
$t$, for all test functions $\xi^n\in \Bbb{W}_n$:
$$
\int_{\cO}\pt \gamma^n\, \xi^n \dx 
=\int_{\cO}G(\gamma^n,w^n) \, \xi^n \dx.
$$
By utilizing this equation in conjunction with {\bf (A.3)}, we 
can infer that
\begin{equation}\label{est2-mechanic}
	\frac{\eps}{2}\dfrac{d}{dt}\int_{\cO}|\gamma^n|^2\dx 
	=\int_{\cO}G(\gamma^n,w^n)\,\gamma^n \dx
	\leq \bar{K}_3 \int_{\cO} \abs{\gamma^n}^2
	+\abs{w^n}^2\dx,
\end{equation}
for some constant $\bar{K}_3$ independent of $n,\eps$. 
We apply Gr{\"o}nwall's inequality in \eqref{est2-mechanic}, take the 
supremum over $t\in [0,T]$ and subsequently take the expectation. 
We combine this with the $L^2$ estimate \eqref{Gal:est1} for $w^n$, 
which allows to conclude \eqref{Gal:est1-mechanic-gamma}.

Let $E$ denote the left-hand side of \eqref{Gal:est3}.
To establish the estimate \eqref{Gal:est3}, we proceed 
by taking the supremum over the interval $[0, T]$ in \eqref{eq:v+w-tmp2} 
and then the expectation $\E[\cdot]$. Utilizing 
\eqref{Gal:est1}, the Burkholder-Davis-Gundy inequality, the 
Cauchy-Schwarz inequality, the assumption \eqref{eq:noise-cond2} 
regarding $\beta_v$, Cauchy's product 
inequality, and once again \eqref{Gal:est1}, we obtain 
\begin{equation}\label{Esup1}
	\begin{split}
		E  & \leq K_4\left(1 + \sum_{z=v,w}
		\E\left[\, \sup_{t\in [0,T]} \abs{\int_0^t\int_{\cO} 
		z^n \beta_z^n(z^n) \dx \dW^{z,n}(s)}\, \right]\right)
		\\ &   \leq K_5 \left(1+\sum_{z=v,w}\E \left[ \left(\int_0^T 
		\sum_{k=1}^n \abs{\int_{\cO} z^n 
		\beta^n_{z,k}(z^n) \dx}^2 \dt \right)^{\frac12}\right]\right)
		\\ &  \le  K_6 +\sum_{z=v,w} 
		\delta \E \left[\sup_{t\in [0,T]}\int_{\cO} \abs{z^n}^2\dx\right]
		\\ & \qquad\qquad
		+K_7\sum_{z=v,w} 
		\E\left[ \int_0^T\int_{\cO} \abs{z^n}^2 \dx\dt
		+T \abs{\cO}\right]
		\\ &  \leq K_8 +  \delta\sum_{z=v,w} \E\left[ \sup_{t\in [0,T]} 
		\norm{z^n(t)}_{L^2(\cO)}^2\right]
		\leq K_8 + \delta E,
	\end{split}
\end{equation}
for any (small) $\delta>0$, where the constants $K_4,\ldots,K_8$ 
are independent of $n,\eps$. This establishes \eqref{Gal:est3}.

Next, we shall establish the validity of \eqref{Gal:est2-new}. 
Using \eqref{compat1} and applying the Poincar\'e inequality, we 
can find a constant $\tilde C > 0$ that depends on $\cO$ but is 
independent of $n$, $\eps$, $\omega$, and $t$, such that 
the following inequality holds for 
$(\omega, t) \in \Omega \times (0, T)$:
$$
\norm{v^n_e(\omega,t,\cdot)}_{L^2(\cO)}^2
\le \tilde C \norm{\nabla v^n_e(\omega,t,\cdot)}_{L^2(\Om)}^2.
$$
Consequently, utilizing \eqref{Gal:est2-bis}, we deduce that
\begin{equation}\label{Gal:est2-new-tmp}
	\E \left[ \int_0^T 
	\norm{v^n_e(\omega,t,\cdot)}_{L^2(\cO)}^2\dt\right]
	\lesssim 1.
\end{equation}
Since $v^n$ complies with \eqref{Gal:est1}, it follows 
that also $v^n_i$ satisfies \eqref{Gal:est2-new-tmp}. 
Thus, \eqref{Gal:est2-new} holds.

Next, let us establish the gradient 
estimates \eqref{Gal:est2-mechanic-new} for $\gamma^n$ and $w^n$.  
We can make the assumption, without loss of generality, that the 
basis functions of $\Bbb{W}_n$, see \eqref{eq:Galerkin-spaces}, 
are $C^2$ regular and have a zero normal derivative 
on the boundary $\partial \cO$. This implies that 
$\nabla \gamma^n\cdot \bn$ is equal to zero on $\partial \cO$. 
Consider that $\gamma^n$ resides within the 
finite-dimensional space $\mathbb{W}_n$, which is 
spanned by the eigenfunctions $\Set{\zeta_l}$ of the 
Neumann-Laplace operator, with eigenvalues $\Set{\lambda_l}$. 
We can compute 
the Laplacian of $\gamma^n$: 
$\Delta \gamma^n = \sum_{l=1}^n \bigl(\gamma,\zeta_l\bigr) 
\Delta\zeta_l=-\sum_{l=1}^n 
\lambda_l \bigr(\gamma,\zeta_l\bigr) \zeta_l$. 
Hence, $\Delta \gamma^n \in \mathbb{W}_n$. 
So, by considering $\xi^n:=\Delta \gamma^n \in \Bbb{W}_n$ ($t$ is treated 
as a parameter) in the equation \eqref{syst:weakRegTimeElast} 
for the activation variable $\gamma^n$, 
we can perform spatial integration by parts and 
apply the temporal chain rule. This allows us 
to derive the following equation:
\begin{equation}\label{eq:gamma-gradient}
	\begin{split}
		&\frac{d}{dt}\int_{\cO}\abs{\nabla \gamma^n}^2 \dx 
		=\int_{\cO}\nabla G(\gamma^n,w^n) 
		\cdot \nabla \gamma^n \dx
		\\ &\qquad 
		\overset{{\bf (A.3)}}{=}
		\int_{\cO}\eta_1 \beta \nabla w^n
		\cdot \nabla \gamma^n
		-\eta_1\eta_2 \abs{\nabla \gamma^n}^2 \dx
		\leq  K_9\int_{\cO} \abs{\nabla w^n}^2 \dx,
	\end{split}
\end{equation}
by an appropriate use of Young's product inequality. 
Thus, to establish the proof of \eqref{Gal:est2-mechanic-new} 
for both $\gamma^n$ and $w^n$, we only need 
to establish a gradient bound specifically for $w^n$. 

Given the SDE \eqref{syst:weakRegElast} 
for $w^n$, we deduce that
\begin{equation*}
	\begin{split}
		d\int_{\cO} w^n \xi^n\dx
		=\int_{\cO} 
		H(v^n,w^n)\xi^n\dx\dt
		+\sum_{k=1}^n\int_\cO
		\beta_{w,k}(v^n)\xi^n\dx \dW_k^{w}(t).
	\end{split}
\end{equation*}
for all test functions $\xi^n\in \Bbb{W}_n$. Following a similar line 
of reasoning as presented above for $\nabla \gamma^n$, we 
can now proceed to derive
\begin{align*}
	d\int_{\cO}\abs{\nabla w^n}^2 \dx 
	&= \int_{\cO} 
	\nabla H(v^n,w^n)
	\cdot \nabla w^n\dx\dt
	\\ & \qquad
	+\sum_{k=1}^n\int_\cO
	\nabla \beta_{w,k}(v^n) \cdot 
 	\nabla w^n \dx \dW_k^{w}(t).
\end{align*}
Recalling {\bf (A.4)}, 
$\abs{\partial_v H(v,w)}=\abs{h'(v)}\lesssim 1$ and 
$\abs{\partial_w H(v,w)}=c_{H,1}$. 
Using this, it is not difficult to deduce that
\begin{align*}
	d\int_{\cO}\abs{\nabla w^n}^2 \dx 
	& \leq K_{10}\left(\int_{\cO} 
	\abs{\nabla v^n}^2\dx
	+\int_{\cO} \abs{\nabla w^n}^2 \dx
	\right)\dt
	\\ & \qquad
	+\sum_{k=1}^n\int_\cO
	\beta_{w,k}'(v^n) \nabla v^n
	\cdot \nabla w^n \dx \dW_k^{w}(t),
\end{align*}
where the first term on the right-hand side 
can be controlled by \eqref{Gal:est2-bis} (keeping in mind 
that $v^n=v^n_i-v^n_e$). We can now proceed with 
a similar approach as employed 
in establishing the $L^2$ estimates \eqref{Gal:est3}. This involves 
utilizing the BDG and Gronwall inequalities. Let us 
here only discuss the application of the BDG inequality, 
following \eqref{Esup1}. By \eqref{eq:noise-cond}, we can observe 
that $\sum_{k=1}^n\absb{\beta_{w,k}'(v)}^2\lesssim 1$. This, 
when combined with the BDG inequality, 
yields the following result:
\begin{align*}
	& \E \left [\, \sup_{t\in [0,T]} 
	\abs{\int_0^t\sum_{k=1}^n \int_\cO
	\beta_{w,k}'(v^n) \nabla v^n
	\cdot \nabla w^n \dx \dW_k^{w}(s)} \, \right]
	\\ & \qquad \leq \E \left[ \left(\int_0^T 
	\sum_{k=1}^n \abs{\int_{\cO} 
	\beta_{w,k}'(v^n) \nabla v^n
	\cdot \nabla w^n \dx}^2 \dt \right)^{\frac12}\right]
	\\ & \qquad \leq  
	\delta \E \left[\sup_{t\in [0,T]}
	\int_{\cO} \abs{\nabla w^n(t)}^2\dx\right]
	+K_{11} \E\left[ \int_0^T\int_{\cO} \abs{\nabla v^n}^2 \dx\dt\right]
	\\ & \qquad  \overset{\eqref{Gal:est2-bis}}{\leq}  
	\delta \E \left[\sup_{t\in [0,T]}
	\int_{\cO} \abs{\nabla w^n(t)}^2\dx\right]+K_{12}.
\end{align*}
By utilizing this inequality, we can proceed in a similar manner as before, 
see \eqref{Esup1}. This approach leads us to the first part 
of \eqref{Gal:est2-mechanic-new} for $\nabla w^n$, which, via 
\eqref{eq:gamma-gradient}, implies the 
second part for $\nabla \gamma^n$.

Finally, we proceed with the derivation of the 
estimates for the Stokes-like part of the system \eqref{syst:weakRegElast}. 
Consider the following system, which holds pointwise in 
time $t$, for all test functions $\boldsymbol{\psi}^{n}\in \mathbf{H}_n$ 
and $\rho^{n}\in \Bbb{W}_n$:
\begin{equation}\label{syst:weakRegTimeElast}
	\begin{split}
		&\eps\int_{\cO}\pt\bu^n \cdot \boldsymbol{\psi}^{n}\dx
		+ \int_{\cO} \nabla\bu^n\, 
		\bsigma(\bx,\gamma^n)
		:\nabla \boldsymbol{\psi}^n
		-p^n\Grad\cdot\boldsymbol{\psi}^n\dx
		\\ 
		& \qquad  \qquad 
		+\int_{\partial \cO}\alpha\bu^n\cdot 
		\boldsymbol{\psi}^n \, dS(\bx)
		=\int_{\cO}\ff^n\cdot \boldsymbol{\psi^n} \dx,
		\\
		& \eps\int_{\cO}\pt p^n{\rho}^n \dx
		+\int_{\cO}\rho^n\nabla\cdot\bu^n\dx=0,
\end{split}\end{equation}
where $\ff^n=\ff(\bx,\gamma^n)$, cf.~\eqref{def:ff}.

Now, we replace $\boldsymbol{\psi}^n$ with 
$\bu^n$ and $\rho^n$ with $p^n$ 
in \eqref{syst:weakRegTimeElast}. 
We then sum the resulting equations to yield
\begin{equation}\label{est1-mechanic}
	\begin{split}
		&\frac{\eps}{2}\dfrac{d}{dt}\int_{\cO}|\bu^n|^2\dx
		+\int_{\cO} \nabla\bu^n
		\, \bsigma(\bx,\gamma^n)
		:\nabla \bu^n \dx
		\\ 
		& \qquad
		+\frac{\eps}{2}\dfrac{d}{dt}\int_{\cO}|p^n|^2\dx 
		+\alpha\int_{\partial \cO} \abs{\bu^n}^2 \, dS(\bx) 
		=\int_{\cO}\ff^n\cdot \bu^n \dx.
	\end{split}
\end{equation}

Let us introduce the continuous bilinear form
\begin{equation}\label{bilinear-forma}
	a(\bu,\mathbf{v})=a(\bu,\mathbf{v};\gamma)=\int_{\cO}
	\nabla\bu\bsigma(\bx,\gamma):\nabla \mathbf{v} \dx
	+\int_{\partial \cO} \alpha\bu
	\cdot \mathbf{v} \, dS(\bx).
\end{equation}
We claim that $a(\cdot,\cdot)$ is coercive on $(H^1(\cO))^3$. 
First, by the uniform ellipticity of 
$\boldsymbol{\sigma}$, see assumption {\bf{(A.1)}},
\begin{align*}
	a(\bu,\bu) & =\int_{\cO} \nabla\bu
	\boldsymbol{\sigma}(\bx,\gamma):\nabla\bu \dx
	+\alpha\norm{\bu}_{L^2(\partial\cO)}^2
	\gtrsim \norm{\nabla\bu}^2
	+\alpha\norm{\bu}_{L^2(\partial\cO)}^2.
\end{align*}
Next, we want to show that there exists a constant 
$ \nu>0$ such that 
$$
\nu \left( \norm{\nabla{\bu}}^2_{(L^2(\cO))^{3\times3}}
+\alpha\norm{\bu}^2_{L^2(\partial\cO)}\right)
\geq \norm{\bu}_{(H^1(\cO))^3}, 
\quad \forall \bu\in (H^1(\cO))^3.
$$ 
We argue by contradiction. Suppose, for each $m\in \N$, 
$\exists\;\bu_m\in (H^1(\cO))^3$ such that
$$
\norm{\nabla\bu_m}_{(L^2(\cO))^{3\times3}}^2
+\alpha\norm{\bu_m}_{(L^2(\partial\cO))^3}^2
\leq \frac{1}{m} \norm{\bu}_{(H^1(\cO))^3}^2.
$$
Setting 
$\mathbf{v}_m=\dfrac{\bu_m}{\|\bu_m\|_{(H^1(\cO))^3}}$, we obtain  
\begin{equation}\label{eq:vm-prop}
	\norm{\mathbf{v}_m}_{(H^1(\cO))^3}=1, 
	\quad
	\norm{\nabla\mathbf{v}_m}^2_{(L^2(\cO))^{3\times3}}
	+\alpha\norm{\mathbf{v}_m}^2_{(L^2(\partial\cO))^3}\leq\frac{1}{m},
\end{equation}
which implies that
 \begin{equation}\label{proof:coercivity1-2}
 	 \textrm{(i)} \; \;  
	 \text{$\nabla\mathbf{v}_m \tom 0$ 
	 in $(L^2(\cO))^{3\times3}$},
	 \quad 
	 \textrm{(ii)} \; \; 
	 \text{$\mathbf{v}_m
	 \tom 0$ in $(L^2(\partial\cO))^3$}.
 \end{equation}
On the other hand, since $\mathbf{v}_m$ 
is bounded in $(H^1(\cO))^3$ and $\cO\subset \R^3$ is bounded and 
smooth, there exists $\mathbf{v}\in (H^1(\cO))^3$ 
and a subsequence $\Set{\mathbf{v}_{m_k}}\subset (H^1(\cO))^3$ 
such that 
$$
\text{$\mathbf{v}_{m_k}\tok \mathbf{v}$ in $(L^2(\cO))^3$}, 
\quad
\text{$\nabla\mathbf{v}_{m_k}\tok\nabla\mathbf{v}$ 
in $(\cDp(\cO))^3$}.
$$
Utilizing \eqref{proof:coercivity1-2}(i), we infer that $\nabla\mathbf{v} = 0$. 
Given that $\cO$ is connected, this leads us to conclude 
that $\mathbf{v} = C$. Reapplying \eqref{proof:coercivity1-2}(i) and 
considering the convergence of $\mathbf{v}_{m_k} \to C$ in 
$(L^2(\cO))^3$, it follows that $\mathbf{v}_{m_k} \to C$ 
in $(H^1(\cO))^3$. This, in view of the continuity of 
the trace map, suggests that
$$
\text{$\mathbf{v}_{m_k}\tok C$ in $(L^2(\partial\cO))^3$.}
$$
On the other hand, by \eqref{proof:coercivity1-2}(ii), 
$\mathbf{v}_{m_k}\to 0$ in $(L^2(\partial\cO))^3$, and so $C=0$. 
However, this is contradictory to the initial assertion 
given by the first part of \eqref{eq:vm-prop}.
This affirms the claimed coercivity of $a(\cdot,\cdot)$.

By the coercivity of $a$ and Young's product inequality, we obtain 
from \eqref{est1-mechanic} that
\begin{equation}\label{ineq:elast-m}
	\frac{1}{2}\dfrac{d}{dt}\left(\norm{\sqrt{\eps}\bu^n}^2_{L^2(\cO)^3}
	+\norm{\sqrt{\eps}p^n}^2_{L^2(\cO)}\right)
	+C \norm{\bu^n}^2_{H^1(\cO)^3}
	\leq C \norm{\ff^n}^2_{L^2(\cO)},
\end{equation}
for some constant $C>0$ independent of $n,\eps$. By 
\eqref{def:ff} and the assumption {\bf (A.1)} on $\boldsymbol{\sigma}$, 
we have $\norm{\ff^n}^2_{L^2((0,T)\times \cO)}
\lesssim 1+\norm{\nabla \gamma^n}_{L^2((0,T)\times\cO)}^2$. 
By integrating \eqref{ineq:elast-m} over the interval $(0,t)$, where 
$0<t\leq T$, and subsequently taking the supremum over 
$t$, we can then apply the expectation operator. 
Keeping in mind that $\bu^n(0)=0$, $\norm{p_{0}^n}_{L^2(\cO)}
\leq \norm{p_0}_{L^2(\cO)}$, and using the estimate 
\eqref{Gal:est2-mechanic-new}, we arrive at 
\eqref{Gal:est1-mechanic-u-p}. 

By integrating \eqref{ineq:elast-m}, we can also confirm 
that \eqref{Gal:est2-mechanic} is satisfied.
\end{proof}

The subsequent corollary, encompassing higher 
moments estimates, arises as a direct consequence 
of Lemma \ref{lem:apriori-est} and the BDG inequality. 

\begin{cor}\label{cor:Lq0-est}
Along with the assumptions stated in Lemma \ref{lem:apriori-est}, 
let us further assume that \eqref{eq:vie-init-ass-q0} holds, 
with $q_0$ defined as specified in \eqref{eq:moment-est}. 
Then there exists a constant $C > 0$, which remains 
independent of both $n$ and $\eps$, satisfying 
the following estimates:
\begin{equation}\label{eq:Lq0-est}
	\begin{split}
		&\E\left[ 
		\, \norm{v^n(t)}_{L^\infty(0,T;L^2(\cO))}^{q_0}
		\, \right] 
		+\sum_{j=i,e}\E\left[ \, 
		\norm{\sqrt{\eps}v^n_j(t)}_{L^\infty(0,T;L^2(\cO))}^{q_0}
		\, \right]
		\\ &\qquad 
		+\E\left[ \, 
		\norm{w^n(t)}_{L^\infty(0,T;L^2(\cO))}^{q_0}
		\, \right]
		+\E\left[ \, 
		\norm{\gamma^n(t)}_{L^\infty(0,T;L^2(\cO))}^{q_0}
		\, \right]
		\\ & \qquad 
		+\E\left[ \, 
		\norm{\sqrt{\eps}\bu^n(t)}_{L^\infty(0,T;L^2(\cO))}^{q_0}
		\, \right]
		+\E\left[ \, 
		\norm{\sqrt{\eps}p^n(t)}_{L^\infty(0,T;L^2(\cO))}^{q_0}
		\, \right]\le C.
	\end{split}
\end{equation}
Moreover,
\begin{equation}\label{eq:Lq0-est2}
	\sum_{j=i,e}
	\E\left[ \,
	\norm{\Grad v^n_j}_{L^2((0,T)\times \cO)}^{q_0}
	\right] 
	+\E\left[\, \norm{v^n}_{L^4((0,T)\times \cO)}^{2q_0} 
	\, \right]  \le C.
\end{equation}
\end{cor}

\begin{proof}
We exponentiate both sides of the inequality \eqref{eq:v+w-tmp2} by $q_0/2$, 
recalling that $q_0>9/2$. Subsequently, we employ 
a series of elementary inequalities to arrive at
\begin{equation}\label{eq:v+w-tmp2-new1}
	\begin{split}
		&\norm{v^n(t)}_{L^2(\cO)}^{q_0}
		+\sum_{j=i,e}\norm{\sqrt{\eps}v^n_j(t)}_{L^2(\cO)}^{q_0}
		+\norm{w^n(t)}_{L^2(\cO)}^{q_0}
		\\  
		& \quad \leq C_1\norm{v^n(0)}_{L^2(\cO)}^{q_0}
		+C_1\sum_{j=i,e}\norm{\sqrt{\eps}v_j^n(0)}_{L^2(\cO)}^{q_0}
		+C_1\norm{w^n(0)}_{L^2(\cO)}^{q_0}
		\\ 
		& \quad\qquad
		+C_1 +C_1\int_0^t \norm{v^n(s)}_{L^2(\cO)}^{q_0} \ds 
		+C_1\int_0^t \norm{w^n(s)}_{L^2(\cO)}^{q_0} \ds
		\\ 
		& \quad\qquad
		+C_1\abs{\int_0^{t}\int_{\cO} v^n 
		\beta_v(v^n) \dx \dW^{v,n}(s)}^{q_0/2}
		\\ 
		& \quad\qquad
		+C_1
		\abs{\int_0^{t}\int_{\cO} w^n \beta_w(v^n) 
		\dx \dW^{w,n}(s)}^{q_0/2},
	\end{split}
\end{equation}
for some constant $C_1$ that is independent of $n$ (but 
dependent on $T$). For any $R>0$, we define the stopping time
$$
\tau^n_R := \inf\bigl\{t \in [0, T] : \norm{Y^n(t)} > R\bigr\} 
\wedge T,
$$
where $\inf \emptyset := \infty$ and 
$$
\norm{Y^n(t)}=\|v^n(t)\|_{\Bbb{W}_n}
+\|v^n_i(t)\|_{\Bbb{W}_n}
+\|v^n_e(t)\|_{\Bbb{L}_n}
+\|w^n(t)\|_{\Bbb{W}_n}.
$$
Given Lemma \ref{lem:fg-solutions}, it is clear that
$$
\lim_{R \to \infty} \tau^n_R = 
T, \quad \text{$P$-a.s.}, \quad \text{for each fixed $n \in \N$}.
$$
Following \eqref{Esup1}, we apply the BDG 
inequality along with several elementary inequalities. This 
enables us to derive the following two bounds for any $\delta > 0$:
\begin{align*}
	& \E\left[\sup_{t'\in \left[0,\tau^n_R\wedge t\right]}
	\abs{\int_0^{t'}\int_{\cO} v^n 
	\beta_v(v^n) \dx \dW^{v,n}(s)}^{q_0/2}\right]
	\\ & \,  
	\le \delta \E\left[\sup_{t'\in \left[0,\tau^n_R\wedge t\right]}
	\norm{v^n(t')}_{L^2(\cO)}^{q_0}\right]
	+C_2\int_0^t\E\left[
	\sup_{t'\in \left[0,\tau^n_R\wedge s\right]}
	\norm{v^n(t')}_{L^2(\cO)}^{q_0}\right]\, ds+C_3,
	\\ & 
	\E\left[\sup_{t'\in \left[0,\tau^n_R\wedge t\right]}
	\abs{\int_0^{t'}\int_{\cO} w^n 
	\beta_w(v^n) \dx \dW^{w,n}(s)}^{q_0/2}\right]
	\\ & \,  
	\le \delta \E\left[\sup_{t'\in \left[0,\tau^n_R\wedge t\right]}
	\norm{w^n(t')}_{L^2(\cO)}^{q_0}\right]
	+C_2\int_0^t
	\E\left[\sup_{t'\in \left[0,\tau^n_R\wedge s\right]}
	\norm{v^n(t')}_{L^2(\cO)}^{q_0}\right]\, ds+C_3,
\end{align*}
for certain constants $C_2, C_3$ that are independent of $n$ 
(though they depend upon $T$). By combining this 
with \eqref{eq:v+w-tmp2-new1}, applying Gronwall's inequality, 
and subsequently taking the limit as $R \to \infty$ through the 
application of the monotone convergence theorem, we arrive 
at the initial three estimates presented in \eqref{eq:Lq0-est}. 

After establishing these estimates, we return 
to the inequality \eqref{eq:v+w-tmp2}, which also features the two 
terms $\sum_{j=i,e} \int_0^t \int_{\cO}
\abs{\Grad v^n_j}^2 \dx \ds$ and $\int_0^t \int_{\cO}\abs{ v^n}^4 \dx \ds$.  
We deduce \eqref{eq:Lq0-est2} by exponentiating this 
inequality  with $q_0/2$ and then 
applying the estimates we derived earlier.

The estimation \eqref{eq:Lq0-est} of $\gamma^n$ is derived 
in a similar manner from \eqref{est2-mechanic}. 
The latter involves no stochastic integral, rendering 
the process comparatively simpler.  Likewise, by 
employing the equations in \eqref{est1-mechanic}, which 
are also devoid of stochastic integrals, we can establish the 
prescribed bounds \eqref{eq:Lq0-est} 
for the velocity $\bu^n$ and pressure $p^n$.
\end{proof}

To ensure the tightness of the probability laws 
of the Faedo-Galerkin solutions, we rely on an Aubin-Lions-Simon 
compactness lemma. To apply this lemma effectively, we will require 
the following ($n,\eps$ uniform) temporal translation estimates.

\begin{lem}\label{Space-Time-translate}
Suppose conditions {\bf (A.1)}-{\bf (A.5)}, \eqref{eq:noise-cond} 
and \eqref{eq:vie-init-ass-q0} hold. Let 
$v^n_i(t)$, $v^n_e(t)$, $v^n(t)$, $\bu^n(t)$, $p^n(t)$, 
$w^n(t)$, $\gamma^n(t)$, $t\in [0,T]$, 
satisfy \eqref{syst:weakRegElast} , \eqref{syst:ICreg}. 
There exists a constant $C>0$, independent of 
$n$ and $\eps$, such that for any 
sufficiently small $\delta>0$,
\begin{equation}\label{Time-translate}
	\begin{split}
		&\E \left[\, \sup_{0\leq \tau\leq \delta} 
		\norm{z^n(\cdot+\tau,\cdot)
		-z^n}_{L^2(0,T-\tau;L^2(\cO))}^2
		\, \right] \leq C \delta^{\frac14},
	\end{split}
\end{equation}
for $z^n=v^n,\, w^n,\, \gamma^n$.
\end{lem}

\begin{proof}
We refer to Lemma 6.3 in \cite{Bendahmane:2018aa} 
for further details. The primary distinction from the 
study \cite{Bendahmane:2018aa} is that the 
operators $\Pi_n\left[\Div\bigl(\bM_j(\bx,\nabla\bu^n)
\Grad v^n_j \bigr)\right]$, where $j=i,e$, 
nonlinearly depend on the displacement 
gradient $\nabla \bu$. Nevertheless, given the $L^\infty$-boundedness 
of $\bM_{i,e}$, as stipulated in assumption {\bf (A.2)}, the 
reasoning proposed by \cite{Bendahmane:2018aa} remains 
applicable to $v^n$. The argument for the $w^n$ 
variable does not necessitate any alterations, and the 
added variable $\gamma^n$ can be addressed 
in a similar (simpler) manner to $w^n$.
\end{proof}

\section{Tightness and a.s.~representations}\label{sec:tightness}
The objective is to establish the existence of a  
solution, drawing inspiration from some techniques 
previously applied to the deterministic system \cite{Bendahmane:2006au}. 
Utilizing the Faedo-Galerkin approach, we have constructed approximate 
solutions originating from a nondegenerate auxiliary system 
and we wish to deduce their convergence by employing 
the stochastic compactness method. 
In contrast to the deterministic case, where the a priori estimates 
facilitate strong compactness across the time and space 
variables ($t,x$)—which is essential 
for addressing the nonlinear terms—the stochastic framework introduces 
the complexity of the probability variable $\omega$, rendering strong 
compactness unattainable. The standard resolution involves 
deducing the weak compactness of the probability laws 
of the Faedo-Galerkin solutions through tightness 
and Prokhorov’s theorem, followed by constructing  
almost surely convergent sequences via 
Skorokhod's representation theorem. 
This enables us to demonstrate that the limit variables 
form a weak martingale solution.  Due to the degenerate 
nature of our model, we observe only weak ($H^1$) convergence 
for the intra- and extracellular potentials. As a result, the probability 
measures of these potentials live on a Banach space 
equipped with the weak topology, which is 
notably non-Polish. To handle this issue, we utilize 
Jakubowski's version \cite{Jakubowski:1997aa} of the 
Skorokhod representation theorem, tailored for non-metric spaces.

Let us establish the required tightness as $n\to \infty$ (with 
$\eps$ kept fixed) of the probability laws generated by the 
Faedo--Galerkin solutions $\Set{\bigl(V^n,W^n,V_0^n\bigr)}_{n\in \N}$, where
\begin{equation}\label{eq:def-U-W-U0}
	\begin{split}
		& V^n=v^n_i,
		\, v^n_e,
		\, v^n,
		\, w^n,
		\, \gamma^n, 
		\,p^n, 
		\, \bu^n, 
		\\ & 
		W^n = W^{v,n},\, W^{w,n},
		\\ & 
		V_0^n =v_{i,0}^n,
		\, v_{e,0}^n, 
		\, v_0^n, 
		\, w_0^n,
		\, \gamma_0^n,
		\, p^n_0, 
		\, \bu_0^n \, (\equiv 0).
	\end{split}
\end{equation}
Let us introduce the path space
$$
\cX=\cX_V\times \cX_W\times \cX_{V_0},
$$
where
\begin{align*}
	& \cX_V =\Bigl[ \bigl(L^2_tH^1_x\bigr)_w\Bigr]^2\times 
	\Bigl [ L^2_{t,x}\Bigr]^3 
	\times  \bigl(L^2_tL^2_\bx\bigr)_w
	\times \bigl(L^2_t[H^1_\bx]^3\bigr)_w,
	\\ & \qquad\qquad
	\text{for $\left(v^n_i,v^n_e\right)$, 
	\, $\left(v^n,w^n,\gamma^n\right)$, 
	\, $p^n$, \, $\bu^n$},
	\\ & 
	\cX_W=\Bigl[C_t \U_0\Bigr]^2,
	\quad \text{for $\bigl(W^{v,n},W^{w,n}\bigr)$},
	\qquad 
	\cX_{V_0} =\Bigl[ L^2_\bx\Bigr]^7,
	\quad \text{for $V_0^n$}.
\end{align*}
The subscript ``$w$" signifies that the topology is weak.

Let $\cB(\cX)$ symbolize the $\sigma$-algebra of 
Borel subsets of $\cX$. We define the $\left(\cX,\cB(\cX)\right)$-valued 
measurable function $\Phi_n$ on the 
measure space $\left(\Omega,\cF,P\right)$ as follows:
$$
\Phi_n(\omega)
= \bigl(V^n(\omega),W^n, V_0^n(\omega)\bigr), 
\quad \omega\in \Omega.
$$
We utilize $\cL_n$ to represent the joint probability law 
of $\Phi_n$ on $\left(\cX,\cB(\cX)\right)$. 
Furthermore, the notation for the marginals of $\cL_n$ substitutes 
the subscript $n$ with the variable being considered. 
For example, $\cL_{v^n}$ stands for the law 
of the variable $v^n$ in the $L^2_{t,x}$ space. 
This convention applies uniformly across all other variables.

The objective is to affirm the tightness of the 
joint laws $\Set{\cL_n}$. This involves 
producing, for every $\kappa>0$, a compact 
set $\cK_\kappa$ within $\cX$ such that 
$\cL_n\left(\cK_\kappa\right)>1-\kappa$. 
This requirement is equivalent  to
\begin{equation}\label{eq:tight-def}
	\cL_n\left(\cK_\kappa^c\right)
	=P\bigl(\Set{\omega\in \Omega: \Phi_n(\omega)
	\in \cK_\kappa^c}\bigr)\leq \kappa, 
	\quad \text{uniformly in $n$}.
\end{equation}
By invoking Tychonoff’s theorem, this 
tightness is a direct consequence of the 
tightness of the product measures $\cL_{v^n_i} 
\otimes \cL_{v^n_e} \otimes \cdots 
\otimes \cL_{\bu_0^n}$ on $\cX$, under 
the product $\sigma$-algebra. Thus, to substantiate 
\eqref{eq:tight-def}, it suffices to identify a compact set 
$\cK_{v^n_i,\kappa}$ in $\bigl(L^2_tH^1_x\bigr)_w$ such 
that $\cL_{v^n_i}\bigl(\cK_{v^n_i,\kappa}^c\bigr)
\leq \kappa$, for each $\kappa>0$,  and 
analogously for the remaining variables 
$v^n_e,\ldots,\bu_0^n$. Indeed, we have the following lemma:

\begin{lem}\label{lem:tight}
Equipped with the estimates in Lemmas \ref{lem:apriori-est} 
and \ref{Space-Time-translate}, the sequence of laws 
$\Set{\cL_n}_{n\ge1}$ is tight on $\left(\cX,\cB(\cX)\right)$, 
in the sense that \eqref{eq:tight-def} holds.
\end{lem}

\begin{proof}
By weak compactness of bounded sets in 
$L^2_tH^1_\bx$,
$$
K_R=\Set{v: \norm{v}_{L^2_tH^1_\bx}\leq R}
$$
is compact in $\bigl(L^2_tH^1_\bx\bigr)_w$, for any $R>0$. 
By Chebyshev's inequality, 
\begin{align*}
	& P\left(\Set{\omega\in \Omega: v^n_i(\omega)\in K_R^c}\right)
	= P\left(\Set{\omega\in \Omega: 
	\norm{v^n_i(\omega)}_{L^2_tH^1_\bx}>R}\right)
	\\ & \qquad \qquad \leq \frac{1}{R^2}
	\E\left[\norm{v^n_i(\omega)}_{L^2_tH^1_\bx}^2\right]
	\overset{\eqref{Gal:est2-bis}, \eqref{Gal:est2-new}}{\lesssim} 
	\frac{1}{R^2}\toR 0.
\end{align*}
Hence, we can specify the required compact 
$\cK_{v^n_i,\kappa}\subset \bigl(L^2_tH^1_\bx\bigr)_w$ 
as $K_R$ for some suitable $R=R(\kappa)$, 
such that $\cL_{v^n_i}
\bigl(\cK_{v^n_i,\kappa}^c\bigr)\leq \kappa$. 
We can argue similarly (via weak compactness) 
for the variables $v^n_e$, $p^n$, and $\bu^n$.

Given that $W^{v,n}$ a.s.~converges in $C([0,T];\U_0)$ 
as $n\to \infty$, it follows that the laws $\Set{\cL_{W^{v,n}}}_{n\in \N}$ 
undergo weak convergence. This implies 
the tightness of $\Set{\cL_{W^{v,n}}}_{n\in \N}$. 
In other words, for any specified $\kappa>0$, we can find 
a compact set $\cK_{W^{v,n},\kappa}\subset C([0,T];\U_0)$ 
such that the measure of the complement of this set with 
respect to $\cL_{W^{v,n}}$ is $\leq \kappa$. 
This reasoning can be analogously applied to $W^{w,n}$.

Similar reasoning can be applied to the initial data variables
 $v_{i,0}^n$, $v_{e,0}^n$, $v_0^n$, $w_0^n$, $\gamma_0^n$, $p_0^n$, 
 $\bu_0^n$. By assumption, these variables are a.s.~convergent 
 in $L^2_\bx$ as $n\to \infty$. Consequently, their corresponding 
 laws also undergo weak convergence. This leads to the 
 conclusion that the laws of these variables are tight on $L^2_\bx$.

Our remaining task is to demonstrate the tightness of 
the laws of $v^n$, $w^n$, and 
$\gamma^n$ on $L^2_{t,\bx}$. This will be a direct 
consequence of the $L^2$ estimates given by 
\eqref{Gal:est3} and \eqref{Gal:est1-mechanic-gamma}, the $L^2$ 
gradient estimates specified in \eqref{Gal:est1-mechanic-u-p} 
and \eqref{Gal:est2-mechanic-new}, and the temporal 
continuity estimates provided by Lemma \ref{Space-Time-translate}, 
drawing upon the insights from \cite{Bensoussan:1995aa} 
and \cite{Bendahmane:2018aa}. Focusing on $v^n$ (although the 
variables $w^n$ and $\gamma^n$ follow suit in an 
analogous manner), consider any two sequences of 
positive numbers $r_m$ and $\nu_m$, which converge 
to zero as $m\to \infty$. 
We will proceed by defining the following set:
\begin{align*}
	\cZ_{(r_m,\nu_m)}
	:=\Biggl\{ & v \in L^\infty(0,T;L^2(\cO))
	\cap L^2(0,T;H^1(\cO)):
	\\ & \qquad
	\sup_{m\ge 1} \frac{1}{\nu_m}\sup_{0\le \tau\le r_m} 
	\norm{v(\cdot+\tau)-v}_{L^2\left(0,T-\tau;L^2(\cO)\right)}
	<\infty \Biggr\}.
\end{align*}
Note that $\cZ_{(r_m,\nu_m)}$ is compactly 
embedded in $L^2_{t,\bx}=L^2((0,T)\times \cO)$, which follows
from \cite[Theorem 3]{Simon:1987vn}.  Fix 
two sequences $r_m$, $\nu_m$ of positive numbers 
numbers tending to zero as $m\to \infty$ (independently of $n$), such that
\begin{equation}\label{eq:r-nu}
	\sum_{m=1}^\infty r_m^{\frac14}/\nu_m^2<\infty,
\end{equation}
and define the $L^2_{t,\bx}$-compact set
$$
K_R = \Set{v: \norm{v}_{\cZ^v_{(r_m,\nu_m)}}
\leq R},
$$
for a number $R>0$ that will be determined later. 
We have
\begin{align*}
	&P\bigl(\Set{\omega\in \Omega: v^n(\omega)\in K_R^c} \bigr)
	\\ & \quad
	\leq P\left(\Set{\omega\in \Omega: 
	\norm{v^n(\omega)}_{L^\infty\left(0,T;L^2(\cO)\right)}
	>R}\right)
	\\ & \quad\quad 
	+P\left(\Set{\omega\in \Omega: 
	\norm{v^n(\omega)}_{L^2\left(0,T;H^1(\cO)\right)}
	>R}\right)
	\\ & \quad\quad 
	+P\left(\Set{\omega\in \Omega: 
	\sup_{0\le \tau \le r_m} 
	\norm{v^n(\cdot+\tau)
	-v^n}_{L^2\left(0,T-\tau;L^2(\cO)\right)}>R \, \nu_m}\right)
	\\ & \quad
	=:P_1+P_2+P_3
	\quad \text{(for any $m\ge 1$)}.
\end{align*}
By Chebyshev's inequality, 
\begin{align*}
	& P_1 \leq \frac{1}{R^2}
	\E \left[\norm{v^n(\omega)}_{L^\infty\left(0,T;L^2(\cO)\right)}^2\right]
	\overset{\eqref{Gal:est1}}{\lesssim} \frac{1}{R^2}\toR 0, 
	\\ & P_2 \leq \frac{1}{R^2}
	\E\left[\norm{v^n(\omega)}_{L^2\left(0,T;H^1(\cO)\right)}^2\right]
	\overset{\eqref{Gal:est2-bis}}{\lesssim} \frac{1}{R^2}
	\toR 0,
	\\ & P_3 
	\leq \sum_{m=1}^\infty \frac{1}{R^2 \, \nu_m^2}
	\E\left[\, \sup_{0\le \tau \le r_m}
	\norm{v^n(\cdot+\tau)
	-v^n}_{L^2\left(0,T-\tau;L^2(\cO)\right)}^2 \right]
	\\ & \quad
	\overset{\eqref{Time-translate}}{\lesssim} 
	\frac{1}{R^2}\sum_{m=1}^\infty 
	\frac{r_m^{\frac14}}{\nu_m^2}
	\overset{\eqref{eq:r-nu}}{\lesssim} 
	\frac{1}{R^2} \toR 0.
\end{align*}
As a result, we can specify the required compact 
$\cK_{v^n,\kappa}\subset L^2_{t,\bx}$ 
as $K_R$ for some suitable $R=R(\kappa)$, 
such that $\cL_{v^n}
\bigl(\cK_{v^n,\kappa}^c\bigr)\leq \kappa$. 
We can argue similarly for $w^n$ 
and $\gamma^n$.
\end{proof}

The primary component of the stochastic compactness method is 
the tightness of the joint laws of the random 
variables under consideration, as demonstrated 
in Lemma \ref{lem:tight}. The Skorokhod 
a.s.~representation theorem enables the replacement of 
these random variables with versions that converge a.s., 
albeit defined on a new probability space. 
However, in our context, the direct application of the Skorokhod theorem 
is unfeasible due to the necessity to operate within non-metric spaces, 
specifically $\bigl(L^2_{t,\bx}\bigr)_w$ and $\bigl(L^2_tH^1_\bx\bigr)_w$. 
As a result, we employ a more recent version of the Skorokhod theorem 
that is applicable to quasi-Polish spaces, due to 
Jakubowski \cite{Jakubowski:1997aa}. A quasi-Polish space is 
defined as a Hausdorff space $\cX$ that 
enables a continuous injection into a Polish space. Notably, 
separable Banach spaces with weak topology are 
quasi-Polish. The Skorokhod-Jakubowski theorem and 
its application to SPDEs is further discussed in \cite{Brzezniak:2013aa,Brzezniak:2011aa,Ondrejat:2010aa}, showcasing 
some of the earliest uses of the theorem in this context. 
This theorem was also used in \cite{Bendahmane:2018aa} for the 
stochastic bidomain model. For further references, this paper 
may be consulted.

By the Skorokhod-Jakubowski representation theorem, 
we infer the existence of a new probability space, denoted 
as $\bigl(\tilde{\Omega},\tilde{\cF}, \tilde{P}\bigr)$, and new sets 
of random variables
\begin{equation}\label{eq:def-tvar0}
	\tilde \Phi_n=\bigl(\tilde V^n,\tilde W^n,\tilde V_0^n\bigr),
	\quad
	\tilde \Phi_\eps=\bigl(\tilde V_\eps,\tilde W_\eps,\tilde V_{0,\eps}\bigr),
\end{equation}
where $\Set{n=n_j}_{j\in \N}$ represents a subsequence that 
converges to infinity as $j\to \infty$. For ease of notation, 
we use $n$ instead of $n_j$ to avoid relabelling the subsequence.
The new random variables hold respective joint laws $\cL_n$ and $\cL$. 
Specifically, the law of $\tilde \Phi_n$ coincides with that of $\Phi_n$. 
Crucially, the following almost sure convergence is valid:
\begin{equation}\label{eq:def-tvar0-conv}
	\tilde \Phi_n \ton \tilde \Phi_\eps
	 \quad\text{a.s.~in $\cX$}.
\end{equation}

Regarding the tilde notation, the components of 
$\tilde V^n,\tilde W^n,\tilde V_0^n$ bear the same symbolic 
representations as $V^n,W^n,V_0^n$, except with a tilde symbol 
overhead. For instance, $v^n_i$ would transform 
into $\tilde v^n_i$, and a similar rule applies to the other variables.  
The components of $\tilde V_\eps,\tilde W_\eps,\tilde V_{0,\eps}$ are denoted 
in the same way. For example, the limit of $\tilde v^n_i$ 
as $n\to \infty$ would be symbolized as $\tilde v_{i,\eps}$ (the 
parameter $\eps$ will be sent to zero in the next section).

For the sake of clarity and to facilitate future reference, the 
almost sure convergence \eqref{eq:def-tvar0-conv} can be recast into 
the following \textit{strong} convergences:
\begin{equation}\label{eq:strong-conv1}
	\begin{split}
		& \tilde v^n \ton \tilde v_\eps, 
		\,\, 
		\tilde w^n \ton \tilde w_\eps,
		\,\, 
		\tilde \gamma^n \ton \tilde \gamma_\eps
		\quad \text{a.s.~in $L^2_{t,\bx}$,}
		\\ &
		\tilde W^{v,n} \ton \tilde W^v_\eps, 
		\,\,
		\tilde W^{w,n} \ton \tilde W^w_\eps 
		\quad \text{a.s.~in $C_t\U_0$,}
		\\ & 
		\tilde V_0^n\ton \tilde V_{0,\eps}				
		\quad  \text{a.s.~in $L^2_\bx$, 
		see \eqref{eq:def-U-W-U0}}.
	\end{split}
\end{equation}
Moreover, the following \textit{weak} convergences hold: 
\begin{equation}\label{eq:weak-conv1}
	\begin{split}
		& \tilde v^n_i \weakn \tilde v_{i,\eps},
		\,\,
		\tilde v^n_e \weakn \tilde v_{e,\eps}
		\quad \text{a.s.~in $L^2_tH^1_\bx$},
		\\ & 
		\tilde \bu^n \weakn \tilde \bu_\eps 
		\quad \text{a.s.~in $L^2_t[H^1_\bx]^3$}, 
		\quad 
		\tilde p^n \weakn \tilde p_\eps
		\quad \text{a.s.~in $L^2_{t,\bx}$.}
	\end{split}
\end{equation}

Leveraging the equality of laws, the a priori 
estimates in Lemma \ref{lem:apriori-est}
and Corollary \ref{cor:Lq0-est} continue to hold 
for the new random variables 
$\tilde v^n_i$, $\tilde v^n_e$, 
$\tilde v^n$, 
$\tilde w^n$, 
$\tilde p^n$, 
$\bu^n$, defined on $\bigl(\tilde{\Omega},\tilde{\cF}, \tilde{P}\bigr)$. 
Besides, we have $\tilde \bu_0^n\equiv 0$ 
and $\tilde v^n=\tilde v^n_i-\tilde v^n_e$ a.e. 
In subsequent discussions, even when referencing estimates 
for variables with a tilde atop, we will consistently use the 
equation numbers originally associated 
with their non-tilde counterparts.

Moving forward, we will use the following stochastic 
basis associated with $\tilde \Phi_n$, $n\in \N$:
\begin{equation}\label{eq:stoch-basis-n}
	\begin{split}
		& \tilde \cS_n=\left(\tilde \Omega,\tilde \cF, 
		\Setb{\tilde \cF_t^n}_t,\tilde P,
		\tilde W^{v,n},\tilde W^{w,n}\right),
		\\ & 
		\tcFt^n =\sigma\left(\sigma \bigl(\tilde \Phi_n\big |_{[0,t]}\bigr)
		\bigcup \bigl\{N \in \tilde \cF:\tilde P(N)
		=0\bigr\}\right), \qquad t\in [0,T],
	\end{split}
\end{equation} 
where $\Setb{\tcFt^n}_t$ is the smallest filtration making the 
processes \eqref{eq:def-tvar0} adapted.  Note that, by the equality of laws 
and \cite{DaPrato:2014aa}, $\tilde W^{v,n}$ and $\tilde W^{w,n}$ are 
cylindrical Brownian motions. We also need the $n$-truncated sums 
\begin{equation*}
	\tilde W^{v,(n)} = \sum_{k=1}^n\tilde W_k^{v,n} \psi_k,
	\qquad
	\tilde W^{w,(n)} = \sum_{k=1}^n\tilde W_k^{w,n} \psi_k,
\end{equation*}
which a.s.~converge respectively to $\tilde W^v_\eps$, $\tilde W^w_\eps$ 
in $C([0,T];\U_0)$, cf.~\eqref{eq:strong-conv1}.

We still need to show that the tilde variables, $\tilde V^n, \tilde W^n, 
\tilde V_0^n$ satisfy the Faedo-Galerkin equations on 
the newly defined stochastic basis \eqref{eq:stoch-basis-n}. 
This is achieved by leveraging the equivalence of laws. 
Multiple strategies can be employed to flesh out the details, as 
seen in various sources like
\cite{Bensoussan:1995aa,Brzezniak:2011aa,Ondrejat:2010aa}.
In particular, \cite{Bendahmane:2018aa} utilized the method presented in \cite{Bensoussan:1995aa} to tackle the stochastic bidomain model. 
The same argument can be adopted here for 
\eqref{eq:approx-eqn-integrated}. However, we will not delve into the 
nitty-gritty of the proof; instead, we direct the interested reader 
to \cite{Bendahmane:2018aa} or any of the other cited 
references for a more in-depth exploration. 

As a result, we can, for the subsequent discussions, make 
the assumption that the tilde variables $\tilde V^n,\tilde W^n,\tilde V_0^n$ 
conform to the Faedo-Galerkin equations. This, symbolically, 
means that equations \eqref{eq:approx-eqn-integrated} 
are upheld when a tilde is placed over each 
variable in the equations. Moving forward, whenever we refer 
to the Faedo-Galerkin equations in the context of the tilde variables, 
we are indeed referencing the equations \eqref{eq:approx-eqn-integrated} 
pertaining to the original variables. This approach avoids 
the unnecessary repetition of equations, thus conserving space.

\section{Solution to non-degenerate model}
\label{sec:conv-nondegen}

This section aims to prove that the limiting functions 
in \eqref{eq:strong-conv1} and \eqref{eq:weak-conv1} form 
a weak martingale solution for the non-degenerate 
system \eqref{eq:nondegen-S1}. This solution also incorporates 
the boundary conditions given by \eqref{S-bc} and the auxiliary initial 
data for the variables $v_{i,\eps}$, $v_{e,\eps}$, 
$\bu_\eps$, and $p_\eps$, which were discussed earlier.

In addition to \eqref{eq:stoch-basis-n}, we need 
a stochastic basis for the a.s.~limit $\tilde \Phi_\eps$ 
in \eqref{eq:def-tvar0}, \eqref{eq:def-tvar0-conv}: 
\begin{equation}\label{eq:stoch-basis-eps}
	\tilde \cS_\eps=\left(\tilde \Omega,\tilde \cF, 
	\Setb{\tilde \cF_t^\eps}_t,\tilde P,
	\tilde W^v_\eps,\tilde W^w_\eps\right),
\end{equation}
where $\Setb{\tilde \cF_t^\eps}_t$ is the 
smallest filtration that makes all the 
relevant processes in $\tilde \Phi_\eps$ adapted. 
Then $\tilde W^v_\eps$, $\tilde W^w_\eps$ 
are (independent) cylindrical Wiener processes 
with respect to sequences $\Setb{\tilde W_{\eps,k}^v}_{k\in \N}$, 
$\Setb{\tilde W_{\eps,k}^w}_{k\in \N}$ of mutually 
independent real-valued Wiener processes adapted 
to the filtration $\Setb{\tilde \cF_t^\eps}_{t}$, such that 
$$
\tilde W^v_\eps=\sum_{k\ge 1}\tilde W_{k,\eps}^v \psi_k, 
\quad 
\tilde W^v_\eps=\sum_{k\ge 1}\tilde W_{\eps,k}^w \psi_k.
$$ 
See \cite{Bendahmane:2018aa} for further details and references. 

At this point, we have established 
the almost sure convergences in \eqref{eq:strong-conv1} 
and \eqref{eq:weak-conv1} for the tilde variables. 
In addition, the tilde variables adhere to the statistical estimates 
provided by Lemma \ref{lem:apriori-est} and Corollary \ref{cor:Lq0-est}.
Considering the implications of the Banach-Alaoglu theorem 
along with Vitali's convergence theorem, 
after possibly thinning the sequence, we are justified in 
assuming that the following convergences occur as $n\to \infty$:
\begin{equation}\label{eq:tilde-weakconv1}
	\begin{split}
		& \tilde v^n \to \tilde v_\eps, 
		\,\, 
		\tilde w^n \to \tilde w_\eps,
		\,\,
		\tilde \gamma^n \to \tilde \gamma_\eps
		\quad \text{a.e.~and in $L^2\bigl(\tilde \Omega;
		L^2((0,T)\times \cO)\bigr)$},
		\\ &
		\tilde v^n_i \weak \tilde v_{i,\eps},
		\,\,
		\tilde v^n_e \weak \tilde v_{e,\eps}
		\quad \text{in $L^2\bigl(\tilde \Omega;
		L^2(0,T;H^1(\cO))\bigr)$},
		\\ &
		\tilde \bu^n \weak \tilde \bu_\eps
		\quad \text{in $L^2\bigl(\Omega;L^2(0,T;
		[H^1(\cO)]^3)\bigr)$},
		\\ &
		\tilde p^n \weak \tilde p_\eps \quad
		\text{in $L^2\bigl(\tilde\Omega;
		L^2((0,T)\times \cO)\bigr)$},
		\\ &
		\tilde W^{v,n} \to \tilde W^v_\eps,
		\,\, \tilde W^{w,n} \to \tilde W^w_\eps
		\quad  \text{in $L^2\bigl(\tilde \Omega;
		C([0,T]; \U_0)\bigr)$},
		\\ &
		\tilde v^n_0 \to \tilde v_{0,\eps},
		\,\,
		\tilde v^n_{j,0} \to \tilde v_{j,0,\eps},\,\, j=i,e,
		\quad  \text{in $L^2\bigl(\tilde \Omega;L^2(\cO)\bigr)$},
		\\ &
		\tilde w^n_0 \to \tilde w_{0,\eps},
		\,\,
		\tilde \gamma^{n}_{0} \to \tilde \gamma_{0,\eps}
		\quad  \text{in $L^2\bigl(\tilde \Omega;L^2(\cO)\bigr)$},
		\\ &
		\tilde p^{n}_{0} \to \tilde p_{0,\eps}
		\quad \text{in $L^2\bigl(\tilde \Omega;L^2(\cO)\bigr)$},
	\end{split}
\end{equation}
where the limit functions are bounded
in the spaces suggested by the uniform 
estimates supplied by Lemma \ref{lem:apriori-est} 
and Corollary \ref{cor:Lq0-est}. Besides, 
$$
\tilde v_\eps = \tilde v_{i,\eps}- \tilde v_{e,\eps}
\quad 
\text{a.e.~in $\tilde \Omega\times [0,T]\times \cO$}.
$$

With the convergences established 
above \eqref{eq:tilde-weakconv1}, the 
concluding step involves taking the limit in the 
Faedo-Galerkin equations \eqref{eq:approx-eqn-integrated} as $n\to \infty$ 
(in this section, $\eps$ is kept fixed). The details of this step will 
be provided in the subsequent lemma.

\begin{lem}\label{eq:limit-eqs-tilde}
The following equations hold for a.e.~$t\in [0,T]$, 
$\tilde P$-almost surely: First,
\begin{equation}\label{eq:weakform-tilde-Stokes}
	\begin{split}
		& \eps\int_{\cO}\tilde\bu_{\eps}(t)\cdot \bv \dx
		+\int_0^t\int_{\cO} \nabla\tilde\bu_{\eps}
		\bsigma(\bx,\tilde\gamma_{\eps}):
		\nabla \bv-\tilde p_{\eps}\Grad\cdot \bv\dx \ds
		\\ & \qquad
		+\int_0^t\int_{\partial \cO}
		\alpha\tilde\bu_{\eps}\cdot \bv\, dS(\bx)\ds
		=\int_0^t\int_{\cO} \ffte \cdot \bv \dx\ds, 
		\\ & 
		\eps\int_{\cO}\tilde p_{\eps}(t) \vphi_p\dx
		+\int_0^t\int_{\cO} (\nabla \cdot \tilde \bu_{\eps})
		\vphi_p\dx\ds=\eps\int_{\cO}\tilde p_{0,\eps} \vphi_p\dx,
	\end{split}
\end{equation}
for all $\bv\in [H^1(\cO)]^3$ and $\vphi_p \in H^1(\cO)$.
Second, for all $\vphi_i\in H^1(\cO)$ and $\vphi_e\in H^1(\cO)$,
\begin{equation}\label{eq:weakform-tilde-bidomain}
	\begin{split}
		& \int_{\cO}\tilde v_{\eps}(t)\vphi_i\dx
		+\eps\int_{\cO}\tilde v_{i,{\eps}}(t) \vphi_i\dx
		+\int_0^t\int_{\cO} \bM_i(\bx,\Grad\tilde\bu_{\eps})
		\nabla \tilde v_{i,\eps}\cdot\nabla\vphi_i\dx\ds
		\\ & \qquad 
		+\int_0^t\int_\cO
		\Ion(\tilde v_{\eps}, \tilde w_{\eps})\vphi_i
		=\int_{\cO}\tilde v_{0,\eps}\vphi_i\dx
		+\eps\int_{\cO}\tilde v_{i,0,\eps}\vphi_i\dx
		\\ & \qquad \qquad
		+ \int_0^t\int_{\cO}I_{app}\vphi_i
		+\int_0^t \int_\cO\beta_v(\tilde v_\eps)
		\vphi_i \dx \,d\tilde W^{v}(s),
		\\ &
		\int_{\cO}\tilde v_{\eps}(t)\vphi_e
		-\eps \int_{\cO}\tilde v_{e,\eps}(t)\vphi_e
		-\int_0^t\int_{\cO} \bM_e(\bx,\Grad\tilde\bu_{\eps})
		\nabla \tilde v_{e,\eps} \cdot \nabla\vphi_e\dx\ds
		\\ & \qquad 
		+\int_0^t\int_\cO\Ion(\tilde v_{\eps},\tilde w_{\eps})
		\vphi_e\dx\dt
		=\int_{\cO}\tilde v_{0,\eps}\vphi_e\dx
		-\eps \int_{\cO}\tilde v_{e,0,\eps}\vphi_e\dx
		\\ & \qquad \qquad 
		-\int_0^t\int_{\cO}I_{app}\vphi_e\dx\ds
		+\int_0^t\int_\cO\beta_v(\tilde v_\eps)
		\vphi_e \dx\,d\tilde W^{v}(s).
	\end{split}
\end{equation}
Finally, for all $\vphi_w,\vphi_\gamma \in H^1(\cO)$,
\begin{equation}\label{eq:weakform-tilde-ODE-SDE}
	\begin{split}
		& \int_{\cO}\tilde {w}_{{\eps}}(t)\vphi_w\dx
		= \int_{\cO}\tilde {w}_{0,\eps}\vphi_w\dx
		\\ & \qquad 
		+\int_0^t\int_{\cO}H(\tilde v_{\eps},\tilde w_{\eps})\vphi_w\dx\ds
		+\int_0^t\int_\cO\beta_w(\tilde v_\eps)\vphi_w 
		\dx \,d\tilde W^{w}(s),
		\\ & 
		\int_{\cO} \tilde\gamma_{\eps}(t)\vphi_\gamma\dx 
		=\int_{\cO}\tilde\gamma_{0,\eps}\vphi_\gamma\dx
		+\int_0^t\int_{\cO} 
		G(\tilde \gamma_{\eps},\tilde w_{\eps})
		\vphi_\gamma \dx\ds,
	\end{split}
\end{equation}
The laws of $\tilde v_\eps(0)=\tilde v_{0,\eps}$, 
$\tilde w_\eps(0)=\tilde w_{0,\eps}$, 
$\tilde \gamma_\eps(0)=\tilde \gamma_{0,\eps}$
are $\mu_{v_{0}}$, $\mu_{w_{0}}$, 
$\mu_{\gamma_{0}}$, respectively.
\end{lem}

\begin{rem}
Note that the (regularized) Stokes-type 
equations \eqref{eq:weakform-tilde-Stokes} hold a.e.~in 
the time variable  $t$. On the other hand, 
Definition \ref{def:martingale-sol} 
stipulates them to be valid only in the sense of 
distributions on $(0,T)$. Nevertheless, by weakly 
differentiating the a.e.~formulation 
\eqref{eq:weakform-tilde-Stokes} with 
respect to time, we obtain
\begin{align*}
	& \eps\int_{\cO} \partial_t \tilde\bu_{\eps}\cdot \bv \dx
	+\int_{\cO} \nabla\tilde\bu_{\eps}
	\bsigma(\bx,\tilde\gamma_{\eps}):
	\nabla \bv-\tilde p_{\eps}\Grad\cdot \bv\dx
	\\ & \qquad
	+\int_{\partial \cO}
	\alpha\tilde\bu_{\eps}\cdot \bv\, dS(\bx)
	=\int_{\cO} \ffte \cdot \bv \dx 
	\quad \text{in $\Dp(0,T)$}, 
	\\ & 
	\eps\int_{\cO} \partial_t \tilde p_{\eps} \vphi_p\dx
	+\int_{\cO} (\nabla \cdot \tilde \bu_{\eps})
	\vphi_p\dx=0
	\quad \text{in $\Dp(0,T)$}.
\end{align*}
\end{rem}

\begin{proof}
The proof follows \cite{Bendahmane:2018aa}, while focusing on 
the specific modifications introduced by our extended model. 
These modifications primarily pertain to the analysis of 
the Stokes-like system \eqref{eq:weakform-tilde-Stokes}, the necessity 
of taking the limit in the nonlinear conductivities within the 
bidomain component \eqref{eq:weakform-tilde-bidomain}, 
and finally the treatment of the ODE for $\gamma_\eps$ 
in \eqref{eq:weakform-tilde-ODE-SDE}.

We start by establishing \eqref{eq:weakform-tilde-Stokes}.
Fix $\bv\in [H^1(\cO)]^3$ and $\vphi_p\in H^1(\cO)$, and note 
that the Faedo--Galerkin equations 
\eqref{eq:approx-eqn-integrated} imply that
\begin{align}
	& \eps \int_{\cO}\tilde\bu^n(t) \cdot \Pi_n \bv\dx
	+\int_0^t\int_{\cO} \nabla\tilde\bu^n
	\bsigma(\bx,\tilde\gamma^n):
	\nabla (\Pi_n\bv) \dx\ds
	-\tilde p^n\Grad\cdot (\Pi_n\bv)\dx\ds
	\notag \\  & \qquad \qquad
	+\int_0^t\int_{\partial \cO}\alpha 
	\tilde\bu^n\cdot (\Pi_n \bv)\, dS(\bx)\ds
	=\int_0^t\int_{\cO} \fftn\cdot \Pi_n\bv \dx \ds,
	\label{eq:weakform-tilde-tmp1-meca1} 
	\\ & \eps\int_{\cO}\tilde p^n(t) \vphi_p\dx
	+\int_0^t\int_{\cO} (\nabla\cdot\tilde\bu^n)\vphi_p\dx\ds
	=\eps\int_{\cO}\tilde p^n_{0}\vphi_p\dx.
	\notag
\end{align}

To state the desired equations \eqref{eq:weakform-tilde-Stokes} 
in a more concise manner, we can use the notations 
$I_{\bu}^{\bv}(\tilde \omega,t)=0$ 
and $I_p^{\vphi_p}(\tilde \omega,t)=0$, 
where $I_{\bu}^{\bv}$ and $I_p^{\vphi_p}$ are 
appropriate (integrable) functions on $\tilde \Omega \times [0,T]$ 
defined by \eqref{eq:weakform-tilde-Stokes}. 
Let $Z\subset \tilde \Omega\times [0,T]$ be a 
measurable set, and denote by 
\begin{equation*}
	\En_Z(\tilde\omega,t)\in 
	L^\infty\bigl(\tilde \Omega\times [0,T];
	\tilde P\times dt \bigr)
\end{equation*}
the characteristic function of $Z$. We will show that 
\begin{equation*}
	\tilde \E \left[\int_0^T \En_Z I_{\bu}^{\bv}\dt\right] =0,
	\quad
	\tilde \E \left[\int_0^T \En_Z I_p^{\vphi_p}\dt\right] =0,
\end{equation*}
for any measurable set $Z\subset \tilde \Omega\times [0,T]$, 
and for a countable collections of test functions 
$\bv\in \Setb{\bv^{(\ell)}}_{\ell\in \N}\subset [H^1_\bx]^3$, 
$\vphi_p\in \Setb{\vphi_p^{(\ell)}}_{\ell\in \N}
\subset H^1_\bx$. This is enough to conclude that $I_{\bu}^{\bv}=0$ 
and $I_p^{\vphi_p}=0$ for ($\tilde P\times dt$) a.e.~$(\tilde \omega,t)$, 
for all $H^1$ test functions $\bv$, $\vphi_p$, i.e., 
\eqref{eq:weakform-tilde-Stokes} holds.

To demonstrate this, we will employ the following approach: first, we 
multiply the Faedo-Galerkin equations \eqref{eq:weakform-tilde-tmp1-meca1} 
by $\En_Z$ and integrate them over $\tilde \omega$ and $t$. 
Subsequently, we will take the limit as $n\to \infty$.  
A similar line of reasoning will be applied to the 
components \eqref{eq:weakform-tilde-bidomain} 
and \eqref{eq:weakform-tilde-ODE-SDE} of the system.

By the weak $L^2_{\tilde \omega,t,x}$ convergence 
of $\tilde \bu^n$ to $\tilde \bu_\eps$, 
see \eqref{eq:tilde-weakconv1} and the strong 
$L^2$ convergence of $\Pi_n \bv$ to $\bv$. we obtain
$$
\tilde \E \left[\int_0^T \En_Z 
\left(\eps\int_{\cO}\tilde \bu^n(t) 
\cdot \Pi_n \bv \dx\right) \dt \right]
\ton \tilde \E \left[\int_0^T \En_Z
\left(\eps \int_{\cO}  \tilde \bu_\eps(t)
\cdot\bv \dx\right) \dt \right],
$$
Moreover, by \eqref{eq:tilde-weakconv1},
\begin{align*}
	&\tilde \E \left[\int_0^T \En_Z
	\left(\, \int_{\cO} \tilde p^n(t) 
	\vphi_p \dx\right) \dt \right]
	\ton \E \left[\int_0^T \En_Z 
	\left(\, \int_{\cO} \tilde p_\eps 
	\vphi_p \dx\right) \dt \right].
\end{align*}
Similarly, we can pass to the limit in the 
terms involving $-\tilde p^n\Grad\cdot (\Pi_n\bv)$, 
and  $(\nabla\cdot\tilde\bu^n)\vphi_p$. 
For the term associated with $\alpha \tilde{\bu}^n \cdot (\Pi_n \bv)$, 
the weak convergence of $\alpha \bu^n$ on the boundary is 
attributed to the continuity of the trace map from 
$[H^1(\cO)]^3$ to $[L^2(\partial \cO)]^3$. 
Moreover, we will soon establish the strong convergence 
$\tilde{\bu}^n \to \tilde{\bu}_\varepsilon$ in 
$L^2_{\tilde{\omega},t}[H^1_{\bx}]^3$ as $n \to \infty$. 
By using \eqref{eq:tilde-weakconv1} once more, we 
can easily handle the initial data terms 
in \eqref{eq:weakform-tilde-tmp1-meca1}.

Recall the well-known fact 
that if $a_n\to a$ a.e.~with $\norm{a_n}_{L^\infty}\lesssim 1$ 
and $b_n\weak b$ in $L^2$, then 
$a_nb_n \weak a b$ in $L^2$. By the almost everywhere 
convergence $\tilde \gamma^n\to \tilde \gamma_\eps$, 
see \eqref{eq:tilde-weakconv1}, $\bsigma(\bx,\tilde \gamma^n)
\to \bsigma(\bx,\tilde \gamma_\eps)$ a.e. Besides,
the $L^\infty$ bound on the tensor $\bsigma$, which follows from 
the assumption {\bf (A.1)}, implies that
$\norm{\bsigma(\cdot,\tilde \gamma^n)}_{L^\infty_{\tilde \omega,t,\bx}}
\lesssim 1$. According to \eqref{eq:tilde-weakconv1}, 
$\Grad \tilde \bu^n\to\Grad \tilde \bu_\eps$ 
weakly in $L^2_{\tilde \omega,t,\bx}$. Thus, the product 
$\Grad \tilde \bu^n\bsigma(\bx,\tilde \gamma^n)$ 
converges weakly in $L^2_{\tilde \omega,t,\bx}$. 
Combining this and the strong $L^2$ convergence 
$\Grad (\Pi_n \bv)\to \Grad \bv$, we conclude that 
the weak-strong product $\Grad \tilde \bu^n
\bsigma(\bx,\tilde \gamma^n):\Grad (\Pi_n \bv)$ 
converges weakly in $L^1_{\tilde \omega,t,\bx}$, and therefore
\begin{align*}
	&\tilde \E \left[\int_0^T \En_Z
	\left(\, \int_0^t \int_{\cO} \bsigma(\bx,\tilde \gamma^n) 
	\Grad \tilde \bu^n:\Grad (\Pi_n \bv) \dx\ds \right)\dt \right]
	\\ & \qquad \ton 
	\tilde \E \left[\int_0^T \En_Z
	\left(\, \int_0^t \int_{\cO} \bsigma(\bx,\tilde \gamma_\eps) 
	\Grad \tilde \bu:\Grad \bv \dx\ds\right)\dt \right].
\end{align*}

In relation to the remaining term in the elasticity equation 
involving $\fftn\cdot \bv$, where 
$\fftn=\ff(t,\bx,\tilde \gamma^n)$, 
we can refer back to \eqref{def:ff} which, through 
\eqref{eq:tilde-weakconv1}, shows that $\fftn$ weakly 
converges to $\ffte=\ff(t,\bx,\tilde \gamma_\eps)$ 
in $L^2_{\tilde \omega,t,\bx}$. Simultaneously, we know 
that $\Pi_n\bv\to \bv$ strongly in $L^2$. Therefore,
\begin{align*}
	&\tilde \E \left[\int_0^T \En_Z 
	\left(\, \int_0^t \int_{\cO} \ff(t,\bx,\tilde \gamma^n)
	\cdot  \Pi_n\bv  \dx\ds \right) \dt \right]
	\\ & \qquad
	\ton \tilde \E \left[\int_0^T \En_Z 
	\left(\, \int_0^t \int_{\cO} \ff(t,\bx,\tilde \gamma_\eps)
	\cdot \bv \dx\ds \right)\dt \right].
\end{align*}

Let us now shift our attention to a new aspect of the analysis for 
the electrical (bidomain) component of the model. 
The focus will be on elucidating the process of taking the limit in 
the nonlinear conductivities, the remaining terms can be treated 
as in \cite{Bendahmane:2018aa}. This includes passing to the limit in 
the ODE \eqref{eq:approx-eqn-integrated} for the 
new variable $\tilde \gamma^n$. 
Notably, this step is simpler than handling the SDE for 
the variable $\tilde w^n$, which was already adeptly 
addressed in \cite{Bendahmane:2018aa}.

To illustrate the new difficulty, we consider the following 
terms derived from the Faedo--Galerkin equations 
\eqref{eq:approx-eqn-integrated}:
\begin{align*}
	J_j^n := \tilde \E \left[\int_0^T \En_Z \left(\, \int_0^t\int_{\cO}
	\bM_i(\bx,\nabla\tilde \bu^n)\Grad \tilde v^n_i
	\cdot \nabla (\Pi_n \vphi_i) \dx \ds \right)\dt\right],
	\quad j=i,e,
\end{align*}
where $\nabla (\Pi_n \vphi_j)\to \nabla \vphi_j$ strongly 
in $L^2$ and $\Grad \tilde v^n_j\to \Grad \tilde v_{j,\eps}$ 
weakly in $L^2_{\tilde \omega,t,\bx}$, 
see \eqref{eq:tilde-weakconv1}. 

To facilitate the computation of the limit of $J_j^n$, it is 
crucial to establish the strong convergence, rather than 
weak convergence, of $\nabla\tilde \bu^n$ in $L^2_{\tilde \omega,t,\bx}$. 
This strong convergence is essential as $\nabla\tilde \bu^n$ is 
a component of the nonlinear functions $\bM_j$. 
Although the weak convergence is provided 
by \eqref{eq:tilde-weakconv1}, to make progress, we must 
prove strong convergence.  To achieve this, we have 
two potential approaches. Firstly, we can employ the well-known 
Minty-Browder argument. Alternatively, we can leverage the 
structural characteristics of the purely 
mechanical equation along with the coercivity 
of the bilinear form $a(\cdot,\cdot)$---introduced 
in \eqref{bilinear-forma}---to derive the following estimate:
\begin{align*}
	& \dfrac{1}{c} \tilde \E \left[\norm{\tilde \bu^n
	-\tilde \bu_\eps}^2_{L^2(0,T;[H^1(\cO)]^3)}\right]
	\leq \tilde \E \left[\int_0^T a\bigl(\bu^n-\tilde\bu_\eps,
	\tilde \bu^n-\tilde \bu_\eps\bigr) \dt \right]
	\\ & \qquad 
	\leq 
	\eps \tilde \E\left[\norm{\tilde p^n(0)
	-\tilde p_\eps(0)}^2_{L^2((0,T)\times \cO)}\right]
	\\ & \qquad \qquad
	-\tilde \E\left[\int_0^T\int_{\cO} \nabla \tilde \bu_\eps
	\bigl[ \bsigma(\bx,\tilde \gamma^n)
	-\bsigma(\bx,\tilde \gamma_\eps)\bigr]:
	\nabla \left(\tilde \bu^n-\tilde \bu_\eps\right)\dx\dt\right]
	\\ & \qquad\qquad
	-\tilde \E\left[\int_0^T\int_{\cO} 
	\bigl[\ff(\bx,\tilde \gamma^n)
	-\ff(\bx,\tilde \gamma_\eps)\bigr]
	\cdot \left(\tilde \bu^n-\tilde \bu\right)\dx\dt\right],
\end{align*}
for some constant $c>0$ independent of $n$ (and $\eps$).  
We claim that the three terms on the right-hand side converge to zero 
as $n\to \infty$. The initial data term converges to zero as 
stated in \eqref{eq:tilde-weakconv1}. 
By recalling the definition of $\ff$ given in \eqref{def:ff}, we 
can employ integration by parts. This approach 
effectively brings us to the consideration of the following 
strong-weak product, disregarding the straightforward 
term that involves $\bg$:
$$
\bigl[\bsigma(\bx,\tilde \gamma^n)
-\bsigma(\bx,\tilde \gamma_\eps)\bigr]:
\left(\nabla\tilde \bu^n-\nabla \tilde \bu\right).
$$
This product term converges to zero---weakly in 
$L^1_{\tilde \omega,t,\bx}$---by the weak $L^2_{\tilde \omega,t,\bx}$
convergence of $\nabla \tilde u_\eps^n$ and the strong 
$L^2_{\tilde \omega,t,\bx}$ convergence of 
$\bsigma(\bx,\tilde \gamma^n)$. 
To establish the latter, we have utilized the fact that 
$\bsigma(\bx,\tilde \gamma^n) \to \bsigma(\bx,\tilde \gamma_\eps)$ 
almost everywhere. This, combined with the assumption 
$\sigma\in L^\infty$, establishes the strong convergence 
in $L^2_{\tilde \omega,t,\bx}$. Thus, as we let $n\to \infty$, the 
third term on the right-hand side converges to zero. 

For the second term, referring to \eqref{eq:tilde-weakconv1}, 
$\abs{\nabla \tilde \bu_\eps
\bigl[ \bsigma(\bx,\tilde \gamma^n)
-\bsigma(\bx,\tilde \gamma_\eps)\bigr]}^2$ converges to zero 
a.e.~in $\tilde \Omega \times [0,T]\times \cO$ as $n\to \infty$. 
Furthermore, the assumption ${\bf (A.1)}$ ensures that $\sigma$ is bounded, 
thereby implying that the above-mentioned term is capped 
by a constant times $\abs{\nabla \tilde \bu_\eps}^2\in 
L^1_{\tilde \omega,t,\bx}$. Consequently, by applying Lebesgue's dominated 
convergence theorem, the term $\nabla \tilde \bu_\eps
\bigl[ \bsigma(\bx,\tilde \gamma^n)
-\bsigma(\bx,\tilde \gamma_\eps)\bigr]\to 0$ 
strongly in $L^2_{\tilde \omega,t,\bx}$ as $n\to \infty$. 
Simultaneously, \eqref{eq:tilde-weakconv1} tells us that the 
gradient of $\tilde \bu^n-\tilde \bu_\eps$ 
tends to zero weakly in $L^2_{\tilde \omega,t,\bx}$.
Hence, the second term on the right-hand side 
of the above inequality converges to zero as $n\to \infty$. 

To sum up, $\tilde \bu^n\to \tilde \bu_\eps$ 
in $L^2_{\tilde\omega,t}[H^1_{\bx}]^3$ as $n\to \infty$. 
Considering the a.s.~convergences \eqref{eq:strong-conv1}, 
\eqref{eq:weak-conv1} supplied by the Skorokhod--Jakubowski procedure, 
we can also deduce that $\tilde \bu^n\to \tilde \bu_\eps$ 
a.s.~in $L^2_t[H^1_{\bx}]^3$. By possibly thinning the 
sequence, we may assume that $\nabla\tilde \bu^n\to \nabla \tilde \bu_\eps$ 
almost everywhere. Due to the $L^\infty$ bound on $\bM_i$---as per 
assumption {\bf (A.2)}---we are led to the conclusion that
$$
\bM_i(\bx,\nabla\tilde \bu^n)\Grad \tilde v^n_i
\ton \bM_i(\bx,\nabla \tilde \bu_\eps)\Grad \tilde v_{i,\eps} 
\quad 
\text{weakly in $L^2_{\tilde \omega,t,\bx}$}.
$$
This weak convergence employs once again the fact that 
products of functions, which exhibit bounded convergence 
a.e.~and weak convergence in $L^2$, 
converge weakly in $L^2$ to the product of the limits. 
Given that $\nabla (\Pi_n \vphi_j)\to \nabla \vphi_j$ strongly 
in $L^2$, we can infer that the sequence $\Set{J_j^n}_{n\in \N}$ 
will indeed converge to the correct limit as $n\to \infty$. This 
concludes the proof of the lemma. For additional particulars, we 
refer to \cite{Bendahmane:2018aa}.
\end{proof}

Recall that a stochastic process $z$ taking values 
in a Banach space $X$---defined on 
$\cS_\eps$, see \eqref{eq:stoch-basis-eps}---is termed 
\textit{predictable} if $z:\tilde \Omega \times [0,T] \to X$ 
is measurable with respect to the product $\sigma$-algebra 
$\cP_T \times \mathcal{B}([0, T])$, 
where $\cP_T$ is the predictable 
$\sigma$-algebra associated with $\Setb{\tilde \cF_t^\eps}$ 
(the $\sigma$-algebra generated by all 
left-continuous adapted processes). 
A predictable process is both adapted and jointly measurable. 
Although the converse implication does not hold, adapted 
processes with regular paths (such as continuous paths) 
are predictable. 

More generally, given an equivalence class 
$z \in L^1\bigl(\tilde \Omega;L^1(0,T; X)\bigr)$, 
we say that $z$ is predictable 
if there exists a predictable 
stochastic process $\hat{z}$ with values in $X$ 
such that, ($\tilde P \times dt$) almost everywhere, $\hat{z}=z$.  
Given the equations \eqref{eq:weakform-tilde-bidomain} 
and \eqref{eq:weakform-tilde-ODE-SDE}, as well as 
$\bM_{i,e}\in L^\infty$ and the higher moment estimates 
in Corollary \ref{cor:Lq0-est}, we can adapt the time-continuity proof 
in \cite{Bendahmane:2018aa}. This allows us to assert that there exist 
representatives (versions) of $\tilde v_\eps$, $\tilde w_\eps$, 
and $\tilde \gamma_\eps$ that belong a.s.~to 
$C([0,T];(H^1(\cO))^*)$. As a specific inference, both 
$\tilde v_\eps$ and $\tilde w_\eps$ are predictable in the space 
$(H^1(\cO))^*$.

\section{Solution to original model}\label{sec:conv-degen}

The analysis of tightness and compactness, 
in the context of the ``$n \to \infty$" limit, as 
elaborated in Lemma \ref{lem:tight},  
\eqref{eq:strong-conv1}, \eqref{eq:weak-conv1}, 
\eqref{eq:stoch-basis-eps}, \eqref{eq:tilde-weakconv1}, 
and Lemma \ref{eq:limit-eqs-tilde}, can be reiterated 
for the ensuing limit $\eps \to 0$, assuming that 
$\eps$ takes values in a sequence 
$\Set{\eps_k}_{k\in \N}$ that tends to zero as $k \to \infty$. 
Thus, by possibly thinning the sequence, we may assume that as $\eps \to 0$, 
\begin{equation}\label{eq:tilde-weakconv1-eps}
	\begin{split}
		& \tilde v_\eps \to \tilde v, 
		\,\, 
		\tilde w_\eps \to \tilde w,
		\,\, 
		\tilde \gamma_\eps \to \tilde \gamma
		\quad \text{a.e.~and in $L^2_{\tilde \omega}L^2_{t,\bx}$},
		\\ & 
		\tilde v_{i,\eps} \weak \tilde v_i, 
                 \,\,
		\tilde v^n_{e,\eps} \weak \tilde v_e 
		\quad 
		\text{in $L^2_tH^1_\bx$, a.s., 
		and in $L^2_{\tilde \omega}L^2_tH^1_\bx$},
		\\ & 
		\tilde \bu_\eps \weak \tilde \bu
		\quad \text{in $L^2_t[H^1_\bx]^3$, a.s.,
		and in $L^2_{\tilde \omega}
		L^2_t[H^1_\bx]^3$},
                 \\ & 
                 \tilde W^v_\eps \to \tilde W^v, 
                 \,\, \tilde W^w_\eps \to \tilde W^w
                 \quad  \text{in $C_t\U_0$, a.s., 
                 and in $L^2_{\tilde \omega}C_t\U_0$},
                 \\ & 
                 \tilde v_{0,\eps} \to \tilde v_0, 
                 \,\, 
                 \tilde w_{0,\eps} \to \tilde w_0,
                 \,\,
                 \tilde \gamma_{0,\eps} \to \tilde \gamma_0 
                 \quad  \text{in $L^2_\bx$, a.s.,
                 and in $L^2_{\tilde \omega}L^2_\bx$},
	\end{split}
\end{equation}
where the limit functions are confined to 
the spaces hinted at by the uniform estimates found in 
Lemma \ref{lem:apriori-est} and Corollary \ref{cor:Lq0-est}.  
Furthermore, $\tilde v = \tilde v_i- \tilde v_e$ 
a.e.~in $\tilde \Omega\times [0,T]\times \cO$. 
Considering \eqref{Gal:est1-mechanic-u-p} and the
equality of laws, we may assume that 
\begin{equation}\label{eq:tilde-weakconv2-u-eps}
	\eps \tilde \bu_\eps\weakeps 0 
	\quad  \text{in $L^2_{t,\bx}$, a.s.,
	and in $L^2_{\tilde\omega}L^2_{t,\bx}$}.
\end{equation}

Without loss of generality, we may assume 
that the probability space remains the same as 
previously defined, that is, $\bigl(\tilde{\Omega},\tilde{\cF}, \tilde{P}\bigr)$. 
Indeed, as stated in \cite[Theorem 2]{Jakubowski:1997aa}, it is 
possible to select the underlying probability space as 
$\bigl ([0,1],\mathcal{B}([0,1]),\operatorname{Leb}\bigr)$, 
where $\mathcal{B}([0,1])$ refers to the Borel 
subsets of $[0,1]$, and $\operatorname{Leb}$ 
represents the Lebesgue measure.

A new stochastic basis is required for the limiting functions:
\begin{equation*}
	\tilde \cS=\left(\tilde \Omega,\tilde \cF, 
	\Setb{\tilde \cF_t}_t,\tilde P,
	\tilde W^v,\tilde W^w\right),
\end{equation*}
where $\Setb{\tilde \cF_t}_t$ is the 
smallest filtration that makes all the 
relevant processes adapted. 
The limits $\tilde W^v$ and $\tilde W^w$ 
are (independent) cylindrical Wiener processes 
on $\tilde \cS$.

The laws of $\tilde v(0)=\tilde v_0$, 
$\tilde w(0)=\tilde w_0$, and 
$\tilde \gamma(0)=\tilde \gamma_0$
are $\mu_{v_{0}}$, $\mu_{w_{0}}$, and 
$\mu_{\gamma_{0}}$, respectively, see 
\eqref{eq:moment-est}. Regarding the 
artificial initial data, we have 
\begin{equation}\label{eq:artificial-init-conv-eps}
	\eps \tilde p_{0,\eps},\,
	\eps \tilde v_{i,0,\eps},\,
	\eps \tilde v_{i,0,\eps} \toeps 0
	\quad  
	\text{in $L^2\bigl(\tilde \Omega;L^2(\cO)\bigr)$}.
\end{equation}

In the course of deriving tightness and a.s.~representations, 
which culminate in \eqref{eq:tilde-weakconv1-eps}, we may 
introduce a new variable, namely $\eps p_\eps$. 
According to the a priori estimate 
\eqref{Gal:est1-mechanic-u-p}, $\eps p_\eps$ can be interpreted 
as a random variable with values in $\bigl(L^2_{t,\bx})_w$. 
The tightness of $\eps p_\eps$ comes from the same estimate. 
Thus, we can identify a new variable with the same law, symbolized by 
$\eps \tilde p_\eps$, and 
a limiting variable which, in light of \eqref{Gal:est1-mechanic-u-p}, 
can be presumed to be the zero function, such that
\begin{equation}\label{eq:tilde-weakconv2-p-eps}
	\eps \tilde p_\eps \weakeps 0 
	\quad  \text{in $L^2_{t,\bx}$, a.s.,
	and in $L^2_{\tilde\omega}L^2_{t,\bx}$}.
\end{equation}

In the subsequent analysis, we aim to establish the 
existence of a certain limit pressure, denoted by $\tilde p$, such 
that $\tilde p_\eps \to \tilde p$ a.s.~in $\Dp\bigl(0,T;(L^2_{\bx})_w\bigr)$, 
see the upcoming Proposition \ref{prop:pressure-conv}. 
Assuming the validity of this assertion and the convergences 
expressed by \eqref{eq:tilde-weakconv1-eps}, \eqref{eq:artificial-init-conv-eps} 
and \eqref{eq:tilde-weakconv2-u-eps}, \eqref{eq:tilde-weakconv2-p-eps}, we 
can modify the proof of Lemma \ref{eq:limit-eqs-tilde} to infer 
that $\tilde \bu,\tilde p,\tilde v_i,\tilde v_e,
\tilde v=\tilde v_i-\tilde v_e, \tilde w, \tilde \gamma$ collectively 
represent a weak martingale solution of the stochastic 
electromechanical bidomain model, thereby concluding the proof 
of Theorem \ref{thm:martingale} (see also the time-continuity 
discussion at the end of this section).

\medskip

Despite the weak convergences \eqref{eq:tilde-weakconv2-p-eps}, 
there exist no $\eps$-independent a priori 
estimates which can guarantee the weak 
convergence of the pressure $\tilde p_\eps$ 
as $\eps\to 0$. This factor constitutes 
a key distinction between the ``$n \to \infty$" limit 
and the ``$\eps \to 0$" limit. More specifically, this discrepancy 
is attributable to the influence of the artificial compressibility 
present in \eqref{eq:weakform-tilde-Stokes}, which causes 
the estimates on $p_{\eps}$ to become $\eps$-dependent. 
Therefore, to tackle the term $-\tilde p_{\eps}\Grad\cdot \bv$ in 
\eqref{eq:weakform-tilde-Stokes} when passing to the limit, we will 
adopt a roundabout strategy which capitalizes on the structure of the equation 
and incorporates a stochastic version \cite{Langa:2003aa} of de Rham's theorem.  
The aim is to prove that the pressure $\tilde p_\eps$ 
converges weakly in $\bx$, in the sense of distributions 
in $t$, and a.e.~in $\tilde \omega$. 
To facilitate this aim, we first recall a 
well-known result \cite[Theorem IV.3.1]{Boyer:2012aa} 
about the existence of a continuous right-inverse 
of the divergence operator.

\begin{thm}\label{thm:pression-rec}
Let $D$ be a bounded, connected and Lipschitz 
open subset of $\R^d$. For all functions $q\in L^2_0(D)
:=\Setb{q\in L^2(D):\int_{D}q\dx=0}$, 
there exists a vector $\bv\in [H^1_0(D)]^d$ 
such that $\nabla\cdot \bv = q$.
\end{thm}

Our analysis also requires  the following stochastic 
version of de Rham's theorem, as proved in the work 
of Langa, Real, and Simon \cite[Theorem 4.1]{Langa:2003aa}.
\begin{thm}\label{lem:Simon}
Let $D$ be a bounded, connected and Lipschitz 
open subset of $\R^d$, let $\bigl(\Omega,\cG,P\bigr)$ 
be a complete probability space and, given $r_0 \in [1, \infty]$, 
$r_1 \in [1, \infty]$ and $s_1 \in \mathbb{Z}$, let
$$
\bh \in L^{r_0} \bigl(\Omega,\mathcal{G},P;
W^{s_1,r_1}(0,T;[H^{-1}(D)]^d)\bigr)
$$
be such that, for all $\bv \in (\D(D))^d$ such 
that $\nabla \cdot \bv = 0$, $P$-a.s.,
$$
\bigl \langle \bh,\bv \bigr\rangle_{[\Dp(D)]^d \times [\D(D)]^d} = 0 
\quad \text{in $\Dp(0,T)$}.
$$
Then there exists a unique
$p \in L^{r_0} \bigl(\Omega,\mathcal{G},P;
W^{s_1,r_1}(0,T; L^2(D))\bigr)$
such that, $P$-a.s., 
$$
\nabla p = \bh 
\,\,\, \text{in $[\Dp((0, T) \times D)]^d$}, 
\qquad 
\int_D p \dx = 0 
\,\,\, \text{in $\Dp(0, T)$}.
$$
Moreover, there exists a positive number $c(D)$, 
independent of $\bh$, such that, $P$-a.s.,
$$
\norm{p}_{W^{s_1,r_1}(0,T;L^2(D))} 
\leq c(D) \norm{\bh}_{W^{s_1,r_1}(0,T;[H^{-1}(D)]^d)}.
$$
\end{thm}

We can now demonstrate  weak convergence of the 
pressure $\tilde p_\eps$.

\begin{prop}\label{prop:pressure-conv}
There exists $\tilde p\in L^2\bigl(\tilde \Omega;
L^2((0,T)\times \cO)\bigr)$ such that
\begin{equation}\label{eq:weak-conv-peps}
	\tilde p_\eps \weakeps \tilde p \quad 
	\text{in $L^2(\cO)$, in $\Dp(0,T)$, a.s.},
\end{equation}
that is, for any $q\in L^2(\cO)$ and $\phi\in \D(0,T)$,
$$
\int_0^T \int_{\cO} 
\tilde p_\eps q\phi \dx\dt
\toeps \int_0^T \int_{\cO} \tilde p q\phi \dx\dt, 
\quad \text{almost surely}.
$$
\end{prop}

\begin{proof}
Given any test function $\bv\in  [\D(\cO)]^3$ such 
that $\nabla\cdot\bv=0$, the first equation 
in \eqref{eq:weakform-tilde-Stokes} becomes, 
for a.e.~$(\tilde \omega,t)$ (w.r.t. $\tilde P\times dt$),
\begin{equation*}
	\eps\int_{\cO}\tilde\bu_{\eps}(t)\cdot \bv \dx
	+\int_0^t\int_{\cO} \nabla\tilde\bu_{\eps}
	\bsigma(\bx,\tilde\gamma_{\eps}):\nabla \bv\dx \ds	
	=\int_0^t\int_{\cO} \ffte \cdot \bv \dx\ds.
\end{equation*}
By echoing the arguments used in the proof of 
Lemma \ref{eq:limit-eqs-tilde} and employing the convergences 
outlined in \eqref{eq:tilde-weakconv1-eps}, \eqref{eq:artificial-init-conv-eps} 
and \eqref{eq:tilde-weakconv2-u-eps}, we can 
send $\eps\to 0$ to ultimately reach the equation
\begin{equation}\label{eq:zero-div-I}
	\int_0^t\int_{\cO} \nabla\tilde\bu
	\bsigma(\bx,\tilde\gamma):\nabla \bv\dx \ds
	= \int_0^t\int_{\cO} \fft \cdot \bv \dx\ds, 
	\quad \text{a.s., $t\in [0,T]$},
\end{equation}
where $\fft=\ff(t,\bx,\tilde \gamma)$, cf.~\eqref{def:ff}.
By defining 
$$
\bh:=\Div \bigl( \nabla\tilde \bu 
\sigma(\bx,\tilde \gamma)\bigr)+\fft(t,\bx,\tilde \gamma)
\in L^2\bigl( \tilde \Omega; L^2(0,T;[H^{-1}(\cO)]^3)\bigr),
$$
and differentiating \eqref{eq:zero-div-I} with respect to $t$, 
it follows in a straightforward manner that
$$
\bigl \langle \bh, \bv \bigr\rangle_{[\Dp(\cO)]^3 \times [\D(\cO)]^3}=0 
\quad \text{in $\Dp(0,T)$, a.s.},
$$
$\forall \bv\in  [\D(\cO)]^3$ with $\nabla\cdot\bv=0$. 
Therefore, by Theorem \ref{lem:Simon} (with 
$r_0=r_1=2$ and $s_1=0$), there exists 
$\tilde P\in L^2\bigl(\tilde \Omega;L^2((0,T)\times \cO)\bigr)$ 
such that $\nabla \tilde P = \bh$ in $[\Dp((0, T) \times D)]^3$ 
with $\int_{\cO} \tilde P \dx=0$ in $\Dp(0,T)$. Given $\tilde P$, let 
us introduce the modified pressure 
\begin{equation}\label{eq:tP-def}
	\tilde p=\tilde P
	+\frac{1}{\abs{\cO}}\bar{C}(\tilde \omega,t),
\end{equation}
where the process $\bar{C}$, which is 
constant in $\bx$, will be specified later. Clearly, we have
\begin{equation}\label{eq:tP-average}
	\int_{\cO} \tilde p\dx = \bar{C} 
	\quad \text{in $\Dp(0,T)$}
\end{equation} 
and
$$
\Div \bigl( \nabla \tilde \bu 
\sigma(\bx,\tilde \gamma)\bigr)+\fft=\nabla \tilde p
\quad \text{in $[\Dp((0,T)\times \cO)]^3$, \, a.s.} 
$$

It takes a standard argument to turn the 
the first equation in \eqref{eq:weakform-tilde-Stokes}, 
which is weak in $\bx$, into the following formulation that 
is weak in $(t,\bx)$:
$$
\eps \pt \tilde \bu_\eps
-\Div \bigl(\nabla \tilde \bu_\eps 
\sigma(\bx,\tilde \gamma_\eps)\bigr)
+\Grad \tilde p_\eps-\tilde\ff_\eps=0
\quad \text{in $[\Dp((0,T)\times \cO)]^3$, \, a.s.},
$$
so that the equation
\begin{equation}\label{eq:weak-conv-peps-tmp0}
	\nabla \bigl(\tilde p_\eps-\tilde p\bigr)
	=-\eps \pt \tilde \bu_\eps
	+\Div \bigl(\nabla \tilde \bu_\eps 
	\sigma(\bx,\tilde \gamma_\eps)
	-\nabla \tilde \bu 
	\sigma(\bx,\tilde \gamma)\bigr)+(\tilde\ff_\eps-\tilde\ff)
\end{equation}
holds a.s.~in $[\Dp((0,T)\times \cO)]^3$.  
Consider a test function $\bv\in \D(\cO)$. 
Then \eqref{eq:weak-conv-peps-tmp0} implies
\begin{align*}
	&\int_{\cO}\bigl(\tilde p_\eps-\tilde p\bigr)
	\nabla\cdot \bv\dx
	=-\int_{\cO} \eps \partial_t \tilde \bu_\eps\cdot \bv\dx
	\\ & \qquad\quad
	+\int_{\cO} 
	\bigl(\nabla \tilde \bu_\eps 
	\sigma(\bx,\tilde \gamma_\eps)
	-\nabla \tilde \bu \sigma(\bx,\tilde \gamma)\bigr)
	: \nabla \bv \dx
	\\ & \qquad\quad
	+\int_{\cO} \bigl(\tilde\ff_\eps
	-\tilde\ff \bigr)\cdot\bv \dx 
	\quad \text{in $\Dp(0,T)$}.
\end{align*}
From \eqref{eq:tilde-weakconv2-u-eps}, as $\eps \to 0$, the first 
term on the right-hand side converges 
to zero in $\Dp(0,T)$. Following a similar reasoning 
as in the proof of 
Lemma \ref{eq:limit-eqs-tilde}, we can 
deduce that 
$$
\nabla \tilde \bu_\eps 
\sigma(\bx,\tilde \gamma_\eps)
-\nabla \tilde \bu \sigma(\bx,\tilde \gamma)\weakeps 0, 
\quad 
\tilde\ff_\eps-\tilde\ff\weakeps 0
\qquad \text{in $L^2_{t,\bx}$, a.s.}
$$
As a result, the second and third terms also tend to zero 
in $\Dp(0,T)$, so that $\forall \bv\in \cD(\cO)$,
\begin{equation}\label{eq:weak-conv-peps-tmp1}
	\lim_{\eps\to 0}\int_{\cO}
	\bigl( \tilde p_\eps-\tilde p\bigr) \nabla\cdot \bv \dx=0 
	\quad \text{in $\Dp(0,T)$, a.s.}
\end{equation}
Through a density argument, this 
statement holds true for all $\bv \in [H^1_0(\cO)]^3$.

To complete the proof, we must show that the statement 
is true for all $\bv\in [H^1(\cO)]^3$. For this sake,
consider an arbitrary $q\in L^2(\cO)$ and set $\bar{q}=q-C_q$, 
where $C_q:=\frac{1}{\abs{\cO}}\int_{\cO} q\dx$; thus, $\bar{q}$ 
belongs to the space $L^2_0(\cO)$. 
By Theorem \ref{thm:pression-rec}, there exists 
a vector $\bar \bv\in [H^1_0(\cO)]^3$ such 
that $\nabla\cdot \bar \bv=\bar q$. 
Subsequently, it can be inferred from 
\eqref{eq:weak-conv-peps-tmp1} that
\begin{equation}\label{eq:t-peps-conv-tmp1}
	\lim_{\eps\to 0} \int_{\cO}
	\bigl( \tilde p_\eps-\tilde p\bigr) q \dx
	=C_q\lim_{\eps\to 0} R_\eps 
	\quad \text{in $\Dp(0,T)$, a.s.},
\end{equation}
where (no spatial test function)
\begin{equation}\label{eq:Reps-def}
	R_\eps=\int_{\cO}
	\bigl( \tilde p_\eps-\tilde p\bigr)\dx
	=\int_{\cO} \tilde p_\eps\dx-\bar{C},
\end{equation}
recalling \eqref{eq:tP-average}. 
To complete the proof of the proposition, we only need 
to demonstrate that
\begin{equation}\label{eq:Reps-limit}
	\lim_{\eps \to 0} R_{\eps} = 0
	\quad \text{a.s.~in $\Dp(0,T)$.}
\end{equation}

Considering the first equation in \eqref{eq:weakform-tilde-Stokes}, we 
differentiate it weakly with respect to $t$. 
As $\eps\to 0$ and by adapting the arguments used in the proof 
of Lemma \ref{eq:limit-eqs-tilde}, combined with the convergences 
described in \eqref{eq:tilde-weakconv1-eps}, \eqref{eq:artificial-init-conv-eps}, 
and \eqref{eq:tilde-weakconv2-u-eps}, we deduce that
\begin{equation}\label{eq:weak-conv-peps-temp2}
	\begin{split}
		&\lim_{\eps \to 0} \int_{\cO}\tilde p_{\eps} 
		\Grad\cdot \bv\dx
		=\int_{\cO} \nabla\tilde\bu
		\bsigma(\bx,\tilde\gamma):\nabla \bv \dx
		\\ & \qquad
		+\int_{\partial \cO}
		\alpha\tilde\bu\cdot \bv \, dS(\bx)
		-\int_{\cO} \tilde\ff \cdot \bv\dx 
		\quad \text{in $\Dp(0,T)$, a.s.}, 		
	\end{split}
\end{equation}
for all $\bv\in [H^1(\cO)]^3$. Let us specify the test function $\bv$ 
in \eqref{eq:weak-conv-peps-temp2} more precisely as 
$\mathbf{v}=(x_1,0,0)$, which belongs to the space $[H^1(\cO)]^3$ 
and satisfies the condition $\nabla \cdot \bv=1$. 
By substituting this vector $\bv$ into 
\eqref{eq:weak-conv-peps-temp2}, we obtain
\begin{equation}\label{eq:weak-conv-peps-temp3}
	\begin{split}
		\lim_{\eps \to 0} \int_{\cO} \tilde p_{\eps}\dx
		& = \int_{\cO} \nabla\tilde\bu
		\bsigma(\bx,\tilde\gamma):\nabla \bv\dx 
		\\ & \quad
		+\int_{\partial \cO}\alpha\tilde\bu\cdot \bv \, dS(\bx)
		-\int_{\cO} \tilde\ff \cdot \bv\dx
		\quad \text{in $\Dp(0,T)$, a.s.}
	\end{split}
\end{equation}
Observe that the right-hand side represents 
a real-valued process, depending only on 
$(\tilde \omega,t)$. Let us identify the $\bar{C}$ in \eqref{eq:tP-def} 
and \eqref{eq:Reps-def} with the right-hand side 
of \eqref{eq:weak-conv-peps-temp3}. With this specification, 
\eqref{eq:Reps-limit} naturally ensues.

In summary, combining \eqref{eq:t-peps-conv-tmp1} 
with \eqref{eq:Reps-limit}, we get
$$
\lim_{\eps\to 0} \int_{\cO}
\bigl( \tilde p_\eps-\tilde p\bigr) q \dx
=0 \quad \text{in $\Dp(0,T)$, a.s.}
$$
Since the choice of $q\in L^2(\cO)$ was arbitrary, we can conclude 
that \eqref{eq:weak-conv-peps} holds true.
\end{proof}

The proof of Theorem \ref{thm:martingale} is nearly 
complete, pending the time-continuity requirements 
of $\tilde v, \tilde w, \tilde \gamma$. 
Revisiting the discussion at the end of 
Section \ref{sec:conv-nondegen}, we may assume that 
$\tilde v$, $\tilde w$, and $\tilde \gamma$ possess 
versions that belong a.s.~to $C([0,T];(H^1(\cO))^*)$. 
Specifically, this implies that $\tilde v$ and $\tilde w$ 
are predictable in $(H^1(\cO))^*$. This establishes 
part \eqref{eq:mart-adapt} of Definition \ref{def:martingale-sol}. 

It is possible to establish even stronger assertions 
regarding the continuity in time for 
$\tilde v$, $\tilde w$ and $\tilde \gamma$. 
Indeed, define $X$ and $Y$ by 
$$
X(t):= \tilde v(t)-\int_0^t \beta_v(\tilde v) \,d \tilde W^v(s), 
\quad 
Y(t) := \tilde w(t) - \int_0^t \beta_w(\tilde w)\, d \tilde W^w(s).
$$ 
Note that a.s., 
$$
X, Y \in L^2(0,T;H^1(\cO)), 
\quad 
\partial_t X, \partial_t Y \in L^2(0,T;(H^1(\cO))^*),
$$
where the second part comes from the equations satisfied by 
$X,Y$, see \eqref{eq:weakform}.  
As a result, according to \cite{Temam:1977aa}, 
$X$ and $Y$ are a.s.~in $C([0, T]; L^2(\cO))$. 
By standard theory, we know that 
$t \mapsto \int_0^t \beta(\tilde v) \dW$ is a.s.~in 
$C([0, T]; L^2(\Omega))$, where $(\beta, W)$ corresponds to 
$(\beta_v, \tilde W^v)$ and $(\beta_w, \tilde W^w)$. 
Consequently, we can deduce that $\tilde v$ and $\tilde w$ 
are a.s.~in $C([0, T]; L^2(\Omega))$. Similarly, it can be 
shown (in a simpler manner) that $\tilde \gamma$ is also a.s.~in 
$C([0, T]; L^2(\Omega))$.

\section{Numerical examples}\label{sec:numeric}
In this section, we numerically solve the proposed 
stochastic electromechanical model \eqref{coupbid}. 
The examples presented are primarily for academic purposes and 
incorporate only additive noise. More comprehensive 3D simulations 
with nonlinear noise functions are slated for future work. 

We consider a simple square domain, $\cO= (0,1) \times (0,1)$, 
employing a finite element discretization to address both 
the stochastic problem \eqref{coupbid} and its deterministic counterpart. 
For the electrical part of the model, we utilize a semi-implicit 
time-stepping scheme. Nonlinear terms are solved explicitly, while linear 
differential operators are handled implicitly. Regarding the mechanical 
components of the system, Newton's method is employed to 
solve the nonlinear equations.

The domain $\cO$  is discretized using an irregular mesh consisting 
of $952$ triangles and $517$ vertices. The finite element method 
is implemented in FreeFem++, utilizing $P_2$ elements for 
the following variables: the displacement $\bu$, 
the transmembrane potential $v$, 
the extracellular potential $v_e$, the gating variable $w$, and the 
activation variable $\gamma$. 
For the pressure $p$, $P_1$ elements are employed. 

The noise added to the simulations takes the form 
$\beta_v\, dW^v$ and $\beta_w\, dW^w$, where $W^v$ and $W^w$ 
are real-valued Brownian motions. The noise coefficients 
$\beta_v$ and $\beta_w$ are both set to a constant 
value $\beta$ equal to $0.5$ or $1.0$. An explicit discretization 
is applied to the noise terms, adding $\beta \sqrt{\Delta t} N(0,1)$ to 
the relevant equations at each time level, 
where $N(0,1)$ is the normal distribution with zero mean and unit variance, 
and $\Delta t$ is the time step.
 
We fix the time step at $\Delta t=0.0125$ and set the model 
parameters $\chi=1$ and $c_m=1$. The material (heart) is characterized as 
Neo-Hookean with an elastic modulus of $\mu=4$. The conductivity 
tensors are taken as follows:
$$
K_i=
\begin{pmatrix}
2 & 0
\\
0 & 1
\end{pmatrix}
\times10^{-2},
\qquad 
K_e= \begin{pmatrix}
4 & 0
\\
0 & 2
\end{pmatrix}
\times10^{-2}.
$$
The functions $\Ion$ and $H$ 
are given by the FitzHugh-Nagumo model:
$$
\Ion(v,w)=k\bigl(w+v(v-a)(v-1)\bigr),
\qquad
H(v,w)=d_1v-d_2w,
$$
where we specify $k=-80$, $a=0.25$, $d_1=0.17$, and $d_2=1$. 
An initial stimulus, shaped as a half-circle and centered 
at the point $(x,y)=(0,0.5)$, is applied at 
time $t = 0$ using the function
$$
v_0(x,y)=1
-\dfrac{1}{1+e^{-50\left(\sqrt{x^2+(y-0.5)^2}-0.18\right)}}.
$$
This stimulus is applied for a short duration of $0.01$ time units using
$$
\Iapp =
\begin{cases}
	v_0(x,y), & \text{for } 0 \leq t < 0.01,
	\\
	0, & \text{for } t \geq 0.01.
\end{cases}
$$
Furthermore, the initial values for the extracellular potential 
$v_e$ and the gating variable $w$ are both set to $0$.
Finally, the activation variable $\gamma$ is initialized 
according to the expression $-0.3 v_0/ (2-v_0)$.

The stimulus propagates throughout the entire domain, 
causing visible deformations. Figure \ref{SimFigures} illustrates 
the wavefront propagation as observed in both the deterministic and 
stochastic models. In the deterministic solutions, the wavefront appears 
smoother, whereas in the stochastic case, it is characterized 
by noticeable fluctuations. To analyze changes in the 
transmembrane potential $v$ at fixed points in the domain, measurements 
were taken at three locations: $(0,0.5)$ (the initiation point of the stimulus), 
$(0.5,0.5)$, and $(1,0.5)$. Figure \ref{APcompare} demonstrates 
the impact of noise on the transmembrane potential at these points, 
highlighting the contrasts with the more regular behavior 
of the deterministic model. 

\begin{figure}[ht]
  \centering
\includegraphics[width=4cm]{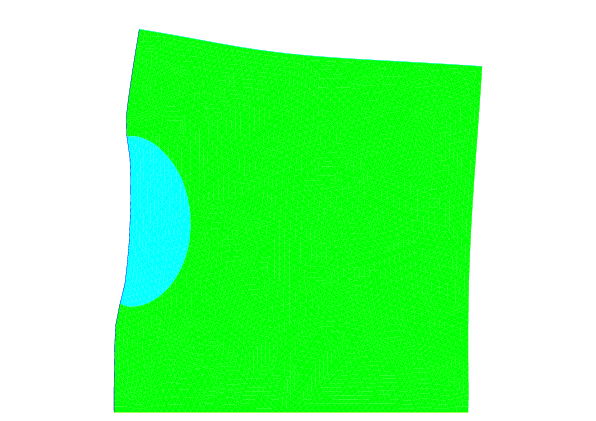}\hspace*{-1.3cm}  
\includegraphics[width=4cm]{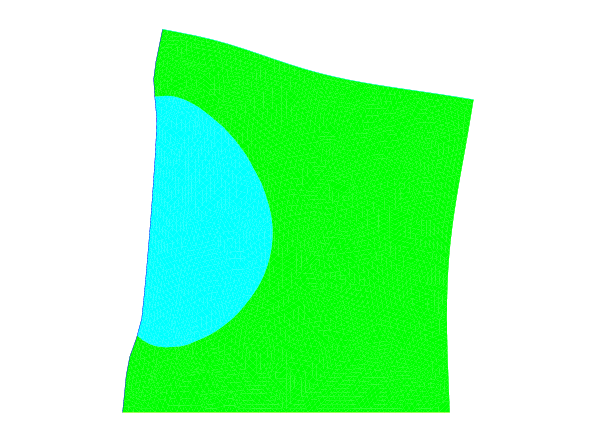}\hspace*{-1.3cm}
\includegraphics[width=4cm]{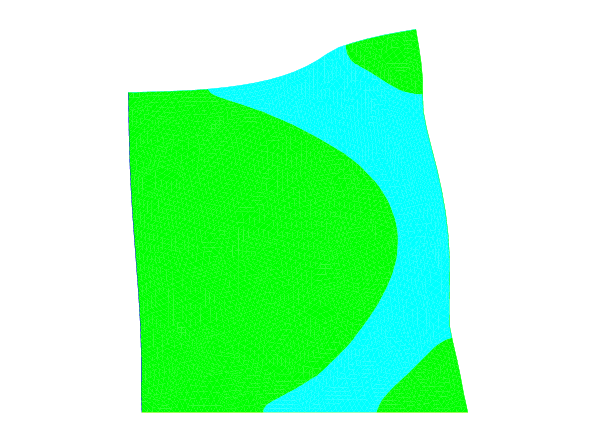} \hspace*{-1.3cm} \includegraphics[width=4cm]{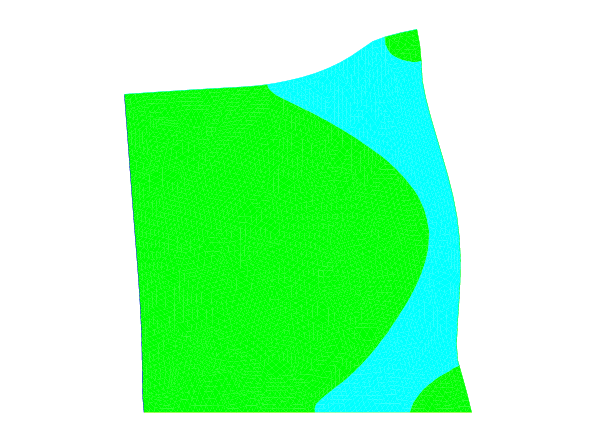}\\
\includegraphics[width=4cm]{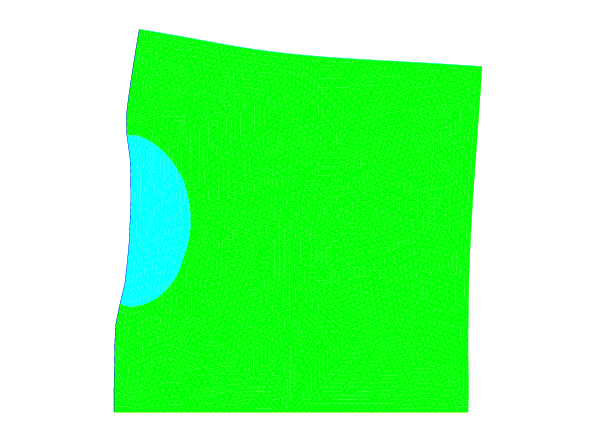}\hspace*{-1.3cm} \includegraphics[width=4cm]{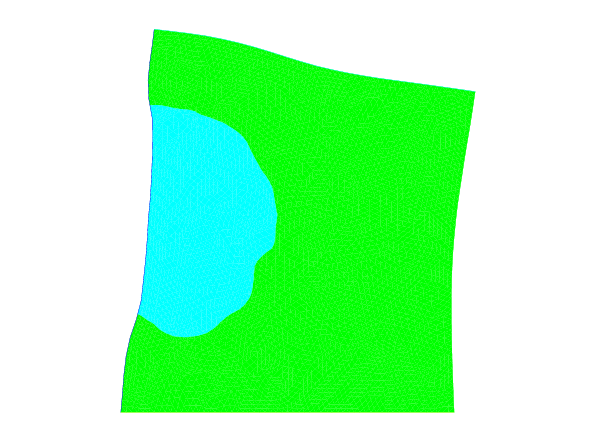}\hspace*{-1.3cm}  \includegraphics[width=4cm]{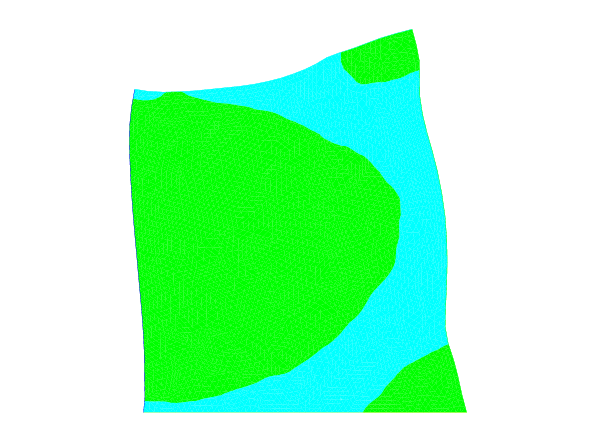}\hspace*{-1.3cm}  \includegraphics[width=4cm]{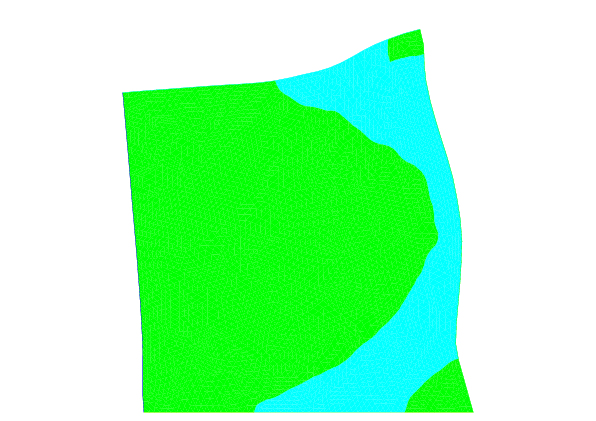}\\
\includegraphics[width=4cm]{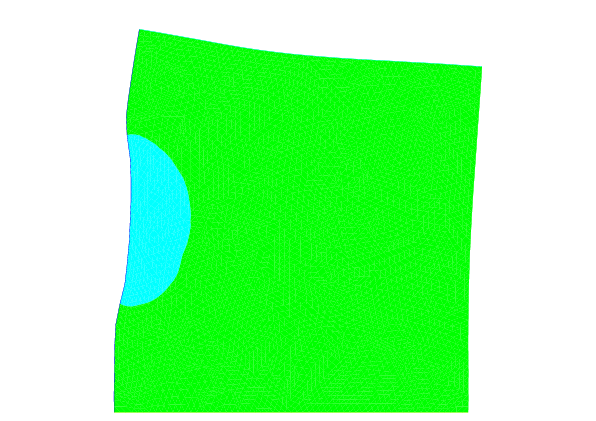}\hspace*{-1.3cm} \includegraphics[width=4cm]{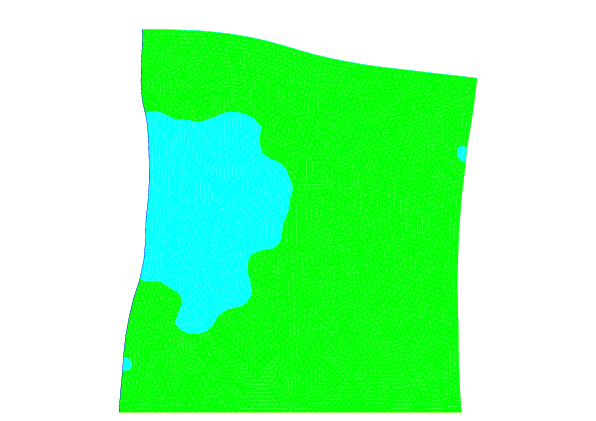}\hspace*{-1.3cm}  \includegraphics[width=4cm]{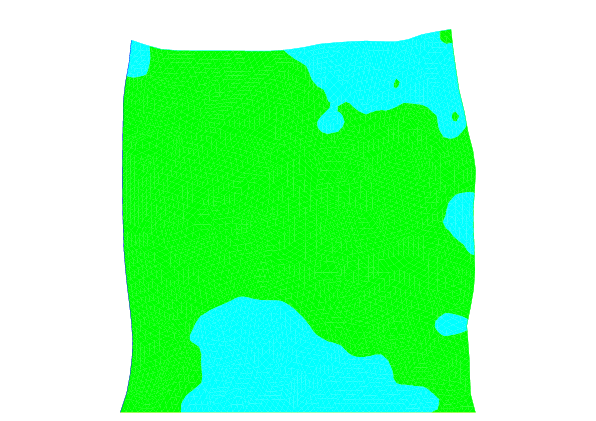}\hspace*{-1.3cm}  \includegraphics[width=4cm]{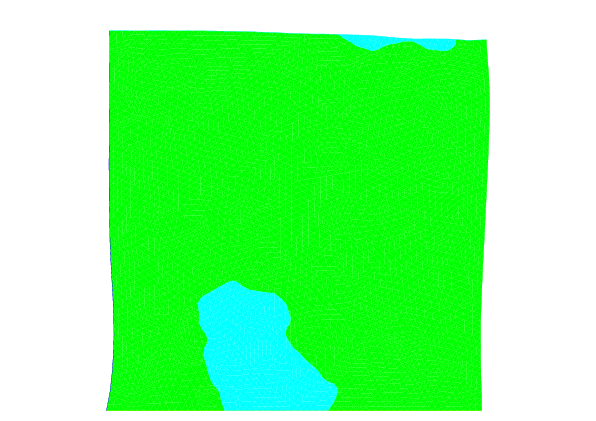}
\caption{The transmembrane potential $v$ at time iterations 
$7$, $70$, $230$ and $250$, from left to right. 
Top row: Deterministic Model. 
Middle row: Stochastic Model with $\beta_v=\beta_w=0.5$. 
Bottom row: Stochastic Model with $\beta_v=\beta_w=1$.
The green region indicates low values of $v$ (at rest), while the blue region 
represents high values of $v$ (activated).}
\label{SimFigures}
\end{figure}

\begin{figure}[ht]
  \centering
\includegraphics[width=10cm]{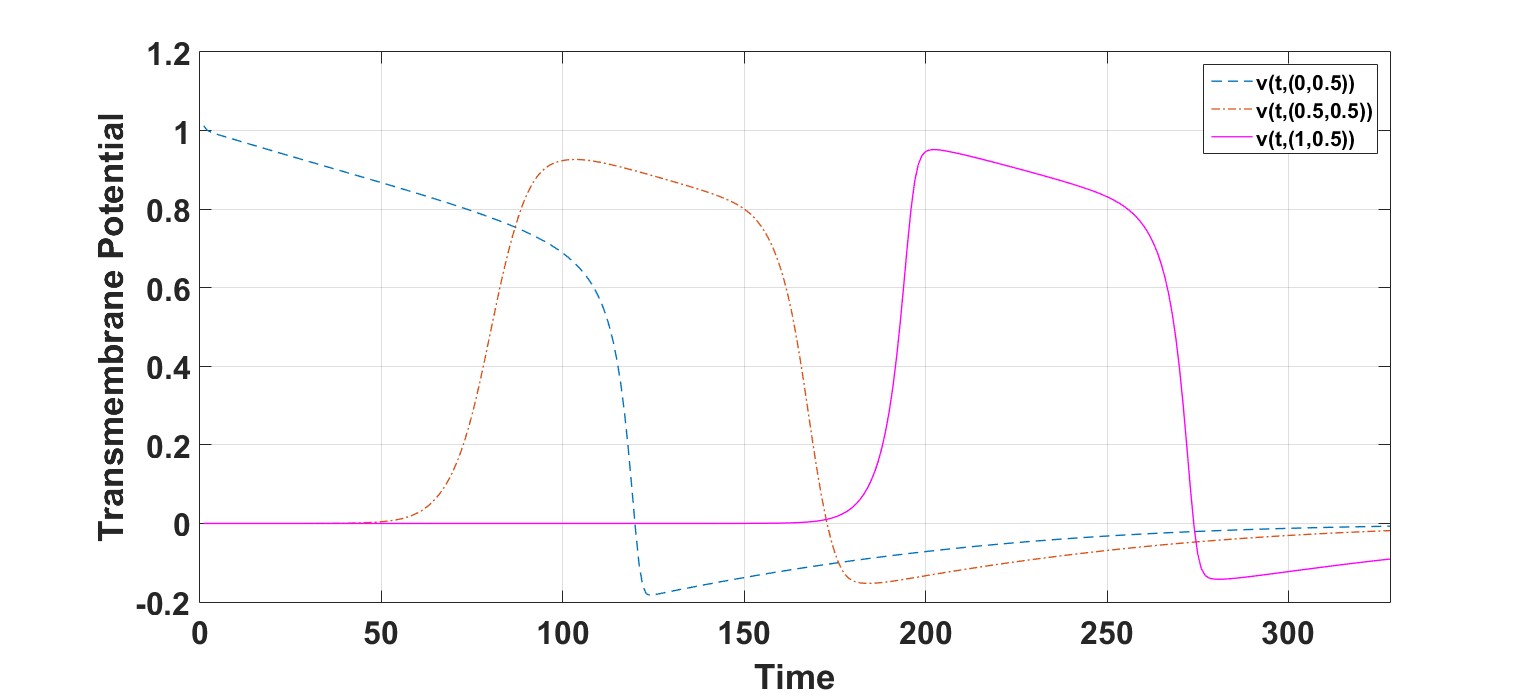}\\ 
\includegraphics[width=10cm]{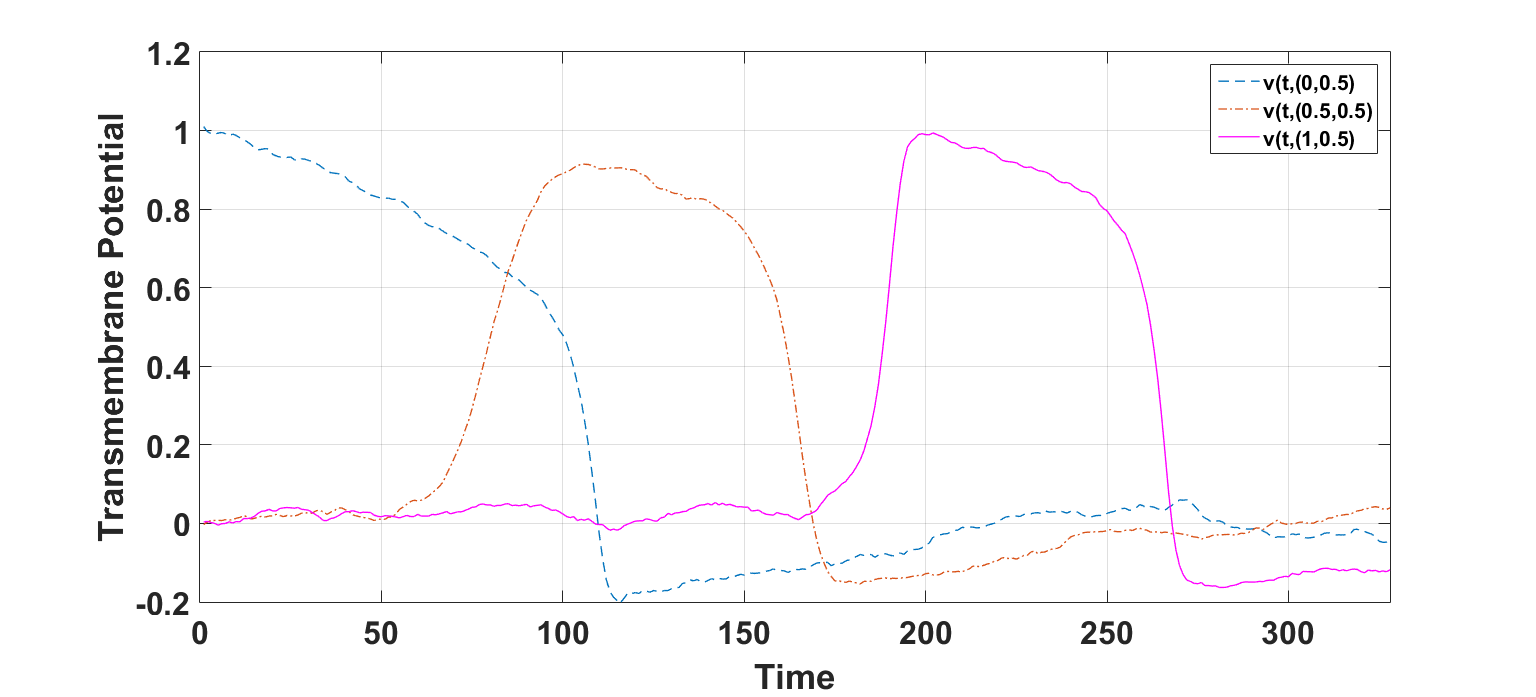}\\
\includegraphics[width=10cm]{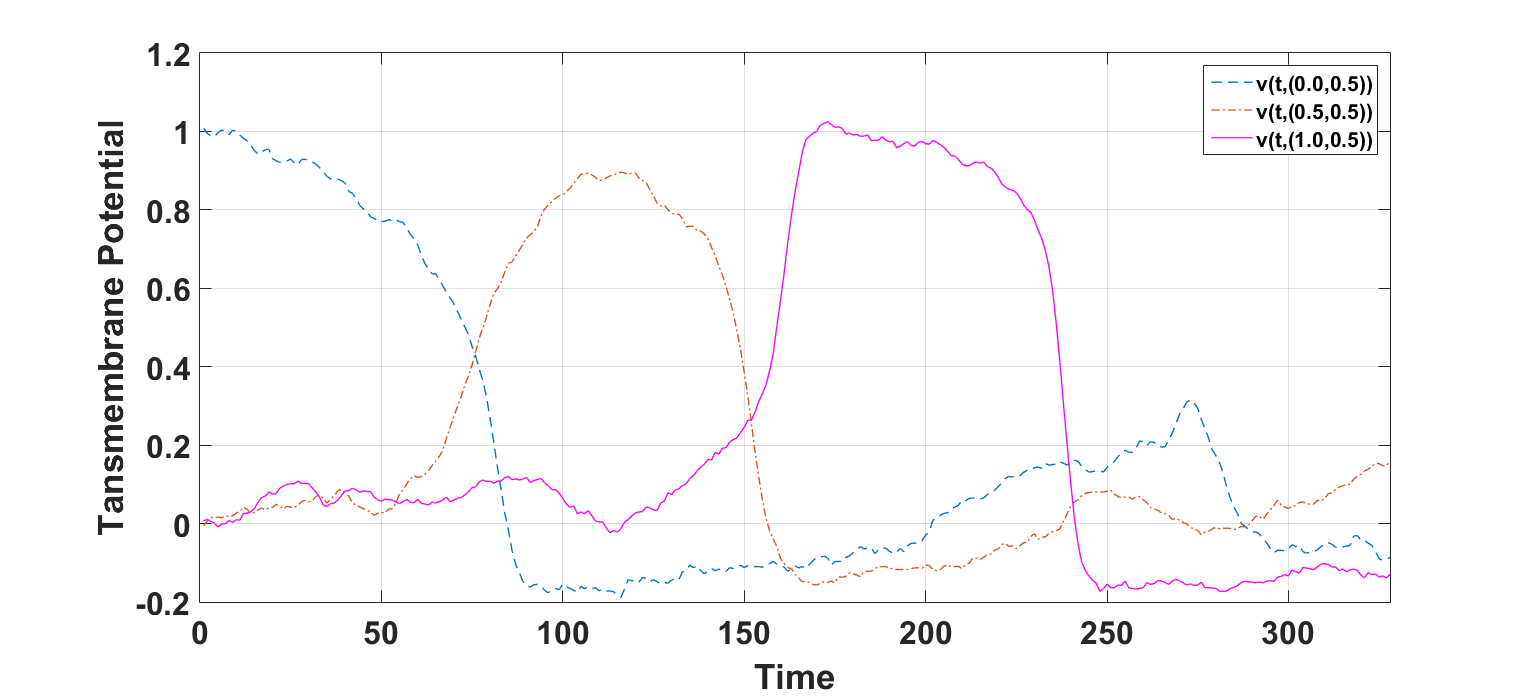}
\caption{Transmembrane potential recorded at 3 different 
points in the domain: $(0,0.5)$, $(0.5,0.5)$ and $(1,0.5)$. 
Top: Deterministic Model. Middle: Stochastic Model with 
$\beta_v=\beta_w=0.5$. Bottom: Stochastic Model 
with $\beta_v=\beta_w=1$.}\label{APcompare}
\end{figure}

\end{document}